\documentclass[11pt,a4paper,reqno]{amsart}%article}%
\usepackage{amsfonts}
\usepackage{amsmath,amssymb,amsthm,amsxtra}
\usepackage{float}
\usepackage[colorlinks, linkcolor=blue!72,anchorcolor=orange,
    citecolor=red,urlcolor=Emerald, bookmarksopen,bookmarksdepth=2]{hyperref}
\usepackage[usenames,dvipsnames]{xcolor}
\usepackage{enumitem}
\usepackage{geometry,array}
\usepackage{graphicx}
\usepackage{subfigure}
\usepackage{bookmark}
\usepackage{tikz}
\usepackage{url}

\usetikzlibrary{matrix,positioning,decorations.markings,arrows,decorations.pathmorphing,
    backgrounds,fit,positioning,shapes.symbols,chains,shadings,fadings,calc}
\tikzset{->-/.style={decoration={  markings,  mark=at position #1 with
    {\arrow{>}}},postaction={decorate}}}
\tikzset{-<-/.style={decoration={  markings,  mark=at position #1 with
    {\arrow{<}}},postaction={decorate}}}

\usepackage[all]{xy}
\usepackage{chngcntr}
\counterwithin{figure}{section}
\theoremstyle{plain}
\newtheorem{theorem}{Theorem}[section]

\newtheorem{lemma}[theorem]{Lemma}
\newtheorem{corollary}[theorem]{Corollary}
\newtheorem{proposition}[theorem]{Proposition}
\newtheorem{conjecture}[theorem]{Conjecture}
\theoremstyle{definition}
\newtheorem{definition}[theorem]{Definition}

\newtheorem{example}[theorem]{Example}

\newtheorem{remark}[theorem]{Remark}

\numberwithin{equation}{section}
\newtheorem{assumption}[theorem]{Assumption}
\newtheorem{construction}[theorem]{Construction}
\newtheorem{problem}[theorem]{Problem}
\newtheorem*{con}{Convention}

\newtheorem*{thma}{Theorem~1}

\def\hh{\mathcal}
\def\ha{\hh{A}}
\def\kong{\mathbb}
\def\<{\langle}
\def\>{\rangle}
\def\ZZ{\mathbb{Z}}
\def\QQ{\mathbb{Q}}
\def\R{\mathbb{R}}
\def\RR{\R}
\def\CC{\mathbb{C}}
\def\TM{\mathcal{T}_M}
\def\O{\mathcal{O}}
\def\rd{\partial}
\def\hs{\mathfrak{h}}
\def\reg{\mathrm{reg}}
\def\alp{\alpha}
\def\lam{\lambda}
\def\CKZ{\mathrm{CKZ}}

\def\Po{\Pol(\xi)}
\def\ii{\mathbf{i}}

\def\XX{\mathbb{X}}
\def\Stab{\operatorname{Stab}}

\def\Stap{\operatorname{Stab}^\circ}
\def\OStab{\QStab^{\oplus}}
\def\CStab{\QStab^{*}}
\def\QQuad{\operatorname{QQuad}}
\def\QStab{\operatorname{QStab}}
\def\QStap{\QStab^\circ}

\def\Stap{\operatorname{Stab}^\circ}

\renewcommand{\mod}{\operatorname{mod}}
\newcommand{\Ho}[1]{\operatorname{\bf H}_{#1}}
\newcommand{\HHo}[1]{\widehat{\operatorname{\bf H}}_{#1}}

\newcommand{\Cone}{\operatorname{Cone}}
\newcommand{\Ae}[1]{\mathcal{#1}^e}
\renewcommand{\Re}{\operatorname{Re}}
\newcommand{\id}{\operatorname{id}}
\newcommand{\HH}{\operatorname{HH}}

\newcommand{\D}{\operatorname{\hh{D}}}
\newcommand{\C}{\operatorname{\hh{C}}}
\newcommand{\per}{\operatorname{per}}

\newcommand\Sph{\operatorname{Sph}}%^\circ}

\newcommand{\Tri}{\Delta}
\newcommand{\ST}{\operatorname{ST}}        %spherical twists
\newcommand{\BT}{\operatorname{BT}}        %braid twists
\newcommand{\MCG}{\operatorname{MCG}}
\newcommand{\Int}{\operatorname{Int}}

\def\Homeo{\operatorname{Homeo}}
\newcommand\coho[1]{\operatorname{H}^{#1}}
\newcommand\ho[1]{\operatorname{H}_{#1}}
\newcommand\coch[1]{\operatorname{Z}^{#1}}

\def\Diff{\operatorname{Diff}}

\def\RHom{\operatorname{RHom}}
\newcommand{\EG}{\operatorname{EG}}
\def\add{\operatorname{add}}
\newcommand{\Quad}{\operatorname{Quad}}

\newcommand{\conf}{\operatorname{conf}}

\newcommand{\TFuk}{\operatorname{TFuk}}

\def\Obj{\operatorname{Obj}}

\def\Aut{\operatorname{Aut}}
\def\Ind{\operatorname{Ind}}
\def\Sim{\operatorname{Sim}}
\def\Hom{\operatorname{Hom}}
\def\hom{{\hh{H}}om}

\def\Ext{\operatorname{Ext}}
\def\Irr{\operatorname{Irr}}
\def\diff{\operatorname{d}}
\def\Br{\operatorname{Br}}

\def\GL{\operatorname{GL}}
\def\rank{\operatorname{rank}}

\def\deg{\operatorname{deg}}
\newcommand{\h}{\hh{H}}            %heart

\newcommand{\ns}{\widehat{\sigma}}
\newcommand{\nz}{\widehat{Z}}
\newcommand{\np}{\widehat{\hh{P}}}
\renewcommand{\k}{\mathbf{k}}
\def\surf{\mathbf{S}}                       %FST's surface

\def\TT{\mathbf{A}}

\def\M{\mathbf{M}}

\def\surfo{{\mathbf{S}}_\Tri}
\def\LS{\log_{\cut}\surfo^\lambda}

\def\CA{\operatorname{CA}}

\def\Grot{\operatorname{\mathrm{K}}}

\def\sadd{\sigma_{\oplus}}
\def\sext{\sigma_{*}}
\def\padd{\hh{P}_{\oplus}}
\def\pext{\hh{P}_{*}}

\def\bi{\mathbf{i}}

\def\sli{\mathcal{P}}
\newcommand{\norm}[1]{\lVert #1 \rVert}
\newcommand{\ind}{\operatorname{ind}}

\def\GAX{\Gamma^{\XX}_{\TT}}
\def\CXT{\hh{C}(\GAX)}
\def\GAN{\Gamma^{N}_{\TT}}
\def\Homeo{\operatorname{Homeo}}

\def\PP{\mathbb{P}}

\def\Diff{\operatorname{Diff}}

\def\M{\mathbf{M}}
\def\P{\mathbb{P}}

\def\RHom{\operatorname{RHom}}

\def\surfo{{\mathbf{S}}_\Tri}
\def\CA{\operatorname{CA}}
\def\OA{\operatorname{OA}}
\def\wCA{\widetilde{\CA}}
\def\wOA{\widetilde{\OA}}
\def\Y{\mathbf{Y}}
\def\add{\operatorname{add}}

\def\DX{\D_{\XX}(\TT)}
\def\DN{\D_{N}(\TT)}
\def\DI{\D_{\infty}(\TT)}
\def\Dinf{\D_{fd}(\ha_\TT)}

\newcommand{\HS}{\operatorname{HS}}

\def\Grot{\operatorname{\mathrm{K}}}
\def\gldim{\operatorname{gldim}}
\def\QQuac{\operatorname{QQuad}^{\mathrm{c}}}
\def\FQuad{\operatorname{FQuad}}
\def\Cut{\operatorname{Cut}}
\def\cut{\operatorname{\mathbf{c}}}
\def\uc{\cut}
\def\ut{\mathbf{t}}

\def\sli{\mathcal{P}}
\def\bC{\mathbb{C}}

\def\Im{\operatorname{Im}}

\def\Zer{\operatorname{Zero}}
\def\Pol{\operatorname{Pol}}
\def\Crit{\operatorname{Crit}}

\def\sadd{\sigma_{\oplus}}
\def\sext{\sigma_{*}}
\def\padd{\hh{P}_{\oplus}}
\def\pext{\hh{P}_{*}}
\newcommand{\SBr}{\operatorname{SBr}}

\newcommand{\DT}{\D_\infty(\TT)}

\def\DXQ{\D_\XX(Q)}

\def\bb{\psi}
\def\ss{{\mathrm{ss}}}

\def\separated{cyan}
\def\grad{\lambda}

\def\DEG{\widetilde{\deg}\,}

\newcommand{\sslash}{\mathbin{/\mkern-6mu/}}

\def\rs{\mathrm{S}}
\def\rsf{\rs^f}
\def\rsp{\rs^{\circ}}
\def\uk{\mathbf{k}}
\def\ul{\mathbf{l}}
\def\Sp{{\operatorname{sp}}}
\def\sg{{\operatorname{sg}}}
\def\Ram{{\operatorname{\mathfrak{R}}}}
\def\surp{\rs^\circ}
\def\rsp{\mathrm{S}^\phi}
\def\Core{\operatorname{Core}}
\def\core{\underline{\Core}}
\def\cz{\Core^0}
\def\cx{\Core^0_{\cut}}
\def\logrs{\log\rs_{\Tri}}
\def\logrsx{\log\rs_{\Tri}^{\;\;\xi}}
\newcommand{\wind}[1]{\operatorname{wind}_{#1}}
\newcommand{\AC}{\operatorname{Ang}}
\newcommand{\LP}{\operatorname{LP}}

\newcommand{\ord}{\operatorname{ord}}
\def\num{\mathfrak{N}}
\def\gms{\surf^\grad}

\def\gzero{\mathbf{C}}
\def\Poly{\hh{M}}
\def\Pone{\overline{\bC}}
\def\dd{\mathbf{D}}
\def\ddo{\mathbf{D}_\Tri}

\def\wg{\widetilde{\gamma}}
\def\we{\widetilde{\eta}}

\def\wX{\widetilde{X}}
\def\Xinf{\wX_{\infty}}
\def\Xxx{\wX_{\XX}}
\def\xs{X_{3}}
\def\ACC{\operatorname{AC}}
\def\wACC{\widetilde{\ACC}}
\def\imc{\mathbf{im}_{\cut}}
\title[$q$-Stability conditions via $q$-quadratic differentials]
{$q$-Stability conditions via $q$-quadratic differentials for
Calabi-Yau-$\XX$ categories}

\author{Akishi Ikeda}
\address{AI: Department of Mathematics,
    Josai University,
    Saitama, Japan}
\email{akishi@josai.ac.jp}

\author{Yu Qiu}
\address{Qy:
	Yau Mathematical Sciences Center and Department of Mathematical Sciences,
	Tsinghua University,
    100084 Beijing,
    China.
    \&
    Beijing Institute of Mathematical Sciences and Applications, Yanqi Lake, Beijing, China}
\email{yu.qiu@bath.edu}

\begin{document}
%=========================================================
%=========================================================

%=========================================================
\begin{abstract}
Categorically, we introduce the Calabi-Yau-$\mathbb{X}$ categories $\mathcal{D}_{\mathbb{X}}$
of a graded marked surface $\mathbf{S}^\lambda$, as a $q$-deformation of
Haiden-Katzarkov-Kontsevich's topological Fukaya category $\mathcal{D}_\infty$ of $\mathbf{S}^\lambda$.
We show that $\mathcal{D}_\infty$ can be identified with
the cluster-$\mathbb{X}$ category associated to $\mathcal{D}_{\mathbb{X}}$.

Geometrically, we construct and identify the space of $q$-quadratic differentials
on the logarithm surface $\operatorname{log}_{\mathbf{c}}\mathbf{S}_\Delta^\lambda$
with the space of induced $q$-stability conditions on $\mathcal{D}_{\mathbb{X}}$,
for a complex parameter $s$ satisfying $\operatorname{Re}(s)\gg1$.
When $s=N$ is an integer,
the result gives an $N$-analogue of Bridgeland-Smith's result for realizing stability conditions
on the Calabi-Yau-$N$ orbit category $\mathcal{D}_{\mathbb{X}}\mathbin{/\mkern-6mu/}[\mathbb{X}-N]$
via CY-$N$ type quadratic differentials.

When the genus of $\mathbf{S}$ is zero, the spaces of $q$-quadratic differentials can be also identified with framed Hurwitz spaces.
As a byproduct, we confirms the conjectural almost Frobenius structure on the spaces of $q$-stability conditions for type $A$.

    \vskip .3cm
    {\parindent =0pt
    \it Key words:}
    $q$-stability conditions, Calabi-Yau-$\XX$ categories,  cluster categories,
    $q$-quadratic differentials, twisted periods

\end{abstract}
\maketitle
\tableofcontents\addtocontents{toc}{\setcounter{tocdepth}{1}}

\setlength\parindent{0pt}
\setlength{\parskip}{5pt}
%=========================================================

%=========================================================
\section*{Introduction}
%=========================================================
In the prequel \cite{IQ1}, we introduced $q$-stability conditions on
a class of triangulated categories $\D_\XX$ with a distinguished auto-equivalence $\XX$.

In this paper, we study the case when $\D_\XX$ are Calabi-Yau-$\XX$ categories
associated to graded decorated marked surfaces.
We introduce $q$-quadratic differentials as counterpart of the $q$-stability conditions in prequel
and prove they can be identified in the surface case.

%Due to the length issue, we  write a separated paper \cite{IQZ},
%dedicated to prove technical results
%on topological realization of Calabi-Yau-$\XX$ categories from graded decorated marked surfaces.

%=========================================================
\subsection{Quadratic differentials as stability conditions}
%=========================================================
Motivated by Douglas' $\Pi$-stability of D-branes in string theory,
Bridgeland \cite{B1} introduced stability conditions on triangulated categories.
Since then, the theory of stability conditions
has played an important role in the study of mirror symmetry, Donaldson-Thomas invariants, cluster theory, etc.
In particular,  Kontsevich and Seidel observed that
spaces of stability conditions and spaces of abelian/quadratic differentials
share similar properties a couple of years ago (cf. \cite{BS}).
Recently, there are two seminal works in this direction:
\begin{itemize}
\item In \cite{BS}, Bridgeland-Smith (BS) showed that
\[
    \Stap\D_3(\surf^3)/\Aut\cong\Quad_3(\surf^3)
\]
which is upgraded in \cite{KQ2} as
\begin{gather}\label{eq:BS}
    \Stap\D_3(\surfo^3)\cong\FQuad_3^\circ(\surfo^3).
\end{gather}
Here $\D_3(\surf^3)\cong\D_3(\surfo^3)$ is the Calabi-Yau-3 category associated to a (decorated) marked surface $\surf^3$ (or $\surfo^3$),
$\Stap\D$ the space of stability conditions on a triangulated category $\D$,
$\Quad_3(\surf^3)$ the moduli space of (GMN/CY-3 type) quadratic differentials on $\surf^3$
and $\FQuad_3(\surfo^3)$ the $\surfo^3$-framed version.

\item In \cite{HKK}, Haiden-Katzarkov-Kontsevich (HKK) showed that
\begin{gather}\label{eq:above}
    \Stap\D_\infty(\gms)\cong\FQuad_\infty(\gms),
\end{gather}
where $\D_\infty(\gms)$ is the topological Fukaya category of a graded marked surface $\gms$,
$\Stap\D_\infty(\gms)$ the space of stability conditions on $\D_\infty(\gms)$
and $\Quad_\infty(\gms)$ the moduli space of (exponential type) quadratic differentials on $\gms$.
\end{itemize}
These works share many similarities,
which fit into the framework \cite{DHKK} that relates dynamical systems and triangulated categories.
For one thing,
these categories are of topological Fukaya type.
More precisely, the (spherical/indecomposable) objects in these categories
correspond to (closed/graded) arcs on the surfaces (i.e., \cite[Thm.~6.6]{QQ}, \cite[Prop.~4.2]{HKK}).
Note that $\D_3(\surfo^3)$ can be embedded into the derived Fukaya category
of some symplectic manifold constructed from $\surfo^3$ (\cite{S}, cf. the survey \cite{Q5}).
For another, the constructions of a stability condition $\sigma=(Z,\hh{P})$
from a quadratic differential $\xi$ are both roughly as follows:
\begin{itemize}
  \item a saddle connection $\eta$ with phase $\varphi$
    corresponds to (semi)stable objects $X_\eta$ in the slicing $\hh{P}(\phi)$;
  \item the integral of the square root of $\xi$ along (the double cover of) $\eta$
  gives the central charge of the corresponding stable object $X_\eta$:
\begin{gather}\label{eq:formula}
    Z(X_\eta)=\int_{\widetilde{\eta}} \sqrt{\xi}.
\end{gather}
\end{itemize}

One of the main motivations of this paper is to give concrete links between these results \cite{BS,HKK}.
Categorically, we construct the (deformed) Calabi-Yau-$\XX$ completion
$\D_\XX(\gms):=\DX$ of $\D_\infty(\gms)$
via a full formal arc system $\TT$ of the decorated version of $\gms$,
such that $\D_3(\surfo^3)$ is the (triangulated hull of the) orbit quotient:
\begin{gather}\label{eq:3}
    \D_3(\surfo^3)=\D_\XX(\gms)\sslash[\XX-3].
\end{gather}
Geometrically, we will $q$-deform HKK's result to produce a Calabi-Yau-$s$ version of BS' result,
that realizes $q$-stability conditions via $q$-quadratic differentials,
for a complex parameter $s$.

A by-product we obtained along the way is that
HKK's topological Fukaya category can be realized as the cluster category in the following sense:
\begin{gather}\label{eq:c0}
    \D_\infty(\gms)\cong\per(\gms)\cong\C_\XX(\gms):\;=\per_\XX(\gms)/\D_\XX(\gms),
\end{gather}
where $\per(\gms)$ is the Koszul dual of $\D_\infty(\gms)$ and
$\C_\XX(\gms)$ is the Calabi-Yau-$\XX$ version of the cluster category,
in the sense of Amiot-Guo-Keller \cite{A,K8,IY}.
This explains the cluster-like structure of $\Stap\D_\infty(\gms)$ in \cite[\S~6]{HKK},
cf. \cite{KQ1,Q2}.

%=========================================================
\subsection{The $q$-deformations}
%=========================================================
Given a triangulated category $\D_\infty$ with Grothendieck group $K(\D_\infty) \cong \ZZ^{\oplus n}$,
Bridgeland shows that the set of all stability conditions (with support property)
on $\D_\infty$ forms a complex manifold with dimension $n$.
There is a local homeomorphism/coordinate
\begin{gather}\label{eq:Z inf}\begin{array}{rcl}
    \mathcal{Z}_\infty \colon \Stab\D_\infty &\longrightarrow& \Hom_{\ZZ}(K(\D_\infty),\bC),\\
        (Z,\sli) &\mapsto& Z
\end{array}\end{gather}
provided by the central charge $Z$ of a stability condition $\sigma=(Z,\hh{P})$.
To $q$-deform the notion of stability conditions and this result,
we consider the Calabi-Yau-$\XX$ version of $\D_\infty$,
which is a triangulated category $\D_\XX$ admitting a distinguish autoequivalence $\XX$,
that makes the Grothendieck group as
$$K(\D_\XX) \cong R^{\oplus n}, \qquad R\colon=\ZZ[q^{\pm 1}].$$
Recall that there are natural $\bC$-action and $\Aut$-action on the set of stability conditions.

A $q$-stability condition on $\D_\XX$ consists of a stability condition $\sigma$
and a complex number $s$ satisfying the Gepner equation
\begin{gather}\label{eq:Gepner}
    \XX ( \sigma)=s \cdot \sigma,
\end{gather}
(and a technical condition, the $q$-support property).
In the prequel \cite{IQ1} we show that when fixing $s$,
the $q$-stability conditions $(\sigma,s)$ also form a complex manifold with dimension $n$
with local homeomorphism/coordinate
\begin{gather}\label{eq:Z s}\begin{array}{rcl}
    \mathcal{Z}_s \colon \QStab_s\D_{\XX} &\longrightarrow& \Hom_{R}(K(\D_{\XX}),\bC_s),\\
        (Z,\sli,s) &\mapsto& Z.
\end{array}\end{gather}
Here $\bC_s$ is the complex plane equipped with the $R$-structure
given by $q \cdot z:=e^{ \ii \pi s}z$.

On the quadratic differential side, we also need to perform $q$-deformation.
The moduli space $\FQuad_\infty(\gms)$ that corresponds to $\Stab\D_\infty(\gms)$ in \eqref{eq:above}
consists of quadratic differentials with exponential type singularities, locally of the form
\begin{gather}\label{eq:exp sing}
    e^{z^{-k}}z^{-l} g(z) \diff z^2,
\end{gather}
for some $(k,l)\in(\ZZ_{>0}\times\ZZ)$ and some locally holomorphic non-vanishing function $g(z)$.
The local homeomorphism/coordinate is given by the period map
\begin{gather}\label{eq:Pi_inf}\begin{array}{rcl}
    \Pi_\infty\colon\FQuad_\infty(\gms)&\longrightarrow&\Hom(  \Ho{1}(\gms,\partial\gms;\ZZ_\Sp),\bC),\\
        \phi&\mapsto&\displaystyle\int\sqrt{\phi},
\end{array}\end{gather}
which corresponds to \eqref{eq:Z inf}.

We would like to $q$-deform the singularity \eqref{eq:exp sing} as an $s$-pole of the form
\begin{gather}\label{eq:sing}
    e^{-k(s-2)-l} \diff z^2
\end{gather}
together with $k$ $s$-zeroes of the form $e^{(s-2)} \diff z^2 $.
Analytically, we introduce a multi-valued quadratic differential $\Theta$ on $\gms$.
Equivalently, we construct $\log$-surface $\LS$ and
view $\Theta$ as a single-valued quadratic differentials on $\LS$.
The deck transformation, also denoted by $q$, on $\LS$ corresponds to the distinguish autoequivalence $\XX$
and the equation \eqref{eq:Gepner} becomes
\begin{gather}
   q^* (\Theta)=e^{\ii \pi s}\cdot \Theta
\end{gather}
in this setting.
Then the local homeomorphism/coordinate is given by the period map
\begin{gather}\label{eq:Pi_s}\begin{array}{rcl}
    \Pi_s\colon\QQuad_s(\LS)&\longrightarrow&\Hom_R(  \HHo{}(\LS;\ZZ),\bC_s),\\
        \xi&\mapsto&\displaystyle\int\sqrt{\xi},
\end{array}\end{gather}
that corresponds to \eqref{eq:Z s},
where $\QQuad_s(\LS)$ is the moduli space of $\LS$-framed $q$-quadratic differentials.

When $\Re(s)\gg1$, we use HKK's isomorphism \eqref{eq:above}
to prove the main result (Theorem~\ref{thm:q=x}) of the paper,
which is a $q$-deformed generalization of BS' isomorphism \eqref{eq:BS}.
\begin{thma}
Suppose $\Re(s)\gg1$.
There is an isomorphism
$$\QQuad_s(\LS)\xrightarrow{\;\cong\;}\QStap_s\DX$$ between complex manifolds,
where $\QStap_s\DX$ consists of (possibly many) connected components in
the space $\QStab_s\DX$ of $q$-stability conditions.
\end{thma}

%=========================================================
\subsection{Frobenius structures on genus zero Hurwitz spaces}
%=========================================================
A Hurwitz space $\HS$ is the moduli space of meromorphic functions on a Riemann surface $\rs$.
In \cite[\S~5]{Dub1}, Dubrovin constructed Frobenius structures of
Saito type (also called flat structure in \cite{Sa1}) on Hurwitz spaces.
The primitive form \cite{Sa1,SaTa} (primary differential \cite{Dub1}) on a Hurwitz space
plays an essential role in his construction.

In Section~\ref{sec:Hurwitz},
we identify the space of $q$-quadratic differentials with the regular locus of the Hurwitz space
when the surface $\rs$ is $\kong{P}^1$.
As in Section~\ref{sec:q_def}, for $q$-quadratic differentials,
we consider $s$-zeroes of the form $z^{(s-2)}\diff z^2$ and $s$-poles of the form $z^{-k(s-2)-l}\diff z^2$,
which have the numerical data/polar type $(k,l)\in \ZZ_{\ge 1}\times\ZZ$.
On the one hand,
the positive integer $k$ determines the complex structure of Hurwitz spaces.
On the other hand,
the choice of the integer $l$ specifies the choice of the primitive form.
Thus our introduction of the new parameter $s$ and
the division of the pole order into $k(s-2)$ plus $l$
are essential for studying the moduli spaces of multi-valued quadratic differentials.
The key numerical calculation is to match the winding numbers (in Section~\ref{sec:wind}),
which shows that a combination of $k$ $s$-zeroes and one $s$-pole is indeed the $q$-deformation
of HKK's exponential type singularity \eqref{eq:exp sing}, with the same parameters $(k,l)$.

On the regular locus $\HS_{\reg}$ of a Hurwitz space $\HS$,
we can consider the almost Frobenius structure \cite{Dub2}
as the almost dual of the original Frobenius structure.
Under the identification between Hurwits spaces and moduli spaces,
the twisted periods of the almost Frobenius structure on $\HS_{\reg}$
can be identified with the periods of $q$-quadratic differentials.
Combining with our main result that identifies the moduli space of $q$-quadratic differentials
with the space of $q$-stability conditions,
we obtain the correspondence between the Hurwitz space
and the space of $q$-stability conditions in this case.

In particular, twisted periods of the almost Frobenius structure
should be identified with the central charges of $q$-stability conditions.
In fact, we present a conjectural description of the almost Frobenius structure
on the space $\QStap$ of $q$-stability conditions on a Calabi-Yau-$\XX$ category
in the case of type ADE through isomorphism between an $\QStap$ and
the universal covering of the regular locus of the Cartan subalgebra of the corresponding type.
One application of our $q$-deformation of
quadratic differentials/stability conditions is to prove this conjecture in the type $A_n$ case.
The observation here is that the Cartan subalgebra of the type $A_n$ can be identified with
the unfolding of a type $A_n$ singularity, %(which is the space of one variable polynomials),
which is a special case of a Hurwitz space.

In all, our new constructions provide an approach to understand
almost Frobenius structure on the spaces of $q$-stability conditions.

%=========================================================
\subsection{Content}
%=========================================================
All the modules and categories will be over $\k$, an algebraically closed field.
Denote $\ZZ/m\ZZ$ by $\ZZ_m$ for a positive integer $m$.

The paper is organized as follows.

In Part~\ref{part:Cat}, we establish categorical setups:
  \begin{itemize}
  \item In Section~\ref{sec:QS},
  we recall the notations and results from the prequel \cite{IQ1} on $q$-deformation of stability conditions.
  \item In Section~\ref{sec:TFuk},
  we describe topological Fukaya categories from graded marked surfaces, partially following \cite{HKK,LP,IQZ}.
  \item In Section~\ref{sec:D-CY},
  we recall Keller's deformed Calabi-Yau completion from \cite{K8} in the Calabi-Yau-$\XX$ setting.
  \item In Section~\ref{sec:CYS}, we introduce the Calabi-Yau-$\XX$ categories associated to
  graded marked surfaces, as $q$-deformations of the topological Fukaya categories.
  We recall technical results for such categories from \cite{IQZ}.
  \end{itemize}

In Part~\ref{part:Geo}, we study the geometric side of the story:
  \begin{itemize}
  \item In Section~\ref{sec:QDR},
  we recall the theory of quadratic differentials of GMN type and exponential type
  on Riemann surfaces from \cite{BS} and \cite{HKK}, respectively.
  \item In Section~\ref{sec:q_def},
  we introduce $q$-quadratic differentials as $q$-deformations of quadratic differentials of exponential type.
  The key result (Theorem~\ref{thm:winding}) here is that
  the numerical data (mainly, the winding numbers) for these two type of differentials do match.
  \item In Section~\ref{sec:main}, we prove the main results of the paper
  (Theorem~\ref{thm:main} and Theorem~\ref{thm:q=x}),
  that the induced $q$-stability conditions on Calabi-Yau-$\XX$ categories from surfaces
  can be identified with $q$-quadratic differentials.
  \end{itemize}

In Part~\ref{part:app}, we give applications:
  \begin{itemize}
  \item In Section~\ref{sec:BS-N},
  we show (Theorem~\ref{thm:BS N}) that when specializing $s=N$ to be a sufficiently large integer,
  the stability conditions on Calabi-Yau-$N$ orbit categories
  can be realized as CY-$N$ type quadratic differentials.
  \item In Section~\ref{sec:Hurwitz},
  we show (Theorem~\ref{thm:HW_Quad_iso}) that when the genus of the surface is zero,
  then framed $q$-quadratic differential can be identified with framed Hurwitz covers.
  In particular, this provides a solution of gluing $q$-stability conditions along the $s$-direction
  (cf. Corollary~\ref{cor:n+2}).
  Furthermore, when restricted to the disk case (type $A_n$),
  the moduli space of framed Hurwitz covers can be identified with the universal cover of
  the regular part of the space of polynomials of degree $n+1$ (cf. Proposition~\ref{prop:uc=HS}).
  \item In Section~\ref{sec:SF},
  we describe a conjectural almost Frobenius structure on spaces of $q$-stability conditions
  for Dynkin quivers and confirm it
  (Theorem~\ref{thm:pline}) for the type A case using the result above on Hurwitz covers.
  \end{itemize}

\begin{table}
\caption{List of notations}
\setlength{\extrarowheight}{2pt}
\begin{tabular}{ccc}
\hline
$R $&& $\ZZ[q,q^{-1}]$ \\ \hline
$\D_?(-)$&& a Calabi-Yau-$?$ category for $?\in \ZZ\cup\{\XX,\infty\}$\\ \hline
$\sigma=(Z,\hh{P}) $&& a stability conditions with central charge $Z$ and slicing $\hh{P}$ \\ \hline

\hline
$\surf=(\surf,\M,\Y)  $&& a marked surface, i.e., surface with sets of open/closed marked points \\ \hline
$\gms $&& a graded marked surface with grading $\grad$\\ \hline
$\surfo  $&& a decorated marked surface (DMS), with decoration $\Tri$ \\ \hline
$\cut  $&& a cut, i.e., a pairing between the set $\Y$ and the set $\Tri$ \\ \hline
$\surfo=(\surfo,\cut,\Lambda)$&& a graded DMS with grading $\Lambda$ (compatible with $\gms$)  \\ \hline
$\LS  $&& the log surface associated to $\surfo$, w.r.t. $\uc$ \\ \hline
$\surfo^{(m)}$&& $m^{\mathrm{th}}$ sheet of $\LS$   \\ \hline
$\gzero$&& a marked surface $\surf$ of genus 0\\ \hline
$\dd$&& a marked disk \\ \hline
$\num(?)$&& the numerical data of $?$ \\ \hline
$(\uk,\ul)$&& the polar type for exponential type/$q$- quadratic differentials \\ \hline
$\TT=\{\wg_i\} $ &&a full formal open arc system for $\gms$\\ \hline
$\TT^*=\{\we_i\} $ &&the dual full formal closed arc system\\ \hline\\[-12pt]
$\wCA(?)$ && the set of graded closed arcs on $?$ \\ \hline \\[-12pt]
$\widehat{\Ho{}}(?)=\Ho{1}(?)^-$&& the hat homology of $?$\\ \hline

\hline
%$(\widetilde{Q}_\TT, W_\TT) $&& \\ \hline
$\DI$&& the topological Fukaya category of $\gms$ \\ \hline
$\DX$&& the Calabi-Yau-$\XX$ category of $\surfo$ \\ \hline

\hline

$(\rs,\phi,h)$&& a quadratic differential $\phi$ on
    a Riemann surface $\rs$ with framing $h$ \\ \hline
$\Core(\phi)$&& the core of a quadratic differential\\ \hline
%$\rs^\phi$ && the real blow-up of $\rs$ with respect to $\phi$ \\ \hline
%$\logrsx$ && the real blow-up of the log surface of $\rs$ with respect to $\xi$ \\ \hline

$\Theta=(\rs,\xi,h;s)$&& a $\LS$-framed $q$-quadratic differential\\ \hline

$\FQuad_\infty(\gms)$ &&  the moduli space of framed quadratic differential on $\gms$ \\ \hline\\[-11pt]

$\QQuad_s^?(\LS)$ &&  the moduli space of $\LS$-framed $q$-quadratic differential \\ \hline\\[-11pt]

$\QStab_s^?\DX$&& the space of certain type $q$-stability conditions on $\DX$ \\ \hline

$\HS(g,\k)$&& the moduli space of Hurwitz cover of genus $g$ and polar type $\uk$ \\ \hline

$\HS(\log \gzero_\Tri)$&& the moduli space of $\log \gzero_\Tri$-framed Hurwitz covers\\ \hline

$\Poly_{n}$&& the space of polynomials of degree $n+1$\\ \hline
$?_{\reg}$&& the regular part of space $?$\\[.5ex] \hline
$\surf^N,\surfo^N$&& CY-$N$ marked surface and DMS   \\ \hline
$\Pone=\mathbb{P}^1$&& Riemann sphere\\ \hline
\hline
\end{tabular}
\end{table}
%=========================================================
\subsection*{Acknowledgments}
%=========================================================
We would like to thank Tom Bridgeland, Yu-Wei Fan, Fabian Haiden, Bernhard Keller,
Alastair King, Kyoji Saito, Atsushi Takahashi and Yu Zhou for inspirational discussion.
%Qy in particular would like to thank Yau Mathematical Science Center in Tsinghua and Jie Xiao for saving his career.
This paper is supported by
National Key R\&D Program of China (No. 2020YFA0713000),
Beijing Natural Science Foundation (Z180003),
World Premier International Research Center Initiative (WPI initiative), MEXT, Japan and
JSPS KAKENHI Grant Number JP16K17588.

%=========================================================
\part{Categorical Part}\label{part:Cat}
%=========================================================
\section{$q$-Deformation of stability conditions}\label{sec:QS}
%=========================================================
In this section, we recall the $q$-deformation of stability conditions from the prequel \cite{IQ1}.
%=========================================================
\subsection{Bridgeland stability conditions}\label{sec:BSC}
%=========================================================
First we recall the definition of Bridgeland stability conditions
on triangulated categories from \cite{B1}.
We assume that for a triangulated category $\D$, its Grothendieck group
$K(\D)$ is free of finite rank in this subsection, i.e., $K(\D) \cong \ZZ^{\oplus n}$ for some $n$.

\begin{definition}
\label{def:stab}
Let $\D$ be a triangulated category.
A {\it stability condition} $\sigma = (Z, \sli)$ on $\D$ consists of
a group homomorphism $Z \colon K(\D) \to \CC$ called the {\it central charge} and
a family of full additive subcategories $\sli (\varphi) \subset \D$ for $\varphi \in \R$
called the {\it slicing}
satisfying the following conditions:
\begin{itemize}
\item[(a)]
if  $0 \neq E \in \sli(\varphi)$,
then $Z(E) = m(E) \exp(\bi \pi \varphi)$ for some $m(E) \in \R_{>0}$,
\item[(b)]
for all $\varphi \in \R$, $\sli(\varphi + 1) = \sli(\varphi)[1]$,
\item[(c)]if $\varphi_1 > \varphi_2$ and $A_i \in \sli(\varphi_i)\,(i =1,2)$,
then $\Hom_{\D}(A_1,A_2) = 0$,
\item[(d)]for $0 \neq E \in \D$, there is a finite sequence of real numbers
\begin{equation}\label{eq:>}
\varphi_1 > \varphi_2 > \cdots > \varphi_m
\end{equation}
and a collection of exact triangles (\emph{Harder-Narasimhan filtration})
\begin{equation}\label{eq:HN}
0 =
\xymatrix @C=5mm{
 E_0 \ar[rr]   &&  E_1 \ar[dl] \ar[rr] && E_2 \ar[dl]
 \ar[r] & \dots  \ar[r] & E_{m-1} \ar[rr] && E_m \ar[dl] \\
& A_1 \ar@{-->}[ul] && A_2 \ar@{-->}[ul] &&&& A_m \ar@{-->}[ul]
}
= E
\end{equation}
with $A_i \in \sli(\varphi_i)$ for all $i$.
\end{itemize}

For each non-zero object $0 \neq E \in \D$, we define two real numbers by
$\varphi_{\sigma}^+(E):=\varphi_1$ and $\varphi_{\sigma}^-(E):=\varphi_m$ where $\varphi_1$ and
$\varphi_m$ are determined by the axiom (d).
Nonzero objects in $\sli(\varphi)$ are called {\it semistable of phase $\varphi$} and simple objects
in $\sli(\varphi)$ are called {\it stable of phase $\varphi$}.
For a non-zero object $E \in \D$ with extension factors $A_1,\dots,A_m$
given by axiom (d), define the mass of $E$ by
\[
m_{\sigma}(E):=\sum_{i=1}^m |Z(A_i)|.
\]

For a stability condition $\sigma = (Z,\sli)$,
we introduce the set of semistable classes
$\ss(\sigma) \subset K(\D)$ by
\begin{equation}\label{eq:ss}
\ss(\sigma) :=\{\,\alpha \in
K(\D)\,\vert\,\text{there is a semistable object }
E \in \D \text{ such that } [E] = \alpha\,\}.
\end{equation}
Let $\norm{\,\cdot\,}$ be some norm on $K(\D) \otimes \R$.
A stability condition $\sigma=(Z,\sli)$
satisfies the {\it support property} if there is some constant $C >0$ such that
\begin{equation}\label{eq:supp}
    C \cdot |{Z(\alpha)}| > \norm{\alpha}
\end{equation}
for all $\alpha \in \ss(\sigma)$.
We will assume that the stability conditions we considered always satisfy support property.
\end{definition}

The key result by Bridgeland is the following.
\begin{theorem}[\cite{B1}, Theorem 1.2]
\label{thm:B}
The projection map of taking central charges
\[
    \mathcal{Z}\colon \Stab\D \longrightarrow \Hom(K(\D),\CC),\qquad
    \quad(Z,\hh{P}) \mapsto Z
\]
is a local homeomorphism of topological spaces. In particular,
$\mathcal{Z}$ induces a complex structure on $\Stab\D$.
\end{theorem}

There are two group actions on $\Stab\D$ commuting with each other.
The first one is the natural $\CC$ action
\[
    s \cdot (Z,\hh{P})=(Z \cdot e^{-\mathbf{i} \pi s},\hh{P}_{\Re(s)}),
\]
where $\hh{P}_x(\phi)=\hh{P}(\phi+x)$.
There is also a natural action on $\Stab\D$ induced by $\Aut\D$, namely:
$$\Phi  (Z,\hh{P})=\big(Z \circ \Phi^{-1}, \Phi (\hh{P}) \big).$$
%=========================================================
\subsection{$q$-stability conditions}
%=========================================================
Let $\D_{\XX}$ be a triangulated category with a distinguished auto-equivalence
$$\XX \colon \D_{\XX} \to \D_{\XX}.$$
%Here $[\XX]$ is not necessarily the Serre functor.
We will write $E[l \XX]$ instead of $\XX^l(E)$ for
$l \in \ZZ $ and $E \in \D_{\XX}$.
Set $R=\ZZ[q^{\pm 1}]$ and define the action of $R$ on $K(\D_{\XX})$ by
\[
    q^n \cdot [E] := [E[n \XX]].
\]
Then $K(\D_{\XX})$ has an $R$-module structure.

\begin{definition}
A \emph{$q$-stability condition} consists of a (Bridgeland) stability condition $\sigma=(Z,\hh{P})$ on $\D_\XX$
and a complex number $s$ satisfying
\begin{gather}\label{eq:X=s}
    \XX(\sigma)=s \cdot \sigma.
\end{gather}
%We may write $\sigma[\XX]$ for $\XX(\sigma)$.
\end{definition}

In the rest of this paper, we assume the following:
\begin{assumption}\label{ass:R}
The Grothendieck group $\Grot (\D_{\XX})$ is free of finite rank
over $R$, i.e., $\Grot(\D_{\XX} )\cong R^{\oplus n}$ for some $n $.
\end{assumption}

Similar to usual stability conditions, we also always require
$q$-stability conditions satisfy the \emph{$q$-support property}
(see \cite[Def.~3.8]{IQ1}).

Denote by $\QStab_s\D_\XX$ the set of all $q$-stability conditions
satisfying $q$-support property and with fixed $s$.
Let $$\QStab\D_\XX=\displaystyle{\bigcup_{s\in\bC}}\QStab_s\D_\XX.$$

Let $\bC_s$ denote the complex numbers with the $R$-module structure
given by $$q \cdot z:=e^{ \ii \pi s}z$$ for $z \in \bC_s$,
that corresponds to the specialization \eqref{eq:qs}.
The following is our $q$-deformed version of Bridgeland's original result (Theorem~\ref{thm:B}).

\begin{theorem}\cite[Thm.~3.10]{IQ1}
\label{thm:localiso2}
Let $\D_\XX$ be a triangulated category with $\Grot(\D_{\XX} )\cong R^{\oplus n}$ and fix $s\in\bC$.
The projection map of taking central charges
\begin{equation*}
\mathcal{Z}_s \colon \QStab_s\D_{\XX}
\longrightarrow \Hom_{R}(K(\D_{\XX}),\bC_s),\qquad
(Z,\sli) \mapsto Z
\end{equation*}
is a local homeomorphism of topological spaces. In particular,
$\mathcal{Z}_s$ induces a complex structure on $\QStab_s\D_{\XX}$.
\end{theorem}

%=========================================================
\subsection{$\XX$-baric heart and $q$-stability conditions}\label{sec:QSC}
%=========================================================
\begin{definition}[$\XX$-baric heart]\label{def:X-baric heart}
%Let $\D_{\XX}$ be a CY-$\XX$ category.
An \emph{$\XX$-baric heart} $\D_{\infty} \subset \D_{\XX}$
is a full triangulated subcategory of $\D_{\XX}$ satisfying the following conditions:
\begin{itemize}
\item[(1)]if $k_1 > k_2$ and $A_i \in \D_{\infty}[k_i\XX]\,(i =1,2)$,
then $\Hom_{\D_{\XX}}(A_1,A_2) = 0$,
\item[(2)]for $0 \neq E \in \D_{\XX}$,
there is a finite sequence of integers
\begin{equation*}
k_1 > k_2 > \cdots > k_m
\end{equation*}
and a collection of exact triangles
\begin{equation*}
0 =
\xymatrix @C=5mm{
 E_0 \ar[rr]   &&  E_1 \ar[dl] \ar[rr] && E_2 \ar[dl]
 \ar[r] & \dots  \ar[r] & E_{m-1} \ar[rr] && E_m \ar[dl] \\
& A_1 \ar@{-->}[ul] && A_2 \ar@{-->}[ul] &&&& A_m \ar@{-->}[ul]
}
= E
\end{equation*}
with $A_i \in \D_{\infty}[k_i\XX]$ for all $i$.
\end{itemize}
\end{definition}

Note that, by definition, the classes of objects in $\D_{\infty}$
span $K(\D_{\XX})$ over $R$ and there is a canonical isomorphism
\begin{gather}\label{eq:KKK}
    K(\D_{\infty}) \otimes_{\ZZ} R \cong K(\D_{\XX}).
\end{gather}
Recall that the triangulated category $\D_\XX$ is Calabi-Yau-$\XX$
if $\XX$ is the Serre functor (cf. \eqref{eq:serre}).
For an $\XX$-baric heart $\D_{\infty}$ in a Calabi-Yau-$\XX$ category $\D_\XX$,
Condition $(1)$ in Definition~\ref{def:X-baric heart} is equivalent to
\begin{gather}\label{eq:01}
    \Hom_{\D_{\XX}}(A_1,A_2) = 0
\end{gather}
for $A_i\in\D_\infty[k_i\XX]$ and $k_1-k_2\notin\{0,1\}$.

For a fixed complex number $s \in \bC$, consider the specialization
\begin{gather}\label{eq:qs}\begin{array}{ccl}
    q_s\colon\bC[q,q^{-1}]&\to&\bC,\\
        q&\mapsto&e^{\mathbf{i} \pi s}.
\end{array}\end{gather}

\begin{construction}\label{con:q}
Consider a triple $(\D_\infty,\ns,s)$ consisting of
an $\XX$-baric heart $\D_\infty$, a (Bridgeland) stability condition $\ns=(\nz,\np)$ on $\D_\infty$
and a complex number $s$.
We construct
\begin{enumerate}
\item the additive pre-stability condition $\sadd=(Z,\padd)$ and
\item the extension pre-stability condition $\sext=(Z,\pext)$,
\end{enumerate}
where
\begin{itemize}
\item we first extend $\nz$ to
\[
    Z_q\colon=\nz \otimes 1\colon K(\D_\XX)\to\bC[q,q^{-1}]
\]
via \eqref{eq:KKK} and
$$Z=q_s\circ (\widehat{Z}\otimes 1)\colon K(\D_\XX)\to\bC$$
gives a central charge function on $\D_\XX$;
\item the pre-slicing $\padd$ is defined as
\begin{equation}\label{eq:padd}
    \padd(\phi)=\add^s\np[\ZZ\XX]\colon=\add \bigoplus_{k\in\ZZ}  \np(\phi-k\Re(s))[k\XX];
\end{equation}
\item the pre-slicing $\pext$ is defined as
\begin{equation}\label{eq:pext}
    \pext(\phi)=\<\np[\ZZ\XX]\>^s\colon=\<  \np(\phi-k\Re(s))[k\XX] \>.
\end{equation}
\end{itemize}
Note that $\sigma$ does not necessarily satisfy condition (d) in Definition~\ref{def:stab}
and hence may not be a stability condition.
\end{construction}

\begin{definition}\cite{IQ1,Q3}
Given a slicing $\hh{P}$ on a triangulated category $\D$.
Define the \emph{global dimension} of $\hh{P}$ by
\begin{gather}\label{eq:geq}
\gldim\hh{P}=\sup\{ \phi_2-\phi_1 \mid
    \Hom(\hh{P}(\phi_1),\hh{P}(\phi_2))\neq0\}.
\end{gather}
For a stability condition $\sigma=(Z,\hh{P})$ on $\D$,
its global dimension $\gldim\sigma$ is defined to be $\gldim\hh{P}$.
\end{definition}

\begin{definition}\label{eq:def:qstab}
An \emph{open $q$-stability condition} on $\D_{\XX}$ is a pair $(\sigma,s)$
consisting of a stability condition $\sigma$ on $\D_\XX$ and a complex parameter $s$ such that
\begin{itemize}
    \item $\sigma=\sadd$ is an additive pre-stability condition induced from some triple
    $(\D_\infty,\ns,s)$
    as in Construction~\ref{con:q}.
\end{itemize}
A \emph{closed $q$-stability condition} on $\D_{\XX}$ is a pair $(\sigma,s)$
consisting of a stability condition $\sigma$ on $\D_\XX$ and a complex parameter $s$
such that
\begin{itemize}
    \item $\sigma=\sext$ is an extension pre-stability condition induced from some triple
    $(\D_\infty,\ns,s)$
    as in Construction~\ref{con:q}.
\end{itemize}
\end{definition}
%Denote by $\OStab_s\D_\XX$ the set of all {\color{red}open $q$-stability conditions with parameter $s\in\bC$ and by $\OStab\D_\XX$ the union of all $\OStab_s\D_\XX$.
%and by $\CStab_s\D_\XX$} the set of all closed $q$-stability conditions with parameter $s\in\bC$.

\begin{theorem}\cite{IQ1}\label{thm:IQ}
Suppose that $\D_\XX$ is Calabi-Yau-$\XX$ with an $\XX$-baric heart $\D_\infty$.
A triple $(\D_\infty,\ns,s)$ induces an open $q$-stability condition $(\sigma,s)$ for $\sigma=\sadd$
if and only if
\begin{equation}\label{eq:open}
    \gldim\ns+1<\Re(s).
\end{equation}
A triple $(\D_\infty,\ns,s)$ induces a closed $q$-stability condition $(\sigma,s)$ for $\sigma=\sadd$
if and only if
\begin{equation}\label{eq:closed}
    \gldim\ns+1\le\Re(s)
\end{equation}
\end{theorem}

We will still write
\[
    (\sigma,s)=(\D_\infty,\ns,s)\otimes_{\oplus} R,\qquad
    (\sigma,s)=(\D_\infty,\ns,s)\otimes_{*} R
\]
for the two induced $q$-stability conditions in the above theorem if there causes no confusion.
We will study subspaces $\OStab_s\D_\XX$ and $\CStab_s\D_\XX$ of $\QStab_s\D_\XX$
consisting of certain induced open and closed $q$-stability conditions.

%=========================================================
\section{Topological Fukaya categories}\label{sec:TFuk}
%=========================================================
%=========================================================
\subsection{Graded marked surface}
%=========================================================
We recall the relative notions and notations on graded marked surfaces,
cf. \cite{IQZ,HKK,LP}.
Note that we will use a single marked point to denote a marked boundary arc in their setting,
where if one performs the real blow-up (with respect to some quadratic differential)
at the corresponding singularity,
such a marked point corresponds to a line (isomorphic to $\RR$).

\begin{definition}\label{def:GMS}
A \emph{graded marked surface} $(\surf,\M,\Y,\grad)$ consists of the following data:
\begin{itemize}
\item $\surf$ is a smooth oriented surface with $\partial\surf\ne\emptyset$,
\item $\M$ is a set of open marked points on $\partial\surf$,
such that $\M\cap\partial_i\neq\emptyset$ for each boundary component $\partial_i$ of $\partial\surf$
(cf. the blue bullet points in Figure~\ref{fig:QR}).
\item $\Y$ is a set of closed marked points on $\partial\surf$,
such that points in $\M$ and in $\Y$ are alternating in $\partial\surf$
(cf. the red circle points in Figure~\ref{fig:QR}).
Let $\aleph:=|\Y|$.
\item $\grad$ is a grading (or line field) on $\surf$, that is, a section $\lambda:\surf\to\PP T\surf$
of the projectivized tangent bundle $\PP T\surf$.
\end{itemize}
\end{definition}

For simplicity, we write $\surf$ for the triple $(\surf,\M,\Y)$, called \emph{marked surface},
and $\gms$ for the graded version.
We will exclude the case when $\surf$ is a disk with $\aleph=2$.

A \emph{morphism} $\gms \to \surf'^{\grad'}$ of graded marked surfaces
is a pair $(f,\widetilde{f})$
consisting of an orientation preserving local diffeomorphism $f\colon \surf\to\surf'$,
that preserves the sets of marked points setwise,
and a (homotopy class of) path $\widetilde{f}$ from $f^*(\grad')$ to $\grad$
in the space of sections of $\PP T\surf$.
Every graded marked surface $\surf$ admits an automorphism $[1]$, known as the shift,
given by the pair $(\mathrm{id}_{\surf},\overrightarrow{\pi})$,
where restricted to the generator of $\pi_1(\PP T_p\surf)$,
$\overrightarrow{\pi}$ is rotating clockwise by an angle of $\pi$, for any $p\in \surf$.

A \emph{curve} in a graded marked surface $\gms$ is
an immersion $c\colon I\to \surf$ for a 1-manifold $I$.
A \emph{grading} $\widetilde{c}$ on \emph{a curve} $c$ is given by a family of (homotopy classes of) paths in $\P T_{c(t)} \surf$ from $\grad(c(t))$ to $\dot{c}(t)$,
varying continuously with $t\in(0,1)$.
The pair $(c,\widetilde{c})$ is called a \emph{graded curve}, and will be simply denoted by $\widetilde{c}$ usually.
The pushforward of a graded curve of a graded morphism $(f,\widetilde{f})$ is given by
$(I,f\circ c,(c^*\widetilde{f})\cdot\widetilde{c})$.

A point of transverse intersection of a pair $\big((I_1,c_1,\widetilde{c}_1), (I_2,c_2,\widetilde{c}_2)\big)$ of graded curves determines an integer as follows.
Suppose $t_i\in I_i$ such that
\begin{equation}
    c_1(t_1)=c_2(t_2)=p\in \surf \quad \text{and} \quad
    \dot{c}_1(t_1)\neq \dot{c}_2(t_2)\in\PP T_p \surf.
\end{equation}
We have the following homotopy classes of paths in $\PP T_p \surf$:
\begin{enumerate}
\item $\widetilde{c}_1(t_1)$ from $\grad(p)$ to $\dot{c}_1(t_1)$,
\item $\widetilde{c}_2(t_2)$ from $\grad(p)$ to $\dot{c}_2(t_2)$,
\item $\kappa$ from $\dot{c}_1(t_1)$ to $\dot{c}_2(t_2)$ given by
counterclockwise rotation in $T_p\surf$ by an angle less than $\pi$.
\end{enumerate}
Define the \emph{intersection index} of $c_1,c_2$ at $p$
\begin{equation}
\ind_p(c_1,c_2)=\widetilde{c}_1(t_1)\cdot\kappa\cdot\widetilde{c}_2(t_2)^{-1}\ \in\pi_1(\PP T_p\surf)\cong\mathbb{Z}.
\end{equation}
It is well-defined if $c_1$ and $c_2$ are in general position.

%=========================================================
\subsection{Grading as cohomology/covering}
%=========================================================
The grading $\lambda:\surf\to\PP T\surf$ can be interpreted as
a cohomology class $[\grad]$ in $\coho{1}(\PP T\surf)$ in this case.
The projection $\PP T\surf\to\surf$ with $\R\PP^1\backsimeq \mathrm{S}^1$-fiber gives a short exact sequence
$$0\to\pi_1(\mathrm{S}^1)\to\pi_1(\PP T\surf)\to\pi_1(\surf)\to0,$$
or (taking the abelization)
$$0\to\ho{1}(\mathrm{S}^1)\to\ho{1}(\PP T\surf)\to\ho{1}(\surf)\to0.$$
Dually, we have
\begin{gather}\label{eq:lambda}
    0\to \coho{1}(\surf) \to \coho{1}(\PP T\surf) \xrightarrow{\pi_\surf} \coho{1}(\mathrm{S}^1)=\ZZ \to0.
\end{gather}
\cite[Lemma~1.2]{LP2} shows that $\lambda$ is determined by a class $[\grad]\in\coho{1}(\PP T\surf)$,
induced from a split of $\pi_\surf$ (i.e., $\pi_\surf([\lambda])=1$ in $\coho{1}(\mathrm{S}^1)$).
Such a data is equivalent to a split of $\ho{1}(\mathrm{S}^1)\to\ho{1}(\PP T\surf)$
or a split of $\pi_1(\mathrm{S}^1)\to\pi_1(\PP T\surf)$.

A more comprehensive way to think about $\grad$ is via the \emph{Maslov ($\ZZ$-)covering}
$$\pi_{\grad}\colon\R T\gms\to\PP T\surf$$
determined by $\grad$, where $\R T\gms$ is the $\R$-bundle of $\surf$
that can be constructed by gluing $\ZZ$ copies of $\PP T\surf$ cut out by $\lambda$.
A morphism between two graded marked surfaces then consists of
a map $f\colon\gms\to\surf'^{\grad'}$ such that it preserves the marked points and
$[\grad']=f^*[\grad']$, regarding $[\grad]\in \coho{1}(\PP T\surf)$, together with
a map $\widetilde{f}$ between their Maslov covering that commutes with $f$.
The automorphism/grading shift $[1]$ is then the deck transformation of the Maslov covering.

In this description, a graded curve
$\widetilde{c}$ is one of $\ZZ$ lifts in $\R T\gms$,
of the tangent $\dot{c}$ of an usual curve $c$ on $\surf$.
The intersection index $i=\ind_p(\widetilde{c}_1,\widetilde{c}_2)$ is the shift $[i]$ such that
the lift $\widetilde{c}_2[i]\mid_p$ of $\widetilde{c}_2[i]$ at $p$ is in the length one interval
$$( \widetilde{c}_1\mid_p , \widetilde{c}_1[1]\mid_p )\subset \mathbb{R}T_p\ \cong \R .$$

%=========================================================
\subsection{Arc systems}
%=========================================================
We have the following notions.
\begin{itemize}
\item $\M$ (resp. $\Y$) divides $\partial\surf$ into $\aleph$
\emph{open (resp. closed) boundary arcs}, each of which contains a closed (resp. open) marked point.
\item An \emph{open (resp. closed) curve} $\gamma$ in $\surf$
is (the isotopy class of) a (non-trivial) curve connecting points in $\M$ (resp. in $\Y$),
such that $\gamma$ is in $\surf^\circ=\surf\setminus\partial\surf$ except for its endpoints.
\item An \emph{(open/closed) arc} is
an (open/closed) curve without self-intersection in $\surf^\circ$.
\item An open (resp. closed) \emph{arc system} $\TT$ is
a collection of pairwise disjoint graded open (resp. closed) arcs.
\item An open (resp. closed) \emph{full formal} arc system is
an open (resp. closed) arc system $\TT$ such that
it cuts $\surf$ into polygons (called $\TT$-polygons),
each of which has exactly one open (resp. closed) boundary arc.
See blue arcs in Figure~\ref{fig:QR}.
\end{itemize}
The graded version of an open/closed curve $a$ will usually be denoted by $\widetilde{a}$.
Denote by $\OA(\gms)$ (resp. $\CA(\gms)$) the set of open (resp. closed) arcs on $\gms$,
and by $\wOA(\gms)$ (resp. $\wCA(\gms)$) the set of graded ones.

Choose (and fix) an initial full formal open arc system $\TT=\{\wg_j\}$.
%(which is in fact induced from some quadratic differentials, cf. \cite{HKK}).
The dual graph of $\TT$, denoted by $\TT^*=\{\we_j\}$,
is also a full formal closed arc system satisfying
\[
    \Int^\bullet(\wg_i,\we_j)=\Int^0(\wg_i,\we_j)=\delta_{ij},
\]
where $\Int^d$ is the index-$d$ intersection number and
$\Int^\bullet=\sum_{d\in\ZZ} \Int^d$ is the total intersection number.
As we only care about full formal arc systems,
we will omit full formal for simplicity.
The picture on the top of Figure~\ref{fig:QR} shows an example of
\emph{a dual arc system} $(\TT,\TT^*)$,
where $\TT$ consists of the blue arcs and $\TT^*$ consists of red arcs.

\begin{figure}[h]\centering
\begin{tikzpicture}[yscale=.4,xscale=.6,arrow/.style={->,>=stealth,thick}]
\draw[thick] (0,0) circle (7);
\draw[blue, \separated ,thick, font=\scriptsize] (180-360/7*5:7)
    edge [bend right=-35] (180-360/7*4:7)
    edge [bend right=-15] (180-360/7*3:7)
    edge [bend right=-5] (180-360/7*2:7);
\draw[blue, \separated ,thick, font=\scriptsize](180-360/7*2:7)
    to[bend right=-35] (180-360/7*1:7)
    to[bend right=-35] (180-360/7*0:7)
    to[bend right=-35] (180-360/7*7:7)
    to[bend right=-35] (180-360/7*6:7);
\foreach \j in {1,...,7}
\draw(180-\j*360/7:7)node[cyan]{$\bullet$};

\draw[red,dashed, thick, font=\scriptsize] (360/7*5:7)
    edge [bend right=5]node[right]{$\we_4$} (360/7*1:7)
    edge [bend right=10]node[right]{$\we_3$} (360/7*2:7)
    edge [bend right=15]node[left]{$\we_2$} (360/7*3:7)
    edge [bend right=25]node[above]{$\we_1$} (360/7*4:7);
\draw[red,dashed, thick, font=\scriptsize] (360/7*1:7)
    to[bend left=-25] node[left]{$\we_5$}(360/7*0:7)
    to[bend left=-25] node[right]{$\we_6$}(-360/7:7)
    (360/7*5:7)node[below]{$Y$};
\foreach \j in {1,...,7}
\draw(\j*360/7:7)node[white]{$\bullet$} node[red]{$\circ$};
\end{tikzpicture}

\begin{tikzpicture}[scale=.9]
\foreach \j in {1,2,3,4}
{\draw[blue] (-54-72*\j:1.7)node(v\j){$\bullet$}(-54-72*\j:2)node{$\j$};;}
\foreach \j/\k in {1/2,2/3}{
    \draw[blue](-54-72-72*\k:1.7)node(w\k){$\bullet$};
    \draw[dashed,-stealth,blue] (v\j)to(w\k);
    }
\foreach \j in {1,2,3}{\foreach \k in {\j+1}{
    \draw[blue](-54-72-72*\k:1.7)node(w\k){$\bullet$};
    \draw[-stealth,blue] (v\j)to(w\k);
    }}
%\draw[](54:1.75)node{\tiny{$d_{3}$}}
%(54+72:1.75)node{\tiny{$d_{2}$}}
%(54+72*2:1.75)node{\tiny{$d_{1}$}};
\draw[blue](-54+72:1.7)node(v4){$\bullet$}
    (0:3)node[blue](v5){$\bullet$}node[right,blue]{$5$}
    (-25:3)node[blue](v6){$\bullet$}node[left,blue]{$6$};

\draw[-stealth,blue] (v5)to(v4);
\draw[-stealth,blue] (v6)to(v5);
\end{tikzpicture}
\quad
\begin{tikzpicture}[scale=.9]
\foreach \j in {1,2,3,4}
{\draw[red] (-54-72*\j:1.7)node(v\j){$\bullet$}(-54-72*\j:2)node{$\j$};;}
\foreach \j/\k in {1/2,2/3,1/3}{
    \draw[red](-54-72-72*\k:1.7)node(w\k){$\bullet$};
    \draw[-stealth,red] (v\j)to(w\k);
    }
\foreach \j in {1,2,3}{\foreach \k in {\j+1}{
    \draw[red](-54-72-72*\k:1.7)node(w\k){$\bullet$};
    \draw[-stealth,red] (v\j)to(w\k);
    }}
%\draw[](54:1.75)node{\tiny{$d_{3}$}}
%(54+72:1.75)node{\tiny{$d_{2}$}}
%(54+72*2:1.75)node{\tiny{$d_{1}$}};
\draw[red](-54+72:1.7)node(v4){$\bullet$}
    (0:3)node[red](v5){$\bullet$}node[right,red]{$5$}
    (-25:3)node[red](v6){$\bullet$}node[left,red]{$6$};

\draw[-stealth,red] (v5)to(v4);
\draw[-stealth,red] (v6)to(v5);
\end{tikzpicture}
\caption{The quiver with relation (in blue) associated to an open arc system and
the Ext-quiver corresponding to the dual closed arc system (in red)}\label{fig:QR}
\end{figure}
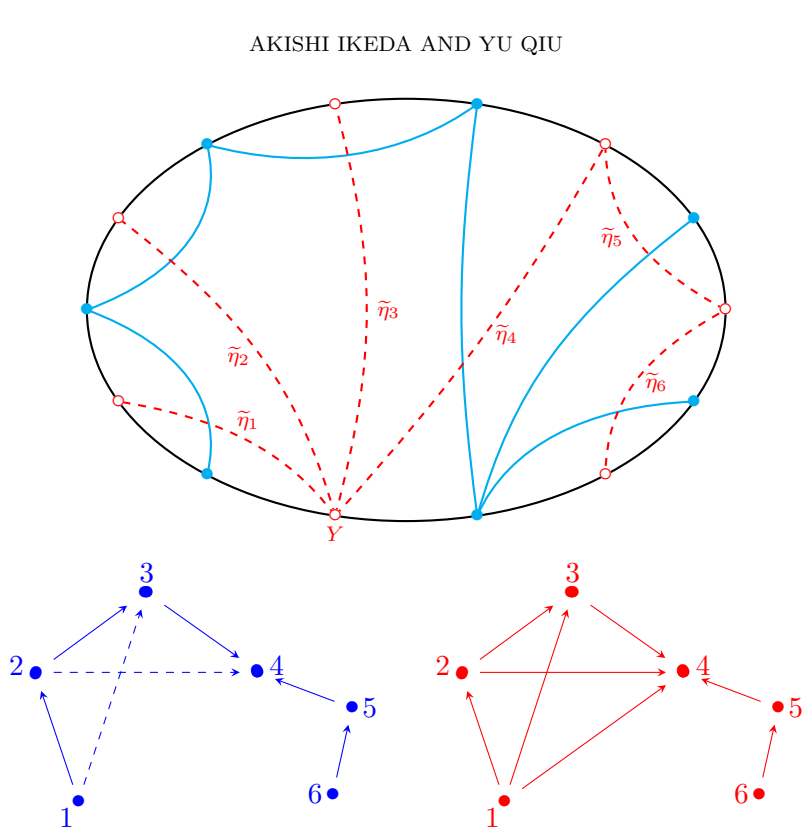
%=========================================================
\subsection{Numerical data}
%=========================================================

\begin{definition}\label{def:numerical}
The \emph{numerical data} $$\num=\num(\gms)=(g,b;\uk,\ul;\LP_g)$$ of $\gms$ consists of
\begin{itemize}
  \item the genus $g=g_\surf$;
  \item the number $b=|\partial\surf|$ of boundary components;
  \item the ordered multi-set $\uk=(k_1,\ldots,k_b)\in\ZZ_{>0}^b$ of \emph{orders},
  where for each boundary component $\partial_i$, $k_i$ is the number of marked points on $\partial_i$,
  satisfying
  \begin{gather}\label{eq:order k_i}
    \sum_{i=1}^b k_i=|\M|=|\Y|;
  \end{gather}
  \item the ordered multi-set $\ul=(l_1,\ldots,l_b)\in\ZZ^b$ of \emph{indices},
  where for each boundary component $\partial_i$, one can calculate $l_i$ as follows:
  let $\{Y^i_j \mid j\in\ZZ_{k_i}\}$ be the closed marked points on $\partial_i$ (in the clockwise order
  of the induced orientation of $\surf$) and
  $\we^i_j=Y^i_jY^i_{j+1}$ be the closed boundary arcs (with any grading), then
  \begin{gather}\label{eq:winding l_i}
    \begin{array}{rl}
        \wind{\lambda}(\partial_i)=&2-l_i\\
        =& \displaystyle{\sum_{j\in\ZZ_{k_i}}} \left( \ind_{Y^i_j}(\we^i_{j-1},\we^i_j)-1\right).
    \end{array}
  \end{gather}
  Here $\wind{\lambda}(\partial_i)$ is the \emph{(clockwise) winding number} of $\partial_i$ with respect to $\lambda$
  (cf. \cite[\S~1.2]{LP2} and \cite[(3.21)]{HKK}).
  More precisely, we have $\wind{\lambda}(-)=\<\lambda,-\>$, cf. \eqref{eq:lambda}, for the natural paring:
  \[
    \<-,-\>\colon\coho{1}(\PP T\surf)\times\ho{1}(\PP T\surf)\to\ZZ.
  \]
  Note that the indices satisfy (\cite[(1.5)]{LP2})
  \begin{gather}\label{eq:4-4g}
    \sum_{i=1}^b l_i=4-4g.
  \end{gather}
  We will use quadratic differentials to calculate
  the winding number $\wind{\lambda}(\partial_i)$ and $l_i$ in Section~\ref{sec:wind}.
  Denote by $\uk=(k_1,\ldots,k_b)$ and $\ul=(l_1,\ldots,l_b)$ the data of $(k_i,l_i)$.
  \item \emph{the Lekili-Polishchuk (LP) data} $\LP_g$, defined as follows:
    \begin{itemize}
      \item if $g=0$, then $\LP_0=\emptyset$;
      \item if $g=1$, then $\LP_1=\widetilde{A}$, where $\widetilde{A}:=\widetilde{A}(\lambda)$ is defined to be the non-negative integer $\gcd\{\wind{\lambda}(\gamma)\mid\text{non-separating $\gamma$}\}$;
      \item if $g>1$, then $\LP_g=(A,\kappa)$ consisting of
      the Arf invariant $A$ (see \cite[\S~1.2]{LP2}) and the indicator $\kappa$,
      where $\kappa=0$ if $\wind{\lambda}^{(2)}\equiv0$ and $\kappa=1$ otherwise. Here
        $$\wind{\lambda}^{(2)}\colon\Ho{1}(\surf;\ZZ_2)\to\ZZ_2 $$
        is the induced function (\cite[Definition~1.2.1]{LP2}).
    \end{itemize}
\end{itemize}
\end{definition}

\begin{example}\label{ex:num}
Consider two genus zero cases.
\begin{itemize}
\item When $\surf$ is a disk, which corresponds to a type $A_n$ quiver,
then the numerical data is $$(g=0,\;b=1;\;\uk=(n+1),\;\ul=(4);\, \LP_0=\emptyset).$$
\item When $\surf$ is an annulus that corresponds to the affine type $\widetilde{A_{p,q}}$ quiver
(with zero grading for all arrows), the numerical data is
\[
    (g=0,\;b=2;\;\uk=(p,q),\;\ul=(2,2);\,\LP_0=\emptyset).
\]
In general, the indices can be chosen to be $\ul=(2+w,2-w)$ for some $w\in\ZZ$
and the corresponding quiver $\widetilde{A_{p,q}}$ will be graded.
\end{itemize}
Note that in the case $g=0$, $A$ and $\kappa$ are irrelevant.
\end{example}

%=========================================================
\subsection{The topological Fukaya categories}
%=========================================================
The \emph{graded quiver with relation} $(Q^{(0)}_\TT, R^{(0)}_\TT)$ associated to
an open arc system $\TT$ is defined as follows:
\begin{itemize}
  \item its vertices $\{i\}$ of $Q^{(0)}_\TT$ correspond to the open arcs $\{\wg_i\}=\TT$;
  \item its arrows correspond to the (anticlockwise) angle between arcs of $\TT$ in the polygons
  of $\surf$ cutting out by $\TT$ and
  the grading of the arrows is given by the intersection index at the intersection (vertex of the angle);
  \item the quadratic relations are non-composable arrows=angles.
\end{itemize}
Let $\k$ be an algebraic closed field and $\ha^0_\TT=\k Q^{(0)}_\TT/R^{(0)}_\TT$.
Next, we produce a differential graded algebra (we will write dga for short) from $\ha^0_\TT$,
replacing the relation with differential.

\begin{construction}\cite[Cons.~2.2]{Op}
Consider the basic finite dimensional $\k$-algebra $\ha^0_\TT$.
Let $Q^{(1)}_\TT$ be the (graded) quiver obtained from $Q^{(0)}_\TT$ by adding
the arrows corresponding to the (minimal set of) relations of $R^{(0)}_\TT$,
whose degrees are the degrees of the corresponding relations minus one,
and a differential (of degree $1$), such that
$\coho{0}(\ha^0_\TT)$ is a quotient of $\coho{0}(\k Q^{(1)}_\TT)$.
Now pick a generating set $R^{(1)}_\TT$ for
$$\ker \left(   \coho{1}(\k Q^{(1)}_\TT)\to \coho{0}(\ha^0_\TT) \right)$$
to satisfy that $\k Q^{(1)}_\TT/R^{(1)}_\TT$ is quasi-isomorphic to $\ha^0_\TT$.
Iterating this process till we obtain a graded quiver $Q_\TT$ (without relation)
and a differential $\diff$
such that $\ha_\TT:=\k Q_\TT$ is quasi-isomorphic to $\ha^0_\TT$.
\end{construction}

\begin{remark}\label{rem:Ext-q}
By simple-projective duality
\[
    \Irr^d(P_i,P_j)\cong\Ext^{1-d}(S_j,S_i)^*,
\]
$Q_\TT$ can be also constructed from Ext-quiver (\cite[Def.~6.2]{KQ1}) of $\ha^0_\TT$.
More precisely,
Let $Q^{\Ext}(\ha^0_\TT)$ be the graded $\Ext$-quiver of $\ha^0_\TT$,
whose vertices are simple $\ha^0_\TT$ module and whose graded arrows are given by graded morphisms between them.
Then $Q_\TT$ is obtained from $Q^{\Ext}(\ha^0_\TT)$ by the degree changing: $d\mapsto 1-d$.

Moreover,
$Q^{\Ext}(\ha^0_\TT)$ is the intersection quiver of $\TT^*$, in the sense that
its vertices one-one correspond to graded closed arcs in $\TT^*$
and its graded arrows one-one correspond to graded (by index) intersections between them.
For instance, the blue solid quiver in Figure~\ref{fig:QR} is $Q^{(0)}_\TT$
while the relation $R^{(0)}_\TT$ is presented by blue dashed arrows;
the red quiver is the Ext-quiver of $Q_\TT$.

Furthermore, such a duality topologically corresponds to the graph duality between $\TT$ and $\TT^*$
(see \cite{Q3} for the Calabi-Yau-3 case of this correspondence).
\end{remark}

Denote by $\TFuk(\gms)$ the \emph{topological Fukaya category} of $\gms$, which
in this case can be identified as
\[
    \TFuk(\gms)\cong\per\ha_\TT\cong\D_{fd}(\ha_\TT).
\]
The above triangle equivalence follows from the fact that
we require that any boundary component contains closed and open marked points
(so that $\ha_\TT$ is homological smooth and proper).
Lekili-Polishchuk proves the following.
\begin{theorem}\cite{LP2}
Let $\surf^{\grad_i}_i$, $i=1,2$, be a graded marked surface with the numerical data $\num(\surf_i)$.
Then $\num(\surf^{\grad_1})=\num(\surf^{\grad_2})$ implies the triangle equivalence
$\TFuk(\surf_1^{\grad_1})\cong\TFuk(\surf_2^{\grad_1})$.
\end{theorem}
We primarily interest in
\begin{gather}
    \DT\colon=\D_{fd}(\ha_\TT).
\end{gather}
%In general, for any arc system $\TT$ of $\surf$,
%the associated strictly unital $A_\infty$ category $\hh{F}_{\TT}$
%associated to $\TT$ is quasi-isomorphic to $\ha_\TT$.

Denote by $\Ind\C$ the set of indecomposable objects in a category $\C$.
The classification of objects in $\DT$ is given in \cite{HKK}.
However, we will need the version in \cite[\S~6]{IQZ},
which is slightly weaker on objects but with an additional intersection formula.

\begin{theorem}\cite{HKK,IQZ}\label{thm:IQZ}
There is an injection
\begin{gather}\label{eq:CA0}
  \Xinf\colon \wCA(\gms) \longrightarrow \Ind\DT
\end{gather}
such that
\begin{gather}\label{eq:CA}
    \Int^d(\we_1,\we_2)=\dim\Hom^d(\Xinf(\we_1),\Xinf(\we_2)),
\end{gather}
for any $\Xinf(\we_i)\in\wCA(\gms)$.
\end{theorem}

Similarly, there is an injection
\[\Xinf^\vee\colon\wOA(\gms) \longrightarrow \Ind\per\ha_\TT,\]
but we will not use this map.
Moreover, Theorem~\ref{thm:IQZ} implies that there is a canonical embedding
(recall that we exclude the case when $\surf$ is a disk with two closed marked points)
\[
    \MCG_\bullet(\gms) \to \Aut^\circ\DT.
\]
Here, $\MCG_\bullet(\gms)$ is the \emph{(marked) mapping class group} of $\gms$, i.e.,
the group of isotopy classes of morphisms of $\gms$,
where all maps are required to fix $\M$ and $\Y$ setwise.
And $\Aut^\circ\DT$ is the quotient group of the \emph{auto-equivalence group} $\Aut\DT$
by those auto-equivalences of $\DT$ which acts trivially on the corresponding space of stability conditions $\Stab\DT$.
Note that in this case, the auto-equivalence that preserves $\Xinf(\wCA(\gms))$
will preserve $\Stab\DT$ since $\Xinf(\wCA(\gms))$ contains stable objects
(see Theorem~\ref{thm:HKK} for more details).

%\begin{remark}
%Let $\widehat{M}_j^i=\widehat{M}(\gamma_j^i)$, where $\gamma_j^i$ are the boundary arcs in
%Definition~\ref{def:numerical}.
%Then,by the definition of the topological Fukaya category, the index $i_{Y^i_j}(\gamma^i_{j-1},\gamma_j^i)$ %equals
%the degree of a non-vanishing element in $\Hom^\bullet(\widehat{M}_{j-1}^i,\widehat{M}_j^i)$.
%\end{remark}

%=========================================================
\section{Deformed Calabi-Yau-$\XX$ completion}\label{sec:D-CY}
%=========================================================
In this section, we review Keller's deformed Calabi-Yau completion \cite{K8}.
%=========================================================
\subsection{$\ZZ^2$-graded differential modules and categories}
%=========================================================
We will work with \emph{$\ZZ^2$-graded $\k$-modules} ($\ZZ^2$-graded mods),
So such a $\ZZ^2$-graded mod $M$ admits a decomposition $M=\bigoplus_{(q,s)\in\ZZ^2} M_s^q$.
Denote by $\deg m=(q,s)\in\ZZ^2$ the degree of an element $m\in M^q_s$.
A $\ZZ^2$-graded morphism $f\colon M\to N$ between $\ZZ^2$-graded mods with degree $(p,t)$
is a $\k$-linear morphism such that $f(M^q_s)\subset N^{q+p}_{s+t}$.
The first component of $\ZZ^2$-grading is the \emph{cohomological grading}
and denote the cohomological degree of $m\in M^q_s$ by $|m|=q$.
Similarly $|f|$ denotes the cohomological degree of a $\ZZ^2$-graded morphism.
The tensor product $M\otimes L$ is
$(M\otimes L)^q_s:=\bigoplus_{}M^{q_1}_{s_1}\otimes L^{q_2}_{s_2}$
for $(q_1,s_1)+(q_2,s_2)=(q,s)$
and $f\otimes g\colon M\otimes L\to M'\otimes L'$ is
\begin{gather*}
(f\otimes g)(m\otimes l)\colon=(-1)^{|m|\cdot|g|}f(m)\otimes g(l),
\end{gather*}
for $f\colon M\to M', g\colon L\to L'$,
where only the cohomological degree effects the sign.

A \emph{differential $\ZZ^2$-graded $\k$-module} ($\ZZ^2$-dg-mod) is a pair
consisting of a $\ZZ^2$-graded mod $M$ and
a ($\k$-linear) \emph{differential} map $\diff_M\colon M\to M$ with degree $(1,0)$ such that $\diff_M^2=0$.
The cochain $\coch{q}(M)$ and cohomology $\coho{q}(M)$ are defined with respect to
the differential $\diff$ (or the cohomological degree),
which admits a $\ZZ$-grading (that corresponds the second grading on $M$)
$$\coch{q}(M)=\bigoplus_s \coch{q}_s(M),\quad \coho{q}(M)=\bigoplus_s \coho{q}_s(M).$$
\begin{remark}
There are two natural shifts $M\{1\}$ and $M[1]$ on $\ZZ^2$-dg-mod $M$:
\[\begin{array}{cl}
(M\{1\})^q_s=M^q_{s+1}, &\diff_{M\{1\}}=\diff_M\\
(M[1])^q_s=M^{q+1}_s, &\diff_{M[1]}=-\diff_M.
\end{array}\]
\end{remark}
A \emph{$\ZZ^2$-dg-morphism} $f\colon M\to L$ between $\ZZ^2$-dg-mods is a $\k$-linear map
such that $\deg f=(0,0)$ and $f\circ\diff_M=\diff_L\circ f$,
which induces a morphism (of graded modules) on cohomologies $\coho{\bullet}$.
The tensor product $M\otimes L$ of two $\ZZ^2$-dg-mods is the tensor $\ZZ^2$-graded mod
with $\diff_{M\otimes L}=\diff_M\otimes \id_L+\id_M\otimes\diff_L$.
The morphism space $\hom(M,L)$ is also a $\ZZ^2$-dg-mod, where the component $\hom^q_s(M,L)$
consists of all $\k$-linear maps of degree $(q,s)$ and
$$\diff_f\colon=\diff_M\circ f-(-1)^{|f|}f\circ\diff_L.$$

A \emph{$\ZZ^2$-graded differential category} ($\ZZ^2$-dg-cat) $\ha$ is a $\k$-category
whose morphism spaces are $\ZZ^2$-dg-mods and whose compositions
$\ha(Y,Z)\otimes \ha(X,Y) \to \ha(X,Z)$
are $\ZZ^2$-dg-morphisms.
A $\ZZ^2$-dg-functor $F\colon \ha\to\hh{B}$ between $\ZZ^2$-dg-cats is given by a map
between their objects and by morphisms of $\ZZ^2$-dg-mods
$F(X,Y)\colon \ha(X,Y)\to\hh{B}(FX,FY)$ for $X,Y\in\Obj\ha$.

The opposite $\ZZ^2$-dg-cat $\ha^{op}$ of $\ha$ has the same objects
with morphisms $\ha^{op}(X,Y)=\ha(Y,X)$ and the composition of
$f\in\ha^{op}(Y,X)$ and $g\in\ha^{op}(Z,Y)$ is given by $(-1)^{|f||g|}gf$.
The cochain and homology category $\coch{0}(\ha), \coho{0}(\ha)$
of $\ha$
have the same objects of $\ha$
with morphisms $(\coch{0}\ha)(X,Y)=\coch{0}(\ha(X,Y))$
and $(\coho{0}\ha)(X,Y)=\coho{0}(\ha(X,Y))$ respectively.

For two $\ZZ^2$-dg-functors $F,G\colon\ha\to\hh{B}$,
the complex of graded morphisms $\hom(F,G)$ consists of a family of morphisms
$\phi_X\colon\hh{B}(FX,FY)^n$ such that $(Gf)(\phi_X)=(\phi_Y)(Ff)$
with the induced differential from $\hh{B}(FX,FY)$.
The set of morphisms between $F$ and $G$ is given by the set $\coch{0}_0\hom(\ha,\hh{B})(F,G)$.

%=========================================================
\subsection{Derived categories of $\ZZ^2$-dg-cats}
%=========================================================
Denote by $\hh{C}_{dg}(A)$ the $\ZZ^2$-dg-cat of a $\ZZ$-graded $\k$-algebra $A=\bigoplus_{s\in\ZZ} A_s$.
Let $\ha$ be a small $\ZZ^2$-dg-cat.
A (right) $\ZZ^2$-dg-$\ha$-mod $M$
is a $\ZZ^2$-dg-functor
\[M\colon\ha^{op}\to\hh{C}_{dg}(\k).\]
There are also two natural shifts $[1],\{1\}$ for a $\ZZ^2$-dg-$\ha$-mod $M$,
the cohomological shift $[1]$ and the extra grading shift $\{1\}$, such that
\[
    M[1]\{1\}(X):=M(X)[1]\{1\}\in\hh{C}_{dg}(\k).
\]
For each object $X$ of $\ha$, there is a right module $X^\wedge=\ha(?,X)$
represented by $X$.

The $\ZZ^2$-dg-cat of $\ha$ is defined to be
$$\hh{C}_{dg}(\ha)\colon=\hom(\ha^{op},\hh{C}_{dg}^\ZZ(\k))$$
and we denote its morphism by $\hom_{\ha}$.
The category of $\ZZ^2$-dg-$\ha$-mods $\hh{C}(\ha)$
has the $\ZZ^2$-dg-$\ha$-mods as objects and the morphisms of $\ZZ^2$-dg-functors as morphisms
and we have
\[\hh{C}(\ha)=\coch{0}\hh{C}_{dg}(\ha).\]
The homotopy category of $\ZZ^2$-dg-$\ha$-mods is $\hh{H}(\ha)=\coho{0}\hh{C}_{dg}(\ha)$
whose morphisms are given by $\coho{0}_0\hom_{\ha}$.
Note that there are canonical isomorphisms
$\hom(X^\wedge,M)\xrightarrow{\sim}M(X)$ and
\[\coho{}(\ha)(X^\wedge, M[n])\xrightarrow{\sim}\coho{n}M(X).\]
\begin{definition}
Denote by $\hh{D}(\ha)$ the \emph{derived category} of $\hh{C}(\ha)$ (or $\h(\ha)$)
with respect to the class of quasi-isomorphisms and its homomorphisms by $\Hom_{\hh{D}(\ha)}$.
An important fact is that $\hh{D}(\ha)$ also admits two equivalences,
the triangulated shift $[1]$ and the extra grading shift $\{1\}$.
We will write $[\XX]$ for $\{1\}$
and $[m+l\XX]$ for $[m]\circ[\XX]^l$, where $m,l\in\ZZ$.
\end{definition}

Recall that a $\ZZ^2$-dg-mod $P$ is \emph{cofibrant} if any morphism $P\to M$ factors through $L$
for a surjective quasi-isomorphism $L\to M$;
a $\ZZ^2$-dg-mod $I$ is \emph{fibrant} if any morphism $L\to I$ extends to $M$
for an injective quasi-isomorphism $L\to M$.

\begin{lemma}\cite[Prop.~3.1]{Ke2}
For each $\ZZ^2$-dg-mod $M$, there is a quasi-isomorphism (known as cofibrant resolution) $\mathbf{p}M\to M$
with cofibrant $\mathbf{p}M$ and a quasi-isomorphism (known as fibrant resolution) $M\to\mathbf{i}M$ with fibrant $\mathbf{i}M$.
Moreover, the projection functor $\h(\ha)\to\hh{D}(\ha)$
admits a fully faithful left/right adjoint given by
$M\mapsto\mathbf{p}M$ and $M\mapsto\mathbf{i}M$ respectively.
Thus,
\[\h(\ha)(\mathbf{p}L,M)=\Hom_{\hh{D}(\ha)}(L,M)=\h(\ha)(L,\mathbf{i}M) .\]
%and $\RHom_{\hh{D}(\ha)}(L,M)=$.
\end{lemma}
%=========================================================
\subsection{The inverse dualizing complex and Calabi-Yau categories}
%=========================================================
Suppose that $\ha$ is homologically smooth, i.e., $\ha$ is perfect as an $\Ae{A}$-mod.
Here, $\Ae{A}$ is the $\ZZ^2$-dg-cat $\ha\otimes\ha^{op}$,
which admits an involution $V^e$ such that
\[V^e(X,Y)=(Y,X)\quad\text{and}\quad V^e(f\otimes g)=(-1)^{|f|\cdot|g|}g\otimes f,\]
which is a preduality on $\Ae{A}$.

A preduality $\ZZ^2$-dg-functor $V$ on a $\ZZ^2$-dg-cat $\ha$ is a $\ZZ^2$-dg-functor $V\colon\ha\to\ha^{op}$,
such that the opposite functor $V^{op}$ is a right adjoint of $V$.
For a (right) $\ZZ^2$-dg-mod $M$, its dual left $\ZZ^2$-dg-mod $M^*\colon\ha\to\hh{C}_{dg}(\k)$
sends $X$ to $\hom_{\ha}(M,X^\wedge)$.
The $V$-dual of $M$ is a $\ZZ^2$-dg-mod
\[\begin{array}{rcl}
    M^\vee=M^*\circ V^{op}\colon \ha&\to&\hh{C}_{dg}(\k),\\
    X&\mapsto&\hom_{\ha}(M,(V^{op}X)^\wedge).
\end{array}\]
This induces a preduality functor, still denoted by $V$:
\[V\colon\hh{C}_{dg}(\ha)\to \hh{C}_{dg}(\ha^{op}).\]
%with derived functor
%\[\OPL V\colon\D(\Ae{A})\to\D((\Ae{A})^{op}).\]

Now let $\ha$ be a small $\ZZ^2$-dg-cat.
As $\k$ is a field, $\ha$ is cofibrant over $\k$.

\begin{definition}
The \emph{inverse dualizing complex} $\Theta_\ha$ is any cofibrant replacement of
the image of $\ha$ (considered as a right $\ZZ^2$-dg-$\Ae{A}$-mod)
under the total derived functor of $V^e$.
\end{definition}
Denote by $\D_{fd}(\ha)$ the finite dimensional derived category
of $\ha$, which is the full subcategory of $\D(\ha)$
formed by the $\ZZ^2$-dg-mod $M$ such that $\sum_i \dim\coho{i}(M)$ is finite.

\begin{lemma}\cite[Lemma~3.4]{K8}\label{lem:CY}
Then for any $L\in\D(\ha)$ and $M\in\D_{fd}(\ha)$,
there is a canonical isomorphism
\[
    \Hom_{\D(\ha)}(L\otimes_{\ha} \Theta_{\ha}, M)\xrightarrow{\sim}
    \mathrm{D}\Hom_{\D(\ha)}(M,L),
\]
where $\mathrm{D}=\Hom_\k(?,\k)$.
\end{lemma}

Let $\hh{N}=m+l\XX$.
A triangulated category $\D$ is called \emph{Calabi-Yau-$\hh{N}$} (CY-$\hh{N}$)
if, for any objects $X,Y$ in $\hh{D}$ we have a natural isomorphism
\begin{gather}\label{eq:serre}
    \mathfrak{S}:\Hom (X,Y)
        \xrightarrow{\sim} \mathrm{D}\Hom (Y,X[\hh{N}]).
\end{gather}
Further, an object $S$ is \emph{$\hh{N}$-spherical} if
$\Hom^{\ZZ^2}(S, S)=\k \oplus \k[-\hh{N}]$
and \eqref{eq:serre} holds functorially for $X=S$ and $Y$ in $\D$,
where
\[
    \Hom^{\ZZ^2}(X,Y)\colon=\bigoplus_{m,l\in\ZZ} \Hom(X,Y[m+l\XX]).
\]
By Lemma~\ref{lem:CY},
if $\Theta_{\ha}\cong\ha[-\hh{N}]$, then $\D_{fd}(\ha)$
is Calabi-Yau-$\hh{N}$.
We are in particular interested in the case when $\hh{N}=\XX$ or $\hh{N}=N\in\ZZ$.
In these two cases, there is a \emph{twist functor} $\Phi_S\in\Aut\D$
for each $\hh{N}$-spherical object $S$,
defined by
\begin{equation}\label{eq:sphtwist+}
    \Phi_S(X)=\Cone\left(S\otimes\Hom^{\ZZ^2}(S,X)\to X\right)
\end{equation}
with inverse
\begin{equation}\label{eq:sphtwist-}
    \Phi_S^{-1}(X)=\Cone\left(X\to S\otimes\Hom^{\ZZ^2}(X,S)^\vee \right)[-1].
\end{equation}
Note that the graded dual of a graded $\k$-vector space
$V=\oplus_{m,l\in\ZZ} V_{m,l}[m+l\XX]$
is
\[V^\vee=\bigoplus_{m,l\in\ZZ} \mathrm{D} V_{m,l}[-m-l\XX].\]
%=========================================================
\subsection{Deformed Calabi-Yau completion}
%=========================================================
Let $\Theta_\ha$ be the inverse dualizing complex of $\ha$
and $\theta=\Theta_{\ha}[\XX-1]$.
The Calabi-Yau-$\XX$ completion of $\ha$ is the tensor DG category
\[
    \Pi_\XX(\ha)=T_{\ha}(\theta)\colon=\ha\oplus\theta\oplus(\theta\otimes_{\ha}\theta)\oplus\cdots.
\]
Moreover, let $c$ be an element of the Hochschild homology
\begin{gather}\label{eq:HH}
    \HH_{\XX-2}(\ha)=\operatorname{Tor}_{\XX-2}^{\Ae{A}}(\ha,\ha)
\end{gather}
or a closed morphism of degree $1$ form $\Theta$ to $\ha$ (i.e., in $\Hom_{\D(\Ae{A})}(\theta,\ha[1])$ (cf. \cite[\S~5.1]{K8}).
\begin{definition}
The deformation $\Pi_\XX(\ha,c)$ obtained from $\Pi_\XX(\ha)$
by adding $c$ to the differential is called a \emph{deformed Calabi-Yau completion} of $\ha$
with respect to $c$.
\end{definition}

\begin{theorem}\cite[Thm.~3.17]{Y}
Suppose $\ha$ is finitely cellular.
Assume that the element $c \in \HH_{\XX-2}(\ha)$ can be lifted to
an element $\tilde{c}$ of the negative cyclic homology $\operatorname{HN}_{\XX-2}(\ha)$.
Then the deformed Calabi-Yau completion
$\Pi_\XX(\ha,c)$ is homologically smooth and Calabi-Yau-$\XX$.
\end{theorem}
\begin{remark}
Recently, Keller \cite{K18} corrected some error in \cite{K8}.
As remarked in \cite{K18}, the assumption in the above theorem can be
satisfied in the case of Ginzburg dga which
we will deal with in Section~\ref{sec:CYS}.
\end{remark}

Moreover, we prove a lemma,
which is a slightly weaker Calabi-Yau-$\XX$ version of \cite[Lemma~4.4 and Remark~5.3]{K8}.
\begin{definition}\cite[Def.~7.2]{KQ1}
A functor $\hh{L}\colon\D(\ha)\to\D(\Pi_\XX(\ha,c))$ is a \emph{Lagrangian(-$\XX$) immersion} if
\begin{itemize}
\item for $L,M\in\D_{fd}(\ha)$, we have
\begin{gather}\label{eq:HOM}
    \RHom_\Pi(\hh{L}(L),\hh{L}(M))=\RHom_{\ha}(L,M)
        \oplus D\RHom_{\ha}(M,L)[-\XX].
\end{gather}
\end{itemize}
In particular, it is fully faithful restricted to the finite dimensional derived categories, i.e.
\begin{gather}\label{eq:hom}
    \Hom^{}_{\D_{fd}(\Pi_\XX(\ha,c))} (\hh{L}(L),\hh{L}(M))=\Hom^{}_{\D_{fd}(\ha)} (L,M)
\end{gather}
for any $L,M\in\D_{fd}(\ha)$.
\end{definition}

\begin{lemma}\label{lem:L-inf}
The canonical projection (on the first component) $\Pi_\XX(\ha,c)\to\ha$ induces
a Lagrangian immersion $$\hh{L}\colon\D_{fd}(\ha) \to \D_{fd}(\Pi_\XX(\ha,c)).$$
Moreover, the image of $\hh{L}$ is an $\XX$-baric heart of $\Pi_\XX(\ha,c)$.
\end{lemma}
\begin{proof}
By definition, \eqref{eq:HOM} holds for any (double) shifts of simple $\ha$-modules.
Denote by $\Sim\D_{fd}(\ha)$ the set of all shifts of simple $\ha$-modules.
Note that any object $M$ in $\D_{fd}(\ha)$ admits a simple filtration with factors
in $\Sim\D_{fd}(\ha)$ (which is a refinement of the canonical filtration with respect to the heart
$\ha-\mod$-the category of $\ha$-modules).
Thus, \eqref{eq:HOM} follows by induction
(on the numbers of factors of simple filtration of $M$ and $L$).

For the second statement: note that any object in $\D_{fd}(\Pi_\XX(\ha,c))$ has a filtration
with factors in $$\{\Sim\D_{fd}(\ha)[m+l\XX]\}_{m,l\in\ZZ}.$$
Moreover, $\ha$ has non-negative $\XX$-grading.
Thus $\< \Sim\D_{fd}(\ha)[m] \>_{m\in\ZZ}=\D_{fd}(\ha)$ is an $\XX$-baric heart.
\end{proof}

%=========================================================
\section{$q$-Deformation of topological Fukaya categories}\label{sec:CYS}
%=========================================================
Recall that we have a graded marked surface $\gms=(\surf,\M,\Y,\grad)$
with dual arc systems $(\TT,\TT^*)$.
%=========================================================
\subsection{Quivers with superpotential}
%=========================================================
\begin{definition}\label{con:Qs}
The \emph{$\ZZ^2$-graded quiver with superpotential}
$(\widetilde{Q}_{\TT}, W_{\TT})$ is defined as follows:
\begin{itemize}
  \item the vertices in $\widetilde{Q}_{\TT}$ are arcs in $\TT$;
  \item for each arrow $a_{}\colon \wg\to \wg'$ in $\widetilde{Q}_\TT$, add a dual arrow
  $a^*_{}\colon \wg'\to \wg$, where their degrees in $\widetilde{Q}_{\TT}$ are
  \[
    \begin{cases}
        \DEG a_{}=\deg a_{},\\
        \DEG a^*_{}=(2-\XX)-\deg a_{};\end{cases}
  \]
  \item for each $\TT$-polygon $D$, let its edges be $\wg_1,\ldots,\wg_m$
  in clockwise order (and follows by the open boundary arc of $D$).
  Then there are graded arrows
  $$\xymatrix{\wg_i \ar@<.5ex>[r]^{ a_{ij} }
    \ar@{<-}@<-.5ex>[r]_{ a^*_{ji} }  & \wg_j}
    \quad\text{for}\quad 1\leq i<j\le m-1.$$
  in $\widetilde{Q}_{\TT}$ and let (composing from left to right)
  \begin{gather}\label{eq:W}
    W_D=\sum_{1\leq i<j<k\leq m}  a_{ij} a_{jk} a^*_{ik};
    %\epsilon_{ijk}\cdot
  \end{gather}
  \item there is also a loop $\wg^*$ at each vertex $\wg$ of degree $1-\XX$ in $\widetilde{Q}_{\TT}$;
  \item the superpotential $W_{\TT}$ is the sum of $W_D$'s for all $\TT$-polygons $D$.
\end{itemize}
\end{definition}

\begin{figure}[h]\centering
\begin{tikzpicture}[scale=.9]
\foreach \j in {1,2,3,4}
{\draw[blue] (-54-72*\j:1.7)node(v\j){$\bullet$}(-54-72*\j:2)node{$\j$};;}
\foreach \j/\k in {1/2,2/3}{
    \draw[blue](-54-72-72*\k:1.7)node(w\k){$\bullet$};
    \draw[dashed,-stealth,blue] (v\j)to(w\k);
    }
\foreach \j in {1,2,3}{\foreach \k in {\j+1}{
    \draw[blue](-54-72-72*\k:1.7)node(w\k){$\bullet$};
    \draw[-stealth,blue] (v\j)to(w\k);
    }}
%\draw[](54:1.75)node{\tiny{$d_{3}$}}
%(54+72:1.75)node{\tiny{$d_{2}$}}
%(54+72*2:1.75)node{\tiny{$d_{1}$}};
\draw[blue](-54+72:1.7)node(v4){$\bullet$}
    (0:3)node[blue](v5){$\bullet$}node[right,blue]{$5$}
    (-25:3)node[blue](v6){$\bullet$}node[left,blue]{$6$};

\draw[-stealth,blue] (v5)to(v4);
\draw[-stealth,blue] (v6)to(v5);
\end{tikzpicture}
\quad
\begin{tikzpicture}[scale=.9]
\foreach \j in {1,2,3,4}
{\draw[red] (-54-72*\j:1.7)node(v\j){$\bullet$}(-54-72*\j:2)node{$\j$};;}
\foreach \j in {1,2,3}{\foreach \k in {\j,...,3}{
    \draw[red](-54-72-72*\k:1.7)node(w\k){$\bullet$};
    \draw[-stealth,red] (v\j.-90-36*\j-36*\k-60)to(w\k.-90-36*\j-36*\k+60);
    \draw[-stealth,thin, Emerald] (w\k.-90-36*\j-36*\k+120)to(v\j.-90-36*\j-36*\k-120);
    }}
%\draw[](54:1.75)node{\tiny{$d_{3}$}}
%(54+72:1.75)node{\tiny{$d_{2}$}}
%(54+72*2:1.75)node{\tiny{$d_{1}$}};
\draw[red](-54+72:1.7)node(v4){$\bullet$}
    (0:3)node[red](v5){$\bullet$}node[right,red]{$5$}
    (-25:3)node[red](v6){$\bullet$}node[left,red]{$6$};

\draw[-stealth,red] (v5.190)to(v4.-40);
\draw[-stealth,red] (v6.100)to(v5.-120);
\draw[-stealth,thin, Emerald] (v4.15)to(v5.135);
\draw[-stealth,thin, Emerald] (v5.-65)to(v6.45);
\end{tikzpicture}
\caption{From the quiver with relation to the double graded quiver with potential}\label{fig:CY double}
\end{figure}
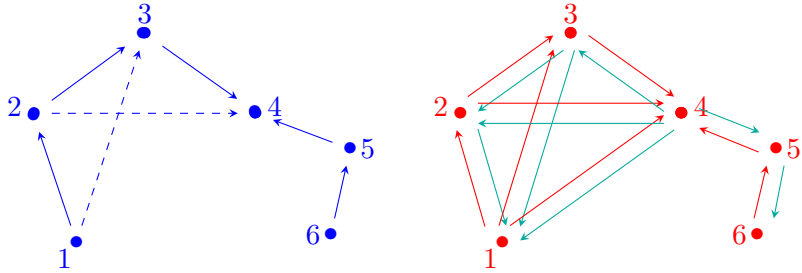

\begin{lemma}
The superpotential $W_\TT$ defined as above is homogenous of degree $3-\XX$.
\end{lemma}
\begin{proof}
It suffices to show that, for any $\TT$-polygon $D$, $W_D$ is homogenous of degree $3-\XX$.
Since $\TT$ is full formal, there is a unique marked point $Y_D$ in $\Y$ contained in $D$.
Then the closed arcs in $\TT^*$ starting at $Y$,
denoted by $\we_1,\ldots,\we_m$ (in clockwise order,
cf. Figure~\ref{fig:QR} for $m=4$),
are the dual arcs of edges $\wg_1,\ldots,\wg_m$ of $D$.
Note that an arc appears twice in $\{\we_i\}$ if it is a loop based at $Y$.
Suppose that the indices of intersection at $Y$ between graded arrows are
$$d_{ij}=\ind_Y(\we_i, \we_j),
  \quad\text{for}\quad 1\leq i<j\le m-1.$$
By the composition rule, we have
\begin{gather}\label{eq:d_ij}
    \deg d_{ij} = \sum_{k=i}^{j-1} d_{k k+1}
\end{gather}
for any $1\leq i<j\le m$.
Hence the corresponding arrow $a_{ij}$ in $Q_{\TT}$ has degree
\begin{gather}\label{eq:a_ij}
    \deg a_{ij}=1- d_{ij},
\end{gather}
and we have
\[\begin{array}{rl}
    &\DEG a_{ij}+\DEG a_{jk}+\DEG a^*_{ik}\\
    =& (1-d_{ij})+(1-d_{jk})+\big( (2-\XX)-(1+d_{ik}) \big)\\
    =&3-\XX
\end{array}\]
for $1\leq i<j<k\leq m$, as required.
\end{proof}

%We will choose the coefficients $\epsilon_{ijk}\in\{\pm\}$
%in \eqref{eq:W} satisfy a relation
%\begin{gather}\label{eq:epsilon}
%    \epsilon_{ijk} \epsilon_{ikl} + \epsilon_{ijl} \epsilon_{jkl}=0
%\end{gather}
%for any vertex $1\leq i<j<k<l\leq m$.
%This can be done inductively.

%=========================================================
\subsection{Calabi-Yau-$\XX$ and cluster-$\XX$ categories of surfaces}
%=========================================================

\begin{definition}\label{def:GXQW}
The \emph{Calabi-Yau-$\XX$ Ginzburg $\ZZ^2$-dg algebra} $\GAX$ is defined as follows:
\begin{itemize}
\item the underlying graded algebra of $\GAX$ is the completion of
the graded path algebra $\k \widetilde{Q}_\TT$;
\item the differential $\diff=\diff_\TT$ of degree $1$
is the unique continuous linear endomorphism satisfying $\diff^2=0$ and the Leibniz rule
(with respect to degree $1$) and it takes the following values:
\begin{itemize}
%\item $\diff a=0$ for $a\in (\widetilde{Q}_\TT)_1$;
\item[$\blacktriangleright\quad$] $\diff a = \partial_{a^*} W_{\TT}$ for $a\in (\widetilde{Q}_\TT)_1$;
\item[$\blacktriangleright\quad$] $\diff \displaystyle{\sum_{\gamma \in (\widetilde{Q}_\TT)_0}} \gamma^* =
    \displaystyle{\sum_{a\in (\widetilde{Q}_\TT)_1 }} \, [a,a^*]$.
\end{itemize}
\end{itemize}
%Let $\XX$ be the grading shift of $(0,1)$.
Note that we have the following convention:
\begin{itemize}
\item $\partial abc=(-1)^{\deg c}ab+(-1)^{\deg a}bc+(-1)^{\deg b}ca$
for any term $abc$ in $W_\TT$;
\item $[a,a^*]=(-1)^{\deg a}aa^*+(-1)^{\deg a^*}a^*a$;
\end{itemize}
to ensure $\diff^2=0$. %See \cite{IQZ} for the detailed calculation.
\end{definition}

Let $\D(\GAX)$ be the derived category of $\ZZ^2$-dg $\GAX$-modules,
$\D_{fd}(\GAX)$ the finite dimensional derived category and
$\per\GAX$ the perfect derived category.
They satisfy $\D_{fd}(\GAX)\subset\per\GAX\subset\D(\GAX)$.

Equivalently, we have the following.
Let $\hh{R}_\TT$ be the discrete $\k$-category associated to the vertex set of $Q_\TT$
and $\ha_\TT$ the path category of $Q_\TT$.
Then $\GAX$ is the tensor category over $\hh{R}_\TT$ of the bimodule
\[
    \widetilde{Q}_\TT=Q_\TT\oplus \left(Q_\TT\right)^\vee[\XX-2]\oplus
        \hh{R}_\TT[\XX-1].
\]
Denote by $c_\TT$ the image of $W_\TT$ in \eqref{eq:HH}.

\begin{theorem}\cite[Prop.~6.3]{K3} \cite[Thm.~3.17]{Y}
\label{thm:CY}
The Ginzburg dga $\GAX$ is quasi-isomorphic to
 the deformed Calabi-Yau completion $\Pi_\XX(\ha_\TT,c_\TT)$.
Hence it is homologically smooth and Calabi-Yau-$\XX$.
\end{theorem}
Moreover, we obtain the following statement by applying Lemma~\ref{lem:L-inf}.
\begin{corollary}\label{cor:X-heart}
The projection $\Pi_\XX(\ha_\TT,c)\to\ha_\TT$ induces
a Lagrangian immersion
\begin{gather}\label{eq:Lagrangian}
    \hh{L}_\TT\colon \Dinf \to \D_{fd}(\GAX)
\end{gather}
and its image is an $\XX$-baric heart of
\begin{equation}\label{eq:DX=}
    \DX\colon=\D_{fd}(\GAX).
\end{equation}
\end{corollary}
By abuse of notations, we will identify $\Dinf=\DT$ with its image under $\hh{L}_\TT$ sometimes.

In the case for $\D_\XX=\DX$, we will study the principal part of the spaces of
open/closed $q$-stability conditions.
\begin{definition}\label{def:QStap}
Denote by $\CStab_s\DX$ the subspace of $\QStab_s\DX$
consisting of closed $q$-stability conditions $\Psi(\sigma)$
for $\Psi\in\ST\DX$ and $\sigma$ is induced from
some triple $(\DT,\ns,s)$ as in Theorem~\ref{thm:IQ} for some $\ns\in\Stap\DT$ and $s\in\bC$.
Similarly, denote by $\OStab_s\DX$ the open version.
\end{definition}

%\begin{remark}\label{rem:inducing}
%A consequence of Corollary~\ref{cor:X-heart} is that, by applying Theorem~\ref{thm:IQ},
%we can use stability conditions on $\DT$ to construct
%$q$-stability conditions on $\DX$.
%\end{remark}

Applying the proof of \cite[Thm.~6.7]{IQ1}, we have the corresponding statement.
\begin{theorem}\label{thm:p=c}
There is a canonical triangle equivalence
between the topological Fukaya category
\[  \TFuk(\gms)\colon=\per\ha_\TT  \]
and the cluster-$\XX$ category
\begin{gather}\label{eq:CXT}
    \CXT\colon=\per\GAX/\D_{fd}(\GAX).
\end{gather}
\end{theorem}

%=========================================================
\subsection{Graded decorated marked surfaces}
%=========================================================
We fix the dual arc systems $(\TT,\TT^*)$ on $\gms$.
\begin{definition}\label{def:DMS}
We introduce the following notions.
\begin{itemize}
\item The \emph{decorated marked surface} (DMS) $\surfo$ of $\surf$ is obtained from $\surf$ by decorating a set
$\Tri=\{Z_i\}_{i=1}^\aleph$ of points in the interior of $\surf$, where $|\Tri|=|\Y|=|\M|$.
\item A \emph{(topological) cut} $\cut=\{c_i\}$ is a set of curves on $\surf$,
pairing (connecting) points in $\Tri$ and $\Y$ with no intersections or self-intersections.
See green arcs in Figure~\ref{fig:-cut}.
\item A \emph{grading} $\Lambda$ on $\surfo$ (with respect to $\cut$) is a class in $\coho{1}(\mathbb{P}T(\surf\setminus\Tri),\ZZ^2)$, with values
\begin{itemize}
\item $(1,0)$ on each (clockwise) loop $\{p\}\times\R\mathbb{P}^1$ on $T_p(\surf\setminus\Tri)$ for $p\notin\Tri$,
\item $(-2,1)$ on each (clockwise) loop $l_Z\times\{x\}$ on $\surf$ around any $Z\in\Tri,x\in\R\mathbb{P}^1$, and
\item $(?,0)$ on any simple closed curve $\alpha$ (i.e., a loop that is homeomorphic to $S^1$) on $\surf$ that does not intersect with $\cut$.
\end{itemize}
\item A grading $\Lambda$ on $\surfo$ is compatible with (/induced from)
the grading $\grad\in\coho{1}(\PP T\surf)$ if
$\Lambda(\alpha)=\left(\grad(\alpha),0\right)$,
for any simple closed curve $\alpha$ on $\surf$ that does not intersect with $\cut$.
\item A graded DMS $\surfo=(\surfo,\uc,\Lambda)$ associated to $\gms$ is
a DMS such that $(\Lambda,\cut)$ are compatible with $\grad$.
\end{itemize}
\end{definition}
Open curves in $\surfo$ still connect points in $\M$
while closed curves connect points in $\Tri$.
The notions of arcs and (full formal) open arc system are as before.
\begin{itemize}
\item Denote by $\ACC(\surfo)$ the set of admissible closed curves,
where a closed curve is admissible if
it does not cut out a once-decorated monogon by one of its self-intersections in $\surfo-\Tri$.
\item Denote by $\CA(\surfo)\subset\ACC(\surfo)$
the set of closed arcs on $\surfo$, which requires that
they connect different decorations.
\item The \emph{mapping class group} $\MCG_\bullet(\surfo)$
is the group of isotopy classes of diffeomorphisms of $\surfo$,
where all diffeomorphisms and isotopies are required to fix $\M$ and $\Tri$ setwise.
\item The kernel of the forgetful map $\MCG_\bullet(\surfo)\to\MCG_\bullet(\surf)$ is
the \emph{surface braid group} $\SBr(\surfo)$,
that is, the fundamental group of the configuration space of $|\Tri|$ points in
(the interior of) $\surf$, based at the set $\Tri$.
\item Each closed arc $\eta\in\CA(\surfo)$ induces a (positive) \emph{braid twist} $B_\eta\in\MCG_\bullet(\surfo)$.
The braid twist group $\BT(\surfo)$ is the subgroup of $\MCG_\bullet(\surfo)$
generated by $\{B_\eta\mid\eta\in\CA(\surfo)\}$.
Then $\BT(\surfo)\subset\SBr(\surfo)$, cf. \cite{QQ}.
\end{itemize}

%#ADD

We have the following characterization of cuts.

\begin{lemma}\label{lem:cuts}
The set $\Cut(\surfo)$ of cuts is a $\SBr(\surfo)$-torsor.
\end{lemma}
\begin{proof}
First, suppose that there is an element $b\in\SBr(\surfo)$ and a cut $\cut$
such that $b(\cut)=\cut$.
As $\SBr(\surfo)=\MCG_\bullet(\surfo)/\MCG_\bullet(\surf)$,
we can assume that $b$ acts as the identity on $\surfo-\cut$
and hence $b$ is the identity on $\surfo$.
Second, for any two cuts $\cut_1$ and $\cut_2$,
one can construct an element $b\in\SBr(\surfo)$, satisfying $b(\cut_1)=b(\cut_2)$, as follows:
\begin{itemize}
  \item Take a homeomorphism/diffeomorphism $b_1\colon\surfo\to\surfo$
  that acts as identity outside the neighbourhood of $\cut_1=\{c_i^1\}$
  and pulls decorations $\Tri=\{Z_i\}$ along $\{c_i^1\}$ to the neighbourhood of $\Y=\{Y_i\}$;
  \item Then take another homeomorphism $b_{21}\colon\surfo\to\surfo$
  that moves the decorations (around $\Y$) on the arcs $\cut_2=\{c_i^2\}$;
  \item Finally, take a homeomorphism $b_2\colon\surfo\to\surfo$
  that moves the decorations along $\{c_i^2\}$ back to $\Tri$.
\end{itemize}
The composition $b_2\circ b_{21}\circ b_1$ gives the required $b$.
This completes the proof.
\end{proof}

%=========================================================
\subsection{Topological log surfaces}
%=========================================================

To understand the grading of arcs/curves on $\surfo$, we introduce log surface.
In Section~\ref{sec:MC}, we will discuss an alternative construction of log surface.
Such a surface can be viewed as the $q$-deformation of $\surfo$.

\begin{definition}\label{con:top log}
The \emph{(topological) log surface} $\LS$ of $\surfo$ with respect to $\cut$ is constructed as follows:
\begin{itemize}
\item For every $m\in\ZZ$, take a copy of $\surfo$, denoted by $\surfo^{(m)}$ with the cut $\cut^{(m)}$.
    Cut it along $c_i^{(m)}\in\cut^{(m)}$ from the decorating point.
\item Gluing $c_{i+}^{(m)}$ with $c_{i-}^{m+1}$ for $m\in\ZZ$ and $c_i\in\cut$,
we obtain $\LS$.
\item $\LS$ inherits a ($\ZZ$-)grading $\Lambda_1\in\coho{1}(\mathbb{P}T(\LS\setminus\Tri),\ZZ)$ from $\Lambda$,
which is the projection of $\Lambda$ to the first $\ZZ$-coordinate on each sheet.
\item There is a deck transformation $q$ such that $q(\surfo^{(m)})=\surfo^{(m+1)}$,
which is also known as the Adams grading of $\LS$.
\item Denote by $\pi_{\Tri}\colon\LS\to\surfo$ the covering map.
For any (graded) curve on $\surfo$ (and is in a minimal position with respect to $\cut$),
there are $\ZZ$ lifts on $\LS$ via the deck transformation $q$.
Denote by $\CA(\LS)$ and $\wCA(\LS)$
the set of all lifts of $\CA(\surfo)$ and $\wCA(\surfo)$ respectively.
Similarly for $\ACC(\LS)$ and $\wACC(\LS)$.
\item The grading of $\LS$ inherits from the first grading $\Lambda_1$ of $\Lambda$ on $\surfo$,
which will be denoted by $\log\lambda$.
\end{itemize}
\end{definition}

The mapping class group $\MCG_\bullet(\surfo)$ acts canonically on $\LS$ via the action on each sheet.
Denote by $\MCG^\circ(\LS)$ the mapping class group of $\LS$,
which is generated by $q$ and homeomorphisms induced from $\MCG(\surfo)$.
Clearly, $q$ is in the center of $\MCG^\circ(\LS)$.

A cut $\cut$ is compatible with $\TT$ if they do not intersect.
From now on, we will fix the log surface $\LS$ with respect to an initial cut $\uc$
that is compatible with the initial arc system $\TT$ of $\gms$.
In fact, $\uc$ is uniquely determined by $\TT$.
Moreover,
the (open full formal) arc system $\TT$ is then also an arc system of $\surfo$.

The lifts of graded closed arcs in $\LS$ provide a topological model for $\DX$.
More precisely, we have the following.

\begin{theorem}\cite[Thm.~A]{IQZ}\label{thm:IQZ0}
There is a map
\begin{gather}\label{eq:X lift}
    \Xxx\colon \wACC(\LS)\to\Ind\DX.
\end{gather}
Then restricted to $\wCA(\LS)$, $\Xxx$ gives a bijection between
the set of graded closed arcs and the set of reachable spherical objects (denoted by $\Sph\DX$).
Let $\ST\DX$ be the spherical twists group,
the subgroup of $\Aut\DX$ generated by spherical twist $\Phi_{S}$ of $S\in\Sph\DX$.
Then $X$ induces an isomorphism
\begin{equation}\label{eq:BT=ST}
    \iota_\TT\colon\BT(\surfo)\cong\ST\DX,
\end{equation}
sending the braid twist $B_\eta$ of a closed arc $\eta$
to the spherical twist $\Phi_{\Xxx(\eta)}$.
\end{theorem}

%For the definition of spherical objects and functors, see Appendix~\ref{sec:D-CY}.
Furthermore, similar to \cite[Thm.~9.9]{BS} (cf. \cite[\S~4.4]{KQ2}),
we can consider a quotient group $\Aut^\circ\DX$ of certain subgroup of $\Aut\DX$,
that fits in the short exact sequence
\[
    0\to\ST\DX\to\Aut^\circ\DX\to\MCG(\surf)\to0.
\]

\begin{remark}\label{rem:n.d.}
The numerical data $\num(\LS)$ of $\LS$ is given by $\num(\gms)$ in Definition~\ref{def:numerical}.
\end{remark}

%=========================================================
\subsection{Topological immersion}
%=========================================================
\begin{figure}[hb]\centering
\begin{tikzpicture}[scale=.28,arrow/.style={->,>=stealth,thick}]
\draw[thick] (0,0) circle (10);
\draw[blue,dashed, \separated ,very thick, font=\scriptsize] (180-360/7*5:10)
    .. controls +(90:6) and +(160:6) .. (180-360/7*4:10);
\draw[blue,dashed, \separated ,very thick, font=\scriptsize] (180-360/7*5:10)
    edge [bend right=-35] (180-360/7*3:10)
    edge [bend right=-10] (180-360/7*2:10);
\draw[blue,dashed, \separated ,very thick, font=\scriptsize](180-360/7*2:10)
    .. controls +(-100:6) and +(-45:6) .. (180-360/7*1:10)
    .. controls +(-45:6) and +(-0:6) .. (180-360/7*0:10)
    .. controls +(-0:6) and +(45:6) .. (180-360/7*6:10);
\foreach \j in {1,...,7}
\draw(180-\j*360/7:10)node[cyan]{$\bullet$};

\draw[red,ultra thick, font=\scriptsize] (360/7*5:10)
    edge [bend right=5]node[above right]{ } (360/7*1:10)
    edge [bend right=15]node[right]{ } (360/7*2:10)
    edge [bend right=40]node[left]{ } (360/7*3:10)
    edge [bend right=25]node[left]{ } (360/7*4:10);
\draw[red,ultra  thick, font=\scriptsize] (360/7*1:10)
    to[bend left=-45] node[right]{ }(360/7*0:10)
    to[bend left=-45] node[right]{ }(-360/7:10);

\foreach \j in {1,...,7}
\draw(\j*360/7:10)node[white]{$\bullet$} node[red]{$\circ$};
\end{tikzpicture}
\quad
\begin{tikzpicture}[scale=0.28,arrow/.style={->,>=stealth,thick}]
\draw[thick] (0,0) circle (10);
\draw[blue,dashed, \separated ,very thick, font=\scriptsize] (180-360/7*5:10)
    .. controls +(90:6) and +(160:6) .. (180-360/7*4:10);
\draw[blue,dashed, \separated ,very thick, font=\scriptsize] (180-360/7*5:10)
    edge [bend right=-35] (180-360/7*3:10)
    edge [bend right=-10] (180-360/7*2:10);
\draw[blue,dashed, \separated ,very thick, font=\scriptsize](180-360/7*2:10)
    .. controls +(-100:6) and +(-45:6) .. (180-360/7*1:10)
    .. controls +(-45:6) and +(-0:6) .. (180-360/7*0:10)
    .. controls +(-0:6) and +(45:6) .. (180-360/7*6:10);
\foreach \j in {1,...,7}
\draw(180-\j*360/7:10)node[cyan]{$\bullet$};

\draw[red,ultra  thick, font=\scriptsize] (360/7*5:7)
    edge [bend right=5]node[right]{} (360/7*1:7)
    edge [bend right=10]node[right]{} (360/7*2:7)
    edge [bend right=15]node[left]{} (360/7*3:7)
    edge [bend right=25]node[left]{} (360/7*4:7);
\draw[red,ultra  thick, font=\scriptsize] (360/7*1:7)
    to[bend left=-25] node[right]{}(360/7*0:7)
    to[bend left=-25] node[right]{}(-360/7:7);

\foreach \j in {1,...,7}
\draw[ultra thick, green](\j*360/7:10)to(\j*360/7:7)node[white]{$\bullet$} node[red]{$\circ$};
\end{tikzpicture}
\caption{Constructing DMS $\surfo$ via pulling closed marked points inside}\label{fig:-cut}
\end{figure}
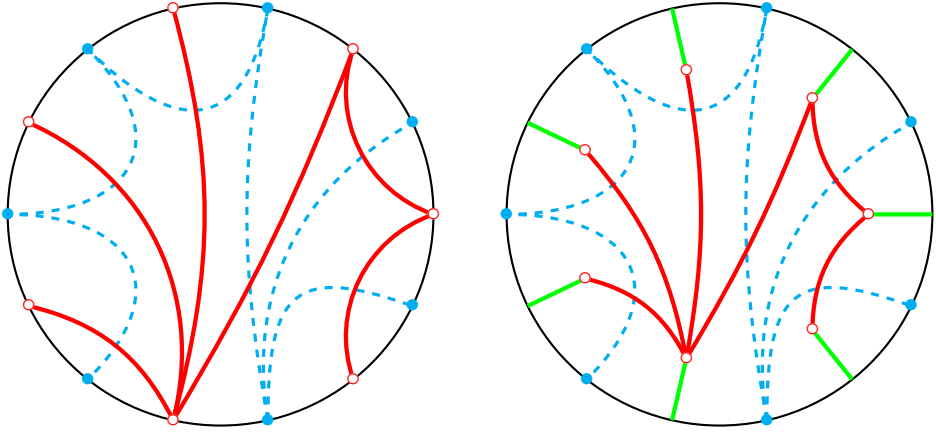

The DMS $\surfo$ can be thought as obtained from $\surf$ by
pulling the closed marked points along $\uc$ that become decorations, cf. Figure~\ref{fig:-cut}.
This operation induces a map
\begin{gather}\label{eq:imc}
    \imc\colon\wCA(\gms)\to  \wCA(\surfo^{(0)})\subset\wCA(\LS)
\end{gather}
in \cite[(6.2)]{IQZ}, cf. \cite[Fig.~24]{IQZ},
that takes a graded closed arc in $\gms$ to
a graded closed arc in the 0-sheet $\surfo^{(0)}$ of $\LS$.

Moreover, we have the following topological realization of Lagrangian immersion.
\begin{theorem}\cite[Thm.~6.6]{IQZ}\label{thm:imc}
The following commutative diagram commutes
\begin{gather}\label{eq:L=im}
\xymatrix@C=3pc{
    \wCA(\gms) \ar[r]^{\imc} \ar[d]_{\Xinf} & \wCA(\LS)
        \ar[d]^{\Xxx} \\
    \Ind\DT \ar[r]^{ \hh{L}_\TT } & \Ind\DX.
}\end{gather}
\end{theorem}

Also notice that the full formal closed arc system $\TT^*$ on $\gms$, dual to $\TT$,
becomes the closed arc system $$\TT^*_\Tri\colon=\imc(\TT^*)$$
on $\surfo$, still dual to $\TT$ (regarded on $\surfo$).

\begin{remark}\label{rem:Ext-Q}
Similar to Remark~\ref{rem:Ext-q},
graph duality between $\TT^*_\Tri$ and $\TT$ corresponds to the simple-projective duality.
More precisely,
$\widetilde{Q}_\TT$ can be obtained from the Ext-quiver
$Q^{\Ext}(\GAX)$ by the degree changing: $d\mapsto 1-d$.
And $Q^{\Ext}(\GAX)$ can be identified with
the intersection quiver of $\TT^*_\Tri$.

Furthermore,
the Ext-quiver $Q^{\Ext}(\GAX)$ is the Calabi-Yau-$\XX$ double,
in the sense of \cite[Def.~6.2]{KQ1}, of the Ext-quiver $Q^{\Ext}(\ha^0_\TT)$.
\end{remark}

%=========================================================
\part{Geometric Part}\label{part:Geo}
%=========================================================
%=========================================================
\section{Quadratic differentials on Riemann surfaces}\label{sec:QDR}
In this section,
we review the theory of Bridgeland-Smith \cite{BS} (cf. \cite{KQ2}) and Haiden-Katzarkov-Kontsevich \cite{HKK},
that the spaces of stability conditions on Fukaya type categories associated to Riemann surfaces
can be realized by the moduli spaces of framed quadratic differentials (of certain type) on the Riemann surfaces.
%=========================================================
\subsection{Preliminaries}\label{sec:HKK}
%=========================================================
Let $\rs$ be a Riemann surface and $\omega_\rs$
its holomorphic cotangent bundle.
A \emph{meromorphic quadratic differential} $\phi$ on $\rs$ is a meromorphic section
of the line bundle $\omega_{\rs}^{2}$.
In terms of a local coordinate $z$ on $\rs$, such a $\phi$  can be written as $\phi(z)=g(z)\, \diff z^{\otimes 2}$,
where $g(z)$ is a meromorphic function.
Denote by $\Zer_j(\phi)$ the set of zeroes of $\phi$ of order $j$
and $\Zer(\phi)=\bigsqcup_{j>0}\Zer_j(\phi)$.
Similar for the set $\Pol(\phi)=\bigsqcup_{j>0}\Pol_j(\phi)$ of poles.
Set $\Crit(\phi):=\Zer(\phi)\cup\Pol(\phi)$.

At a point of $\rs^\circ=\rs \setminus \Crit(\phi)$,
there is a  distinguished local coordinate $\omega$,
uniquely defined up to transformations of the form $\omega \mapsto \pm\, \omega+\operatorname{const}$,
with respect to which one has $\phi(\omega)=\diff \omega^{\otimes 2}$.
In terms of a local coordinate $z$, we have $w=\int \sqrt{g(z)}\diff z$.
A quadratic differential $\phi$ on $\rs$ determines the \emph{$\phi$-metric} on $\surp$,
defined locally by pulling back the Euclidean metric on $\kong{C}$ via the distinguished coordinate $\omega$.
Therefore, there are geodesics on $\surp$ and each geodesic has a constant phase,
with respect to $\omega$.

Next, we summarize the global structure of the horizontal trajectories and horizontal strip decompositions
(cf. \cite[\S~3]{BS} and \cite[\S~2]{HKK}).

\begin{definition}
A \emph{horizontal trajectory} of $\phi$ on $\surp$
is a maximal horizontal geodesic $\gamma\colon(0,1)\to\surp$
with respect to the $\phi$-metric.
The horizontal trajectories of a meromorphic quadratic differential $\phi$
provide the \emph{horizontal foliation} on $\rs$.
\end{definition}

The following are types of trajectories of a quadratic differential $\phi$
we are interested in:
\begin{itemize}
\item \emph{saddle trajectories} with both ends in $\Zer(\phi)$;
\item \emph{separating trajectories} with one end in $\Zer(\phi)$ and the other
in $\Pol(\phi)$;
\item \emph{generic trajectories} with both ends in $\Pol(\phi)$;
\item \emph{closed trajectories} that are simple closed curves;
\item \emph{recurrent trajectories} that are recurrent in at least one direction.
\end{itemize}

\begin{definition}
By removing all horizontal separating trajectories (which are finitely many) from $\surp$,
the remaining open surface splits as a disjoint union of connected components.
Each component is one of the following types:
\begin{itemize}
\item a \emph{half-plane}, i.e., it is isomorphic to
$\{z\in \kong{C}\colon\operatorname{Im}(z)>0\}$
equipped with the differential $\diff z^{\otimes 2}$.
It is swept out by generic trajectories which connect a fixed pole to itself.
\item
a \emph{horizontal strip}, i.e., it is isomorphic to
$\{z\in \kong{C}\colon a<\operatorname{Im}(z)<b\}$
equipped with the differential $\diff z^{\otimes 2}$.
It is swept out by generic trajectories which connect two (not necessarily distinct) poles.
\item a \emph{ring domain}, i.e., it is isomorphic to
$\{z\in\kong{C}\colon a<|z|<b\}\subset\kong{C}^*$
equipped with the differential $r\diff z^{\otimes 2}/z^2$ for some $r\in\kong{R}_{<0}$.
It is swept out by closed trajectories.
\item a \emph{spiral domain},
defined to be the interior of the closure of a recurrent trajectory.
\end{itemize}
We call this union the \emph{horizontal strip decomposition} of $\rs$ with respect to $\phi$.
\end{definition}

A quadratic differential $\phi$ on $\rs$ is called \emph{saddle-free}, if it has no saddle trajectory.
Note the following:
\begin{itemize}
\item In each horizontal strip, the trajectories are isotopy to each other.
\item If $\phi$ is saddle-free, then $\Pol(\phi)$ must be nonempty and,
by \cite[Lemma~3.1]{BS}, $\phi$ has no closed or recurrent trajectories.
Thus, in the horizontal strip decomposition of $\rs$ with respect to $\phi$,
there are only half-planes and horizontal strips.
\item If $\phi$ is saddle-free, then the boundary of any strips consists of separating trajectories.
\item If $\phi$ is saddle-free, then there is a unique geodesic in each horizontal strip,
the \emph{saddle connection}, connecting the two zeroes on its boundary.
\end{itemize}

%=========================================================
\subsection{Classical type of singularities in the Calabi-Yau-3 setting}\label{sec:GMN}
%=========================================================
\begin{definition}\label{def:cqd}
A \emph{classical type of singularity} $p$ of a meromorphic quadratic differential has the local form
\begin{gather}\label{eq:cla}
    z^k g(z)\diff z^{\otimes2},
\end{gather}
where $g(z)$ is some non-zero holomorphic function and
\begin{itemize}
  \item if $k>0$, then $p$ is a \emph{zero} of order $k$;
  \item if $k<0$, then $p$ is a \emph{pole} of order $k$.
\end{itemize}
\end{definition}

The foliation near a singularity $p$ is shown in Figure~\ref{fig:foli}, where
$p$ is a zero of order $1$ and $2$ in the left two pictures and
$p$ is a pole of order $3$ and $4$ in the right two pictures.

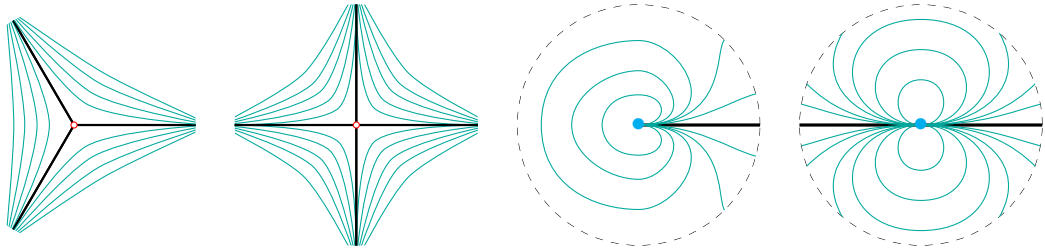
\begin{figure}[ht]\centering
\begin{tikzpicture}[scale=.4] \clip(0,0) circle (4);
\foreach \k in {1,2,0}
{  \path (120*\k:4.5) coordinate (v2)
          (120*\k+120:4.5) coordinate (v1)
          (120*\k+60:2.25) coordinate (v3)
          (0,0) coordinate (v4);
  \draw[thick](v2)to(v4)to(v1);
  \foreach \j in {.36,.54,.72,.88}
    {
      \path (v4)--(v3) coordinate[pos=\j] (m0);
      \draw[Emerald] plot [smooth,tension=.5] coordinates {(v1)(m0)(v2)};}}
\draw[red,fill=white](0,0)circle(.1);
\end{tikzpicture}
\quad
\begin{tikzpicture}[scale=.4] \clip(0,0) circle (4);
\foreach \k in {0,1,2,3}
{  \path (90*\k:4.5) coordinate (v2)
          (90*\k+90:4.5) coordinate (v1)
          (90*\k+45:2.25) coordinate (v3)
          (0,0) coordinate (v4);
  \foreach \j in {.36,.54,.72,.88,1.08}
    {\path (v4)--(v3) coordinate[pos=\j] (m0);
     \draw[Emerald] plot [smooth,tension=.5] coordinates {(v1)(m0)(v2)};}
  \draw[thick](v2)to(v4)to(v1);}
\draw[red,fill=white](0,0)circle(.1);
\end{tikzpicture}
\quad
\begin{tikzpicture}[scale=.4,rotate=0]\clip(0,0) circle (4);
\draw[dashed, thin](0,0)circle(4);
\draw[very thick](0,0)to(0:4)(-90:4.5);
\draw[Emerald] (0,0)
    .. controls +(0:2) and +(195:1) ..(15:4);
\draw[Emerald] (0,0)
    .. controls +(0:2) and +(165:1) ..(-15:4);
\draw[Emerald] (0,0)
    .. controls +(0:3) and +(225:.5) ..(45:4);
\draw[Emerald] (0,0)
    .. controls +(0:3) and +(125:.5) ..(-45:4);

\draw[Emerald] (0,0)
    .. controls +(0:3) and +(0:1.5) ..(90:2.8)
    .. controls +(180:1.5) and +(90:2) ..(180:3.2)
    .. controls +(-90:2) and +(180:1.5) ..(-90:2.8)
    .. controls +(0:1.5) and +(0:3) .. (0,0);
\draw[Emerald] (0,0)
    .. controls +(0:2) and +(0:1) ..(90:1.8)
    .. controls +(180:1) and +(90:1) ..(180:2.2)
    .. controls +(-90:1) and +(180:1) ..(-90:1.8)
    .. controls +(0:1) and +(0:2) .. (0,0);
\draw[Emerald] (0,0)
    .. controls +(0:1) and +(0:1) ..(90:1)
    .. controls +(180:1) and +(90:.3) ..(180:1.2)
    .. controls +(-90:.3) and +(180:1) ..(-90:1)
    .. controls +(0:1) and +(0:1) .. (0,0);
\draw[\separated](0,0)node{$\bullet$};
\end{tikzpicture}
\quad
\begin{tikzpicture}[scale=.4,rotate=0]\clip(0,0) circle (4);
\draw[dashed, thin](0,0)circle(4);
\draw[very thick](180:4)to(0:4)(-90:4.5);
\draw[Emerald] (0,0)
    .. controls +(0:2) and +(195+30:1) ..(20:4);
\draw[Emerald] (0,0)
    .. controls +(0:2) and +(165-30:1) ..(-20:4);

\draw[Emerald] (0,0)
    .. controls +(0:2) and +(195+105:2) ..(50:4);
\draw[Emerald] (0,0)
    .. controls +(0:2) and +(165-105:2) ..(-50:4);

\draw[Emerald] (0,0)
    .. controls +(180:2) and +(180-195-105:2) ..(180-50:4);
\draw[Emerald] (0,0)
    .. controls +(180:2) and +(180-165+105:2) ..(180--50:4);

\draw[Emerald] (0,0)
    .. controls +(0:2) and +(195:1) ..(10:4);
\draw[Emerald] (0,0)
    .. controls +(0:2) and +(165:1) ..(-10:4);

\draw[Emerald] (0,0)
    .. controls +(180:2) and +(15:1) ..(195-5:4);
\draw[Emerald] (0,0)
    .. controls +(180:2) and +(-15:1) ..(165+5:4);

\draw[Emerald] (0,0)
    .. controls +(180:2) and +(45:1) ..(200:4);
\draw[Emerald] (0,0)
    .. controls +(180:2) and +(-45:1) ..(160:4);

\draw[Emerald] (0,0)
    .. controls +(0:3) and +(0:3) ..(90:3.5)
    .. controls +(180:3) and +(180:3) .. (0,0);
\draw[Emerald] (0,0)
    .. controls +(0:2) and +(0:2) ..(90:2.5)
    .. controls +(180:2) and +(180:2) .. (0,0);
\draw[Emerald] (0,0)
    .. controls +(0:1) and +(0:1) ..(90:1.5)
    .. controls +(180:1) and +(180:1) .. (0,0);
\draw[Emerald] (0,0)
    .. controls +(0:3) and +(0:3) ..(-90:3.5)
    .. controls +(180:3) and +(180:3) .. (0,0);
\draw[Emerald] (0,0)
    .. controls +(0:2) and +(0:2) ..(-90:2.5)
    .. controls +(180:2) and +(180:2) .. (0,0);
\draw[Emerald] (0,0)
    .. controls +(0:1) and +(0:1) ..(-90:1.5)
    .. controls +(180:1) and +(180:1) .. (0,0);

\draw[\separated](0,0)node{$\bullet$};
\end{tikzpicture}
\caption{Foliations near classical type of zeroes/poles}\label{fig:foli}
\end{figure}

Next, we briefly review Bridgeland-Smith's result (cf. \cite{BS,KQ2}).
For simplicity, we will not consider the case when the pole has order 1 or 2.
\begin{definition}\label{def:GMN}
A \emph{GMN differential} on $\rs$ is a meromorphic differential
with classical type of singularities satisfying:
\begin{itemize}
\item all the zeroes are simple zeroes (i.e., of order $1$);
\item the order of any pole is at least 3.
\end{itemize}

The \emph{real (oriented) blow-up} of $(\rs,\phi)$ is a differentiable surface $\rs^\phi$ obtained from the underlying differentiable surface by replacing a pole $P\in\Pol(\phi)$
with order $\ord_\phi(P)>3$ by a boundary $\partial_P$,
where the points on the boundary correspond to the real tangent directions at $P$.
We mark the points on $\partial_P$ that correspond to the distinguished tangent directions.
Thus there are $\ord_\phi(P)-2$ (open) marked points on $\partial_P$.
We denote by $\rs^\phi$ the \emph{(open) marked surface} associated to $(\rs,\phi)$,
that consists of the surface $\rs^\phi$ together with those marked points.
The \emph{decorated marked surface} $\rs^\phi_{\Zer(\phi)}$ associated to $(\rs,\phi)$ is
obtained from $\rs^\phi$ by decorating the set $\Zer(\phi)$ of zeroes of $\phi$.
\end{definition}

\begin{example}\label{ex:a2}
Definition~\ref{def:GMN} fits perfectly with triangulated surfaces in the theory of cluster algebras \cite{FST}.
For instance, the foliation of a quadratic differential on $\PP^1$ with type $(1,1,1,-7)$,
i.e., it has three simple zeroes and one order $7$ pole,
is shown in Figure~\ref{fig:Quad A2}.
The left picture is the real blow-up part and the right picture is the neighbourhood of the pole.
Note that, in the pictures,
\begin{itemize}
\item the blue vertices are poles or marked points;
\item the red vertices are simple zeroes;
\item the green arcs are geodesics;
\item the black arcs are separating trajectories;
\item the red arcs are the saddle connections in the horizontal strips.
\end{itemize}
Observe that the isotopy classes of separating trajectories
in the left picture of Figure~\ref{fig:Quad A2}
form a \emph{WKB-triangulation} of the marked surface (a pentagon in this case).
In general, such a WKB-triangulation will be formed
provided the GMN differential is saddle-free.
\end{example}

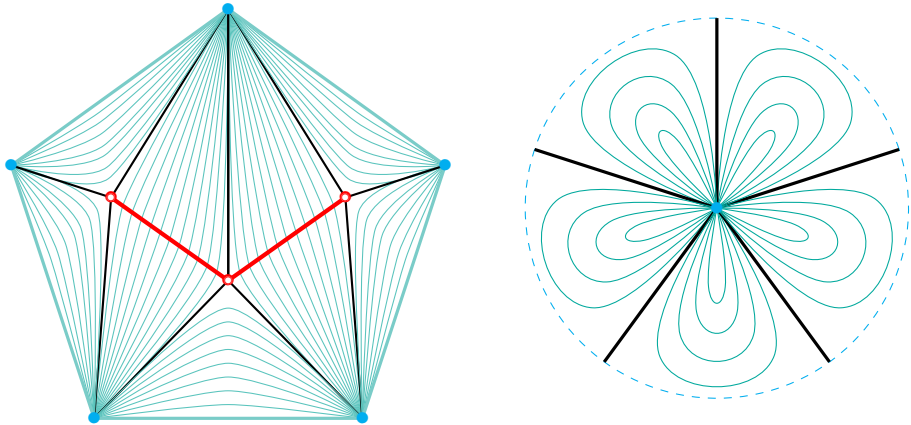
\begin{figure}[ht]\centering
\begin{tikzpicture}[scale=.6]
    \path (18+72:5) coordinate (v2)
          (18+72*3:5) coordinate (v1)
          (18+72*2:2.7) coordinate (v3)
          (0,-1) coordinate (v4);
  \foreach \j in {.1, .18, .26, .34, .42, .5,.58, .66, .74, .82, .9}
    {
      \path (v3)--(v4) coordinate[pos=\j] (m0);
      \draw[Emerald!60, thin] plot [smooth,tension=.3] coordinates {(v1)(m0)(v2)};
    }
\draw[thick](v4)to(v1)to(v3)to(v2)to(v4);

    \path (18+72:5) coordinate (v2)
          (18+72*4:5) coordinate (v1)
          (18+72*0:2.7) coordinate (v3);
  \foreach \j in {.1, .18, .26, .34, .42, .5,.58, .66, .74, .82, .9}
    {
      \path (v3)--(v4) coordinate[pos=\j] (m0);
      \draw[Emerald!60, thin] plot [smooth,tension=.3] coordinates {(v1)(m0)(v2)};
    }
\draw[thick](v4)to(v1)to(v3)to(v2)to(v4);
%==
    \path (18+72:5) coordinate (v2)
          (18+72*2:5) coordinate (v1)
          (18+72*2:2.7) coordinate (v3);
    \path (v1)--(v2) coordinate[pos=.4] (v4);
  \foreach \j in {.2,.32,.45,.55,.68,.8, .9}
    {
      \path (v3)--(v4) coordinate[pos=\j] (m0);
      \draw[Emerald!60, thin] plot [smooth,tension=.4] coordinates {(v1)(m0)(v2)};
    }
%==
    \path (18+72*3:5) coordinate (v2);
    \path (v1)--(v2) coordinate[pos=.4] (v4);
  \foreach \j in {.2,.32,.45,.55,.68,.8, .9}
    {
      \path (v3)--(v4) coordinate[pos=\j] (m0);
      \draw[Emerald!60, thin] plot [smooth,tension=.4] coordinates {(v1)(m0)(v2)};
    }
%==
    \path (18+72:5) coordinate (v2)
          (18+72*0:5) coordinate (v1)
          (18+72*0:2.7) coordinate (v3);
    \path (v1)--(v2) coordinate[pos=.4] (v4);
  \foreach \j in {.2,.32,.45,.55,.68,.8, .9}
    {
      \path (v3)--(v4) coordinate[pos=\j] (m0);
      \draw[Emerald!60, thin] plot [smooth,tension=.4] coordinates {(v1)(m0)(v2)};
    }
%==
    \path (18+72*4:5) coordinate (v2);
    \path (v1)--(v2) coordinate[pos=.4] (v4);
  \foreach \j in {.2,.32,.45,.55,.68,.8, .9}
    {
      \path (v3)--(v4) coordinate[pos=\j] (m0);
      \draw[Emerald!60, thin] plot [smooth,tension=.4] coordinates {(v1)(m0)(v2)};
    }
%==
    \path (18+72*3:5) coordinate (v2)
          (18+72*4:5) coordinate (v1)
          (0,-1) coordinate (v3);
    \path (v1)--(v2) coordinate[pos=.5] (v4);
  \foreach \j in {.2,.3,.4,.5,.6,.7,.8,.9}
    {
      \path (v3)--(v4) coordinate[pos=\j] (m0);
      \draw[Emerald!60, thin] plot [smooth,tension=.4] coordinates {(v1)(m0)(v2)};
    }
%==
\draw[thick](18+72*2:5)to(18+72*2:2.7)(18:5)to(18:2.7);
\foreach \j in {1,2,3,4,5}
    {\draw[Emerald!50,very thick]
     (18+72*\j:5)to(90+72*\j:5);}
\foreach \j in {1,2,3,4,5}
    {\draw[\separated]
        (18+72*\j:5)node{$\bullet$};}
\draw[red,ultra thick]
  (18+72*2:2.7)node{$\bullet$}node[white]{\tiny{$\bullet$}}node{\tiny{$\circ$}}to
  (0,-1)node{$\bullet$}node[white]{\tiny{$\bullet$}}node{\tiny{$\circ$}}to
  (18+72*0:2.7)node{$\bullet$}node[white]{\tiny{$\bullet$}}node{\tiny{$\circ$}};
\end{tikzpicture}
\quad
\begin{tikzpicture}[scale=.7,rotate=18]
\draw[dashed, \separated](0,0)circle(3.6)(-90:4.5);
\foreach \j in {0,1,2,3,4}
{
  \draw[very thick](72*\j+0:3.6)to(0,0);
  \draw[Emerald] plot [smooth,tension=1] coordinates
    {(0,0)(15+72*\j:3)(72-15+72*\j:3)(0,0)};
  \draw[Emerald] plot [smooth,tension=1] coordinates
    {(0,0)(20+72*\j:2.5)(72-20+72*\j:2.5)(0,0)};
  \draw[Emerald] plot [smooth,tension=1] coordinates
    {(0,0)(25+72*\j:2)(72-25+72*\j:2)(0,0)};
  \draw[Emerald] plot [smooth,tension=1] coordinates
    {(0,0)(30+72*\j:1.5)(72-30+72*\j:1.5)(0,0)};
}
\draw[\separated](0,0)node{$\bullet$};
\end{tikzpicture}
\caption{The foliation of a GMN differential on $\PP^1$ for the $A_2$ case}
\label{fig:Quad A2}
\end{figure}

Let $\surfo^3$ be a decorated marked surface with
$\M$ the set of open marked points and $\Tri$ the set of decorations in the interior.
Its numerical data satisfies (cf. \cite[Def.~3.1]{QQ})
\begin{gather}\label{eq:Delta}
    |\Tri|=4g+2|\partial\surf|+|M|-4 \quad(\ge1),
\end{gather}
where $g$ is the genus of $\surf$.

\begin{definition}\label{def:SFquad}
A \emph{$\surfo^3$-framed quadratic differential} $(\rs,\phi,\psi)$
consists of a GMN differential $\phi$ on a Riemann surface $\rs$ together
with a diffeomorphism $\psi\colon\surfo^3 \to\rs^\phi_{\Zer(\phi)}$
preserving the marked/decorating points setwise.
Two $\surfo^3$-framed quadratic differentials $(\rs_1,\phi_1,\psi_1)$ and $(\rs_2,\phi_2,\psi_2)$
are equivalent if
\begin{itemize}
\item there exists a biholomorphism $f\colon\rs_1\to\rs_2$
such that $f^*(\phi_2)=\phi_1$ and,
\item $\psi_2^{-1}\circ f_*\circ\psi_1\in\Homeo_0(\surfo^3)$,
where $f_*\colon(\rs_1)^{\phi_1}_{\Zer(\phi_1)}\to(\rs_2)^{\phi_2}_{\Zer(\phi_2)}$ is the induced homeomorphism.
Recall that $\Homeo_0$ consists of isotopy classes of homeomorphisms that are isotopy to identity.
\end{itemize}
%\begin{enumerate}
%\item $f^*(\phi_2)=\phi_1$;
%\item $\psi_2^{-1}\circ f_*\circ\psi_1\in\Diff_0(\surf)$,
%where $f_*\colon\rs_1^{\phi_1}\to\rs_2^{\phi_2}$ is the induced diffeomorphism;
%\end{enumerate}
\end{definition}

We denote by $\FQuad_3(\surfo^3)$ the moduli space of $\surfo^3$-framed quadratic differentials.
For $\surfo^3$, there is a canonical associated Calabi-Yau-3 category $\D_3(\surfo^3)$,
which can be constructed as follows.
\begin{itemize}
  \item Take any triangulation $T$ of $\surf=\rs^\phi$
  (or a WKB-triangulation mentioned above)
    and consider the corresponding quiver with potential $(Q_T,W_T)$, where
    \begin{itemize}
\item the vertices of $Q_T$ correspond to the open arcs in $T$;
\item the arrows of $Q_T$ correspond to the angles of triangles in $T$;
\item the terms/cycles of $W_T$ correspond to triangles in $T$.
\end{itemize}
  \item Let $\Gamma_T$ be the Ginzburg dg algebra of $(Q_T, W_T)$
  and then $\D_3(\surfo^3)$ is the finite dimensional derived category of $\Gamma_T$.
\end{itemize}

Recall that a triangulation is a maximal compatible collection of open arcs on $\surf$.
Moreover, \cite[Thm.~A]{BQZ} shows that $\D_3(\surfo^3)$ is unique (up to natural transformation) in a suitable sense.
Furthermore, $\D_3(\surfo^3)$ can be embedded into the derived Fukaya category of a symplectic 3-fold,
constructed from fibrations of $\surf$ via a quadratic differential \cite{S}.

Let $\widetilde{\surfo^3}$ be the branched double cover (\emph{spectral cover}) of $\surfo^3$, branching at $\Tri$.
We finish this subsection with the following result, whose original version/idea is due to \cite{BS},
while the quoted upgraded version is taken from \cite{KQ2}.

\begin{theorem}\cite{BS,KQ2}\label{thm:BSKQ}
There is an isomorphism
\begin{gather}\label{eq:chi_3}
    \chi_3\colon \widehat{\Ho{}}(\surfo^3)\to\Grot(\D_3(\surfo^3)),
\end{gather}
where the hat homology $\widehat{\Ho{}}(\surfo^3)=\Ho{1}(\widetilde{\surfo^3};\ZZ)^-$
is the anti-invariant part of $\Ho{1}(\widetilde{\surfo^3};\ZZ)$ with respect to the covering involution
(cf. Definition~\ref{def:hhg} for the precise definition in the general setting).
Moreover, there is an isomorphism $\Xi_3$ of complex manifolds that fits into the commutative diagram
\begin{gather}\label{eq:i 3}
\xymatrix@C=5pc{
    \FQuad_3^\circ(\surfo^3)\ar[r]^{\Xi_3} \ar[d]_{\Pi_3} &
        \Stap\D_3(\surfo^3) \ar[d]^{\hh{Z}_3}\\ %{(Z,\hh{P}) \mapsto Z} \\
    \Hom(  \widehat{\Ho{}}(\surfo^3),\bC) \ar[r]^{ \chi_3} &
        \Hom(  \Grot(\D_3(\surfo^3)),\bC)
,}\end{gather}
%\[\FQuad_3^\T(\surfo^3)\cong\Stap\D_3(\surfo^3),\]
such that the period map $\Pi_3\colon \phi\mapsto\int\sqrt{\phi}$ becomes the map $\hh{Z}_3$ (i.e., central charge).
More precisely, \eqref{eq:formula} holds for
\begin{gather}\label{eq:X3}
    X=\xs\colon \CA(\surfo^3)\to\Sph\D_3(\surfo^3)/[1]
\end{gather}
in \cite[Thm.~6.6]{QQ}, where $\CA(\surfo^3)$ the set of closed arcs in $\surfo^3$
and $\Sph\D_3(\surfo^3)$ the set of reachable spherical objects in $\D_3(\surfo^3)$.

Furthermore, for $\sigma=\Xi_3(\phi)$, the map $\xs$ induces one-one correspondence between
saddle connections of $\phi$ on $\surfo^3$ and $\sigma$-semistable objects in $\D_3(\surfo^3)$.
\end{theorem}
Here, $^\circ$ in $\FQuad_3^\circ$ and $\Stap$ means taking the \emph{principal component}
(cf. \cite[\S~4]{KQ2}).

We also have the moduli space of $\surf^3$-framed quadratic differentials such that
\begin{gather}\label{eq:surf3}
    \FQuad(\surf^3)=\FQuad^\circ(\surfo^3)/\BT(\surfo^3)\cong
    \Stab\D_3(\surfo^3)/\ST\D_3(\surfo^3),
\end{gather}
noticing that \eqref{eq:X3} induces the isomorphism between $\BT(\surfo^3)$ and $\ST\D_3(\surfo^3)$,
cf. \cite[Thm.1]{QQ}, which is the Calabi-Yau-3 version of \eqref{eq:BT=ST}.

%=========================================================
\subsection{Exponential type singularities in the Calabi-Yau-$\infty$ setting}\label{sec:Exp}
%=========================================================
\begin{definition}
\cite{HKK}
The \emph{exponential type singularity} $p$ of index $(k,l)\in\ZZ_{>0}\times\ZZ$
of a quadratic differential $\phi$ on a Riemann surface $\rs$
is a singularity $p$ with local coordinate as
\begin{gather}\label{eq:exp}
    e^{z^{-k}}z^{-l} g(z)\diff z^{\otimes2}
\end{gather}
for some non-zero holomorphic function $g(z)$.

A \emph{flat surface} $(\rs,\phi)$ is a Riemann surface $\rs$ together with quadratic differential $\phi$
whose singularities are all of exponential type as above.

\end{definition}
By \cite[Prop.~2.5]{HKK}, for a singularity $p$ in $\rsp$ with local coordinate \eqref{eq:exp},
there are exactly $k$ additional points, with respect to the $\phi$-metric completion in the neighborhood of $p$;
all of which are $\infty$-angle singularities, also known as the conical singularity with infinity angle
described as below.

\begin{definition}\label{def:conical}
Consider the spaces
\begin{gather}\label{eq:C_m}
    C_N \colon=\ZZ_N\times\RR\times\RR_{\geq0}/\sim, \qquad (k,x,0)\sim (k+1,-x,0)
        \text{ for } x\le 0
\end{gather}
for $N\in\ZZ_{>0}$ and
\begin{gather}\label{eq:C_infty}
    C_\infty\colon=\ZZ\times\RR\times\RR_{\geq0}/\sim, \qquad (k,x,0)\sim (k+1,-x,0)
        \text{ for } x\le 0.
\end{gather}
The origin $0=(0,0,0)\sim(k,0,0)$ is a distinguished point in these spaces.
The standard flat structure of $\RR^2$ makes $C_?\setminus\{0\}$ a smooth flat surface
with metric completion $C_?$.
The origin is a \emph{conical singularity with $N\pi$ angle} for $C_N$, $N\ne2$
and is a \emph{conical singularity with infinity angle} for $C_\infty$.
\end{definition}

Note that the conical singularities with $N\pi$ angle are exactly classical type zeroes of order $N-2$
mentioned above, for $N\ge3$.
For GMN differentials (i.e., in BS' setting), we always have $N=3$.
We will discuss the $N\ge3$ case in Section~\ref{sec:BS-N}.
Furthermore, the local behavior in the origin of $C_\infty$ is similar to zeroes instead of poles
since a trajectory attending to such points has finite length.

Let $\rs_\sg$ be the set of singularities of the flat surface $\rsp$
(where the smooth part of $\rsp$ is just $\rs$).
We have the following notions.
\begin{itemize}
\item A \emph{saddle connection} is a maximal geodesic converging towards closed points in $\rs_\sg$
(i.e., zeroes or points in the metric completion)
in both directions. Note that saddle trajectories are saddle connections which are horizonal.
Also, the endpoints are not required to be distinct.
%\item A \emph{geodesic arc system} of $\phi$ is a collection of saddle connections
%which intersect at most in the endpoints.
\item The \emph{core} $\Core(\phi)$ is the convex hull of $\rs_\sg$
%any maximal geodesic arc system $\hh{A}$ of $\phi$, which is independent of $\hh{A}$
by \cite[Prop.~2.2]{HKK}, which contains all saddle connections.
Moreover, $\Core(\phi)$ is a deformation retract of $\rsp$ provided that
each point on $\rsp$ has finite distance to $\Core(\phi)$ (with respect to the $\phi$-metric),
cf. \cite[Prop.~2.3]{HKK}.
\end{itemize}

%We will also use $\rs^\phi$ to denote the oriented real blow-up of $\rs$
%at those singularities, where we only apply the blow-up at points that behave like poles
%in the sense that the trajectories approaching those points have infinity length.
%The rest of the singularities will be marked on $\rs^\phi$.
%Comparing to the convention in \cite{HKK},
%their marked boundary components are marked points in our setting.
%Also note that $\rs^\phi$ is automatically equipped with a grading induced by $\phi$.

\begin{definition}
The \emph{graded marked surface} $\rs^\phi$ associated to a flat surface $(\rs,\phi)$ is
the \emph{real blow-up} of $\rs$ at all the singularities of $\phi$. So
\begin{itemize}
\item Each boundary component of $\rsp$ corresponds to the blow-up of a singularity $p$ of exponential type $(k,l)$.
\item The closed marked points correspond to the additional $k$ points
with respect to $\phi$-metric completion mentioned above.
\item The open marked points correspond to the directions between closed marked points
(hence are dual to closed marked points on the boundaries).
\item The grading is given by the horizontal foliation.
\end{itemize}
\end{definition}

\begin{remark}
In fact, any graded marked surface $(\surf,\M,\Y,\grad)$ in Section~\ref{sec:TFuk} is
the real blow-up of some flat surfaces with conical singularities of infinity angle.
\end{remark}

\begin{definition}\label{def:nn}
Let $\gms$ be a graded marked surface with the numerical data $$\num(\gms)=(g,b;\uk,\ul;\LP_g).$$
A \emph{$\gms$-framed quadratic differential} $(\rs,\phi,h)$ on $\gms$ consists of
a quadratic differential $\phi$ on some Riemann surface $\rs$
together with a diffeomorphism $h\colon\gms\to \rsp$ such that
the numerical data of the grading $h^*(\phi)$ on $\gms$
  (the pull-back of the $\phi$-induced grading on $\rs$) equals $\num(\gms)$.
Two $\gms$-framed quadratic differentials $(\rs_i,\phi_i,h_i)$ are equivalent if
\begin{itemize}
\item there exists a biholomorphism $f\colon\rs_1\to\rs_2$
such that $f^*(\phi_2)=\phi_1$;
\item $h_2^{-1}\circ f_*\circ h_1\in\Diff_0(\gms)$,
where $f_*\colon\rs_1^{\phi_1}\to\rs_2^{\phi_2}$ is the induced diffeomorphism.
\end{itemize}
Denote by $\FQuad_\infty(\gms)$ the space of $\gms$-framed quadratic differentials on $\gms$.
%The exponential type singularities will be discussed a bit in the next section.
For simplicity, we will only mention a quadratic differential $\phi$ on $\gms$
(omitting $\rs$ and $h$).
\end{definition}

Now we recall the main result in \cite{HKK}, where the surjection of the $\Xi_\infty$ is shown by \cite{T}.

\begin{theorem}\cite{HKK,T}\label{thm:HKK}
There is the following identification
\begin{gather}\label{eq:chi_infty}
    \chi_\infty\colon \Ho{1}(\gms,\partial\gms;\ZZ_\Sp)\to\Grot(\DI),
\end{gather}
where $\ZZ_\Sp=\ZZ\otimes_{\ZZ_2}\Sp$ for $\Sp$ being
the canonical double cover of $\gms$. % (cf. \cite[(2.8)]{HKK}).
Moreover, there is an isomorphism $\Xi_\infty$ of complex manifolds that fits into the commutative diagram
\begin{gather}\label{eq:i infinity}
\xymatrix@C=6pc{
    \FQuad_\infty(\gms)\ar[r]^{\Xi_\infty} \ar[d]_{\Pi_\infty} &
        \Stab\DI \ar[d]^{\hh{Z}_\infty}\\ %{(Z,\hh{P}) \mapsto Z} \\
    \Hom(  \Ho{1}(\gms,\partial\gms;\ZZ_\Sp),\bC) \ar[r]^{ \chi_\infty} &
        \Hom(  \Grot(\DI),\bC)
,}\end{gather}
such that the period map $\Pi_\infty\colon \phi\mapsto\int\sqrt{\phi}$ becomes the map $\hh{Z}_\infty$
(i.e., central charge).
More precisely, \eqref{eq:formula} holds for $X=\Xinf$ in \eqref{eq:CA0} in Theorem~\ref{thm:IQZ}.

Furthermore, for $\ns=\Xi_\infty(\phi)$, the map $\Xinf$ induces a one-one correspondence
between saddle connections of $\phi$ on $\gms$ and $\ns$-semistable objects in $\DI$.
\end{theorem}

%=========================================================
\section{$q$-Deformation of quadratic differentials}\label{sec:q_def}
%=========================================================
In this section, we fix a complex number $s \in \bC$ with $\Re(s)>2$ and
set $q_s=e^{\ii \pi s}$.

%=========================================================
\subsection{Multi-valued quadratic differentials}
%=========================================================
Now we introduce the $q$-deformed generalization of GMN-differentials in Section~\ref{sec:GMN}.
\label{sec:q_quad}
We first define $q_s$-quadratic differentials on a disk.
Fix the following notations:
\begin{itemize}
\item $D:=\{z \in \bC\,\vert\,|z|<1\,\}$ is the unit disk;
\item $D^*:=D\setminus\{0\}$ is the punctured unit disk;
\item $\widetilde{D^*}$ is the universal cover of $D^*$.
\end{itemize}
In fact, $\widetilde{D^*}$ is isomorphic to the half plane $\{w \in \bC\,\vert\, \Re(w)<0\,\}$ with covering map
\begin{gather*}\begin{array}{rcl}
    \widetilde{D^*} &\longrightarrow& D^*\\
                w &\mapsto& e^w.
\end{array}
\end{gather*}
Consider a holomorphic quadratic differential
%\[    \xi=e^{(ks+l)w+2}\diff w^{\otimes2}\]
$$e^{(ks+l+2)w}\diff w^{\otimes2}$$
on $\widetilde{D^*}$, where $k,l \in \ZZ$.
Set $z=e^w$. Then $\diff z= e^w \diff w$ and
one can view
$$\xi=z^{ks+l}\diff z^{\otimes2}$$
as a multi-valued quadratic differential on $D$ ramified at $0$.
Let $\gamma$ be a counter-clockwise loop around $0$. Since
the action of the deck transformation group $\pi_1(D^*)\cong \langle \gamma \rangle$
on $\widetilde{D^*}$ is given by $\gamma(w)=w+2 \pi \ii$, the monodromy of $\xi$ is given by
$\gamma^* \xi=q_s^{2k}\xi.$

Now we turn to general case.
Let $\rs$ be a compact Riemann surface and $\Ram \subset \rs$
a set of finite points.
Denote by
\begin{gather}\label{eq:Pi}
    \Pi \colon \widetilde{\rsp}\to \rsp
\end{gather}
the universal cover of $\rsp:=\rs \setminus \Ram$.
In our setup, a \emph{multi-valued quadratic differential on $\rs$ ramified at $\Ram$}
refers to a holomorphic quadratic differential on  $\widetilde{\rsp}$.
%namely a holomorphic section of $\omega_{ \widetilde{\rsp}}^{\otimes 2}$.
The fundamental group $\pi_1(\rsp,*)$ acts on $\widetilde{\rsp}$ as deck transformations.
We denote $\gamma^* \xi$ the pull-back of a multi-valued quadratic differential $\xi$
via the action of a deck transformation $\gamma \in \pi_1(\rsp,*)$.

\begin{definition}
\label{def:q_quad}
A \emph{$q_s$-quadratic differential $\xi$ on $\rs$ ramified at $\Ram$}
is a multi-valued quadratic differential on  $\rs$ ramified at $\Ram$
such that, for any $\gamma \in \pi_1(\rsp,*)$,
\begin{gather}\label{eq:m gamma}
    \gamma^* \xi=q_s^{2m(\gamma)}\xi
\end{gather}
for some $m(\gamma) \in \ZZ$.
\end{definition}

For a point $p \in \rs$, let $D_p$ be a small disk (neighbourhood) centered at $p$ with coordinate $z$.
Similarly as above, $D_p^*=D_p\setminus\{p\}$ and $\widetilde{D^*_p}$ its universal cover.
Then $D_p^* \subset \rsp$ for any $p \in \Ram$ and
the pull-back
\begin{gather}\label{eq:Pi-1}
    \Pi^{-1}(D_p^*)=\coprod_\alpha \widetilde{D^*_p} ^\alpha
\end{gather}
is the disjoint unions of countable copies of $\widetilde{D^*_p}$.
We say a $q_s$-quadratic differential $\xi$ has the local form
\begin{gather}\label{eq:xi}
    \xi=c\, z^{k_p s+l_p}\diff z^{\otimes2}
\end{gather}
on $D_p$ if
%$\xi=c\,e^{(k_p s+l_p)w+2}\diff w^{\otimes2}$
$$\xi=c\,e^{(k_p s+l_p+2)w}\diff w^{\otimes2}$$
on some connected component $\widetilde{D^*_p} ^0 \subset \Pi^{-1}(D_p^*)$,
for some $c\in \bC$ and $k_p, l_p \in \ZZ$.
Note that in such a case, by Definition~\ref{def:q_quad},
$\xi$ can be written as
%\[\xi =c\, q_s^{2m_\alpha}\,e^{(k_p s+l_p)w+2}\diff w^{\otimes2}, \]
\[\xi =c\, q_s^{2m_\alpha}\,e^{(k_p s+l_p+2)w}\diff w^{\otimes2}, \]
for some $m_\alpha \in \ZZ$,
on any connected component $\widetilde{D^*_p} ^\alpha \subset \Pi^{-1}(D_p^*)$.

Now we introduce the notion of zeroes and poles for $q_s$-quadratic differentials.

\begin{definition}
Let $\xi$ be a $q_s$-quadratic differential and
assume that it has the local form \eqref{eq:xi} on $D_p$ for some $p \in \Ram$,
where $c \in \bC$ and $k_p,l_p \in \ZZ$.
We call $p$
\begin{itemize}
\item a \emph{zero of order $(k_p s+l_p)$}
    if $\Re (k_p s+l_p)>0$;
\item a \emph{pole of order $-(k_p s+l_p)$}
    if $\Re (k_p s+l_p)<0$.
\end{itemize}
And a pole $p$ of $\xi$ is a \emph{higher order pole} if $\Re (k_p s+l_p)<-2$.
Moreover,
\begin{itemize}
\item a zero $p$ of $\xi$ is called \emph{$s$-simple} if $\xi$ has the local form
\begin{gather}\label{eq:s-simple}
    \xi=c z^{s-2}\diff z^{\otimes 2};
\end{gather}
\item a pole $p$ of $\xi$ is called an \emph{$s$-pole of type $(k,l)\in\ZZ^2$} if $\xi$ has the local form
\begin{gather}\label{eq:s-pole}
    \xi=c z^{-k(s-2)-l}\diff z^{\otimes 2}.
\end{gather}
\end{itemize}

Associating to a $q_s$-quadratic differential $\xi$, there is
a multi-set of pairs of integers
\begin{gather}\label{eq:ukul}
    (\uk,\ul)\colon=\{(k_1,l_1),\dots,(k_b,l_b)\}
\end{gather}
defined via gathering all types of $s$-poles of $\xi$,
where $b$ is the number of poles of $\xi$.
We call the multi-set $(\uk,\ul)$ the \emph{$s$-polar type of $\xi$}.
\end{definition}
\begin{remark}
If $s\in \QQ$, the pair of integers $(k_i,l_i)$ is not uniquely determined by the local form $\xi$ around the corresponding $s$-pole.
Therefore the $s$-polar type $(\uk,\ul)$ for $\xi$ carries the additional information.
\end{remark}

Recall that a $q_s$-quadratic differential $\xi$ defines the
monodromy representation
\begin{gather}\label{eq:rho}
\begin{array}{rcl}
    \rho_s \colon \pi_1(\rsp,*)& \to&  \bC^*\\
        \gamma &\mapsto& q_s^{2m(\gamma)},
\end{array}
\end{gather}
where $m(\gamma)$ is given by \eqref{eq:m gamma}.

\begin{assumption}
\label{ass:trivial}
%Let $\xi$ be a $q_s$-quadratic differential on $\rs$ ramified at $\Po$,
%and consider the monodormy representation
%$\rho \colon \pi_1(\rsp,*) \to \langle q \rangle$.
For a $q_s$-quadratic differential $\xi$,
there is a simply connected open domain $ U \subset \rs$ such that:
\begin{itemize}
  \item $U$ contains all ramification points $\Ram$, i.e., $\Ram \subset U$;
  \item the induced monodromy representation
$\rho_s \circ \iota_* \colon \pi_1(\rs \setminus U,*) \to \bC^*$,
via the inclusion $\iota \colon \rs \setminus U \to \rsp$,
 is trivial.
\end{itemize}
\end{assumption}

In what follows, we first introduce a special class of $q_s$-quadratic differentials
and then consider their moduli spaces.
\begin{definition}
\label{def:BS_quad}
Let $g$ be a non-negative integer
and $(\uk,\ul)$, as in \eqref{eq:ukul}, a multi-set of pairs of integers
satisfying $k_i>0$ and \eqref{eq:4-4g}.
A \emph{CY-$s$ type $q_s$-quadratic differential} $\xi$ on genus $g$ Riemann surfaces $\rs$
is a $q_s$-quadratic differential satisfying:
\begin{itemize}
\item all zeroes are $s$-simple;
\item the number of $s$-poles is $b$ and the $s$-polar type is $(\uk,\ul)$;
\item all poles are higher order $s$-poles, namely,
\begin{gather}\label{eq:higher}
    \Re(k_i (s-2)+l_i)>2,\qquad\forall i=1,\dots,b;
\end{gather}
%Denote by $\Po=\Pol(\xi)$ the set of higher order $s$-poles.
\item Assumption~\ref{ass:trivial}.
\end{itemize}
We fix the following convention of notations.
\begin{itemize}
  \item $\Zer(\xi)=\{Z_1,\ldots,Z_\aleph\}$ is the set of $\aleph$ $s$-simple zeroes.
  \item $\Pol(\xi)=\{p_1,\ldots,p_b\}$ is the set of $s$-poles with $p_i$ of type $(k_i,l_i)$.
\end{itemize}
\end{definition}
In the setting above, the set of ramification points is $\Ram=\Zer(\xi)\bigcup\Pol(\xi)$.

The next statement clarifies the relationship between the number $\aleph$ of $s$-simple zeroes of a CY-$s$ type $q_s$-quadratic differential
and its $s$-polar type.
\begin{lemma}
\label{lem:zeroes}
Let $\xi$ be a CY-$s$ type $q_s$-quadratic differential of $s$-polar type $(\uk,\ul)$ as in \eqref{eq:ukul}. Then we have
\begin{gather}\label{eq:n}
    \aleph=\sum_{i=1}^b k_i.
\end{gather}
\end{lemma}
\begin{proof}
%Let $Z_1,\dots,Z_{\aleph}$ be the $s$-simple zeroes of $\xi$ and $p_1,\dots,p_b$ be the $s$-poles of $\xi$.

For the case $s \notin \QQ$, denote
by $\gamma_{j}$ an anticlockwise loop around $Z_j$ and by $\delta_i$ an anticlockwise loop around $p_i$.
Then Assumption~\ref{ass:trivial} implies that
\[
\prod_{j=1}^{\aleph} \rho_s(\gamma_j) \cdot \prod_{i=1}^b \rho_s(\delta_i)=1.
\]
Since $\rho_s(\gamma_j)=q_s^2$ and $\rho(\delta_i)=q_s^{-2k_i}$, the above equality
yields
\[
    \left(2\aleph-2\sum_{i=1}^b k_i \right) s\in 2\ZZ.
\]
Therefore, by the fact that $s \notin \QQ$, we have \eqref{eq:n}.

For the case $s \in \QQ$, we can assume
$s=j/m$ for some $j,m \in \ZZ$. Note that $\xi^m$ is a single-valued meromorphic
section of $\omega_S^{\otimes 2m}$ with $\aleph$ zeroes of order $m(s-2)$ and $b$
poles of orders $m(k_i(s-2)+l_i)$, for $1\le i\le b$.
By the Riemann-Roch theorem, we obtain
\[
    m(4g-4)=  m \aleph (s-2)-m\sum_{i=1}^b (k_i(s-2)+l_i)
\]
which implies \eqref{eq:n} noticing \eqref{eq:4-4g}. This completes the proof.
\end{proof}

\begin{remark}
When $s=3$, we can recover the GMN differential on Riemann surfaces in Definition~\ref{def:GMN} (\cite{BS}).
For instance, when $\rs$ is $\PP^1$ with numerical data (i.e., the normal $A_2$ case)
(cf. Example~\ref{ex:num})
$$g=0,\;b=1,\;(k,l)=(3,4),$$
we will get the corresponding CY-3 $A_2$ example in Example~\ref{ex:a2} with singularity type $(1,1,1,-7)$,
whose foliation as shown in Figure~\ref{fig:Quad A2}.
\end{remark}

%=========================================================
\subsection{Log surfaces}
%=========================================================
\label{sec:log_surf}
%Recall that $g$ is a non-negative integer and $(\uk,\ul)$ is a multi-set of pairs of integers satisfying $k_i >0$ and \eqref{eq:4-4g}.
Continuing Section~\ref{sec:q_quad}, we consider a CY-$s$ type $q_s$-quadratic differential $\xi$
on a genus $g$ Riemann surface $\rs$ with an $s$-polar type $(\uk,\ul)$.
As before, denote by $\gamma_{j}$ the loop around $Z_j$ and by $\delta_i$ the loop around $p_i$
in the simply connected domain $U \subset \rs$.

Define the representation of $\pi_1(\rsp,*)$ on
the infinite cyclic group $\langle q \rangle\cong\ZZ$ as follows.
To start with, we have a representation
%For $\xi$, let $Z_1,\dots,Z_{\aleph}$ be its $s$-simple zeroes and $p_1,\dots,p_b$ its $s$-poles of type $(k_1,l_1),\dots,(k_b,l_b)$ respectively.
\[
\underline{\rho} \colon \pi_1(U \setminus \Ram, *) \to \langle q \rangle
\]
defined by setting $\underline{\rho}(\gamma_j):=q$ and $\underline{\rho}(\delta_i):=q^{-k_i}$.
Note that for a loop $\Gamma\subset U$ encircling all points in $\Ram$,
we have $\underline{\rho}(\Gamma)=1$.
Then we can extend  $\underline{\rho}$ to the representation
\begin{gather}\label{eq:rho0}
    \rho \colon \pi_1(\rsp,*)\to \langle q \rangle,
\end{gather}
such that
\begin{itemize}
  \item the restriction of $\rho$ on $U \setminus \Ram$
  is $\underline{\rho}$,
  \item $\rho(\gamma)=1$ for any loop $\gamma \in \pi_1(\rs \setminus U)$.
\end{itemize}

Using the representation $\rho$, we introduce logarithmic surfaces.
\begin{definition}
The \emph{logarithmic ($\log$) surface $\log \rsp$ associated to
$\xi$} is defined to be the covering space $\pi \colon \log \rsp
\to \rsp$ corresponding to the representation $\rho$. Since $\xi$ is
a single-valued quadratic differential on $\log \rsp$,
the latter has a natural grading in the sense of Definition~\ref{def:GMS}
given by the horizontal foliation of $\xi$.
The infinite cyclic group $\langle q\rangle$ acts on $\log \rsp$
as deck transformations.
\end{definition}

We end this subsection with one more remark.
\begin{remark}
For an $s$-simple zero $Z_j$ with loop $\gamma_j$ around it,
we have $\rho(\gamma_j)=q$.
Taking $p=Z_j$, we know that $\Pi^{-1}(D^*_p)$ in \eqref{eq:Pi-1}
consists of just one copy of $\widetilde{D^*_p}$.
Similarly, for an $s$-pole $p_i$ of type $(k_i,l_i)$ with loop $\delta_i$ around it,
we have $\rho(\delta_i)=k_i$.
Taking $p=p_i$, we know that $\Pi^{-1}(D^*_p)$ in \eqref{eq:Pi-1} becomes
\begin{gather}\label{eq:PiDp}
    \Pi^{-1}(D_p^*)=\coprod_{i=1}^{k_i} \widetilde{D^*_p} ^i
\end{gather}
consisting of $k_i$ copy of $\widetilde{D^*_p}$.
The deck transformation $q$ maps $\widetilde{D^*_p} ^i$ to $\widetilde{D^*_p} ^{i+1}$,
where the subscript $i,i+1$ is considered modulo $k_i$.
\end{remark}

%=========================================================
\subsection{Hat homology groups}\label{sec:HatH}
%=========================================================

 Let $\log \rsp$ be the log surface associated to
a $q_s$-quadratic differential $\xi$ on $\rs$ and
\begin{gather}
    \psi_s:=\sqrt{\xi}
\end{gather}
the square root of $\xi$.
Note that $\psi_s$ is a
double-valued holomorphic $1$-form on $\log \rsp$ and it defines the monodromy representation
\begin{gather}\label{eq:rep1/2}
\begin{array}{rcl}
\rho^{\frac{1}{2}} \colon \pi_1(\rsp,*)
&\to& \langle q,-q \rangle\\
\gamma &\mapsto& \pm q^{m(\gamma)},
\end{array}
\end{gather}
where
$\gamma^* \psi_s=\pm q_s^{m(\gamma)}\psi_s$.
The image of $\rho^{\frac{1}{2}}$ is given by
\[
    \Im \rho^{\frac{1}{2}}=
    \langle q \rangle \quad \text{or}\quad
        \langle q,-q \rangle.
\]
Now we consider the Riemann surface associated to $\psi_s$.
\begin{definition}
The \emph{spectral cover }
\[
    \widehat{\Sp} \colon \log \widehat{\rs}^{\circ}  \to \log \rsp
\]
is defined to be the covering space corresponding
to the representation $\rho^{\frac{1}{2}}$.
\end{definition}

The holomorphic $1$-form $\psi_s$ is single-valued on $\log \widehat{\rs}^{\circ}$.
Note that
\begin{itemize}
  \item $\log \widehat{\rs}^{\circ}$ is a double cover of $\log \rsp$ if
    $\Im \rho^{\frac{1}{2}}=\langle q,-q \rangle$;
  \item $\log \widehat{\rs}^{\circ}=\log \rsp$ if $\Im \rho^{\frac{1}{2}}=\langle q \rangle$.
\end{itemize}
Let $\widehat{\rs}^{\circ}$ be the covering space $\Sp \colon \widehat{\rs}^{\circ} \to \rsp$
corresponding to the representation
\[
\pi_1(\rsp,*) \to \{\pm 1\},
\]
which is obtained by setting $q \mapsto 1$ in the definition of $\rho^{\frac{1}{2}}$.
The above constructions give rise to the following commutative diagram:
\[
\xymatrix {
 \log \widehat{\rs}^{ \circ} \ar[d]_{\widehat{\Sp}} \ar[rr]^{\widehat{\pi}}
 && \widehat{\rs}^{\circ} \ar[d]^{\Sp}  \\
 \log \rs^{ \circ} \ar[rr]^{\pi} && \rsp .
}
\]
 Moreover, the double cover $\Sp \colon \widehat{\rs}^{\circ} \to \rsp$
can be extended to a branched double cover
$$\operatorname{Sp} \colon \widehat{\rs} \to \rs,$$
branching at the poles of $\xi$ with odd winding numbers (equivalently, with odd $l_i$).

To introduce an appropriate integration of $\psi_s$ on
$\log \widehat{\rs}^{\circ}$,
we define the following homology group.
\begin{definition}\label{def:hhg}
The \emph{hat homology group} of $\xi$ is defined as follows.
\begin{itemize}
\item If $\log \widehat{\rs}^{\circ} =\log \rsp$, then
$$ \HHo{}(\xi)\colon=\Ho{1}(\log \rsp;\ZZ). $$
\item If $\log \widehat{\rs}^{\circ}$ is a double cover of $\log \rsp$,
denote by $\tau$ the covering involution and
\[  \HHo{}(\xi):=\Ho{1}(\log \rsp;\ZZ)^-, \]
\end{itemize}
where $\Ho{1}(\log \rsp;\ZZ)^-$ is the $\tau$-anti-invariant part, i.e., consists of
$\gamma \in \Ho{1}(\log \widehat{\rs}^{\circ};\ZZ)$ satisfying $\tau_* \gamma=- \gamma$.
The group $\HHo{}(\xi)$ has an $R=\ZZ[q^{\pm 1}]$-module structure
induced from the action of the
deck transformation group $\langle q,-q \rangle$ (or $\langle q \rangle$)
on $\log \widehat{\rs}^{\circ}$.
\end{definition}

%Recall that $\bC_s$ is the complex plane equipped with $R$-structure given by $q \cdot z:=e^{ \ii \pi s}z$.
\begin{definition}
The \emph{period} of $\xi$ is the integration of $\psi_s$ on cycles in $\HHo{}(\xi)$,
which is an $R$-linear map:
\begin{gather}\label{eq:pp}
    \begin{array}{rcl}
    Z_{\xi} \colon \HHo{}(\xi) &\to& \bC_s\vspace{.5pc}\\
    \gamma &\mapsto& \displaystyle{\int_\gamma}\psi_s.
\end{array}\end{gather}
\end{definition}
Note that the construction of $\rho$ implies that
\begin{gather}\label{eq:GG}
    q^* \psi_s=q_s \psi_s
\end{gather}
and hence
\[
    \int_{q_* \gamma}\psi_s=\int_{\gamma}q^* \psi_s=
        q_s \int_{\gamma} \psi_s
\]
for any cycle $\gamma$ in $\HHo{}(\xi)$.

The rest of the subsection is devoted to the calculation of the rank of $\HHo{}(\xi)$.
\begin{proposition}
\label{prop:rank1}
Let $\xi$ be a CY-$s$ type $q_s$-quadratic differential of genus $g$
and $s$-polar type $(\uk,\ul)$ as in \eqref{eq:ukul}.
Then the associated hat homology group
$\HHo{}(\xi)$ is a free $R$-module of rank
\[
    n=2g-2+b+\sum_{i=1}^b k_i.
\]
\end{proposition}

To prove the statement, we prepare several lemmas first.
Let $B^{(m)}$ be a bouquet of $m$ circles joined at
the point $*$, namely,
\[
B^{(m)}=S^1 \vee \cdots \vee S^1.
\]
The fundamental group of $B^{(m)}$ is a free group
with $m$ generators, i.e.,
\[
\pi_1(B^{(m)},*) \cong \langle \gamma_1,\dots,\gamma_m \rangle,
\]
where, for $1\leq i\leq m$, each $\gamma_i$ is a loop around the $i$-th circle of $B^{(m)}$.

Let $$\rho \colon \pi_1(B^{(m)},*) \to \langle q \rangle$$
be any given surjective group homomorphism and take the covering
space $$\pi \colon \widetilde{B}^{(m)} \to B^{(m)}$$ corresponding to $\rho$.
Set $\widetilde{*}=\pi^{-1}(*)$.
\begin{lemma}
\label{lem:bouquet1}
The relative homology group $\Ho{1}( \widetilde{B}^{(m)} ,\widetilde{*};\ZZ)$
is a free $R$-module of rank $m$ generated by the classes
$[\widetilde{\gamma}_1],\dots,[\widetilde{\gamma}_m]$,
where $\widetilde{\gamma}_i$ is a lift
of the loop $\gamma_i$ around the $i$-th circle in $B^{(m)}$.
\end{lemma}
\begin{proof}
For each $1\leq i\leq m$, take a lift $\widetilde{\gamma}_i$ of $\gamma_i$ on $\widetilde{B}^{(m)}$.
Then all other lifts are of the forms $(q_*)^{n_i}\, \widetilde{\gamma}_i$ for some $n_i \in \ZZ$.
Due to the fact that any path which connects two points in $\widetilde{*}$ is homotopy equivalent to
the composition of lifts of $\gamma_i^{\pm 1}$,
any relative homology class in $\Ho{1}( \widetilde{B}^{(m)} ,\widetilde{*};\ZZ)$
is a linear combination of classes $[\widetilde{\gamma}_i]$'s
over $R$. It is clear that the set
$$\big\{ [(q_*)^{n_i} \widetilde{\gamma}_i]\;\mid\; i=1,\dots,m,n_i \in \ZZ\big\}$$
is linearly independent. The claim follows.
\end{proof}

\begin{lemma}
\label{lem:bouquet2}
The homology group $\Ho{1}(\widetilde{B}^{(m)};\ZZ)$ is a free
$R$-module of rank $m-1$.
\end{lemma}
\begin{proof}
Consider the exact sequence
\[
0 \to \Ho{1}(\widetilde{B}^{(m)};\ZZ)\to
\Ho{1}( \widetilde{B}^{(m)} ,\widetilde{*};\ZZ)
\to \widetilde{\Ho{}}_0(\widetilde{*};\ZZ)\to 0,
\]
where $\widetilde{\Ho{}}_0(\widetilde{*};\ZZ)$ is the
$0$-th reduced homology group of $\widetilde{*}$.
By definition and  Lemma~\ref{lem:bouquet1},
we have $\widetilde{\Ho{}}_0(\widetilde{*};\ZZ) \cong R$ and
$\Ho{1}( \widetilde{B}^{(m)} ,\widetilde{*};\ZZ) \cong R^{\oplus m}$.
Therefore the proof completes.
\end{proof}

\begin{proposition}
\label{prop:rank2}
Let $\log \rsp$ be the logarithmic surface associated to
a CY-$s$ type $q_s$-quadratic differential of genus $g$
and $s$-polar type $(\uk,\ul)$ as in \eqref{eq:ukul}.
Then the homology group $\Ho{1}(\log \rsp;\ZZ)$ is a free $R$-module
of rank
\begin{gather}
    n=2g-2+b +\sum_{i=1}^b k_i.
\end{gather}
\end{proposition}
\begin{proof}
By Lemma~\ref{lem:zeroes}, we have
\[
    |\Ram|=b+\aleph=b+\sum_{i=1}^b k_i.
\]
Note that an oriented surface of genus $g$ with $|\Ram|$ points excluded
is homotopy equivalent to a bouquet of $2g-1+|\Ram|$ circles.
Therefore $\rsp=\rs \setminus \Ram$ is homotopy equivalent to the bouquet $B^{(n+1)}$
and $\log \rsp$ is its covering that corresponds to $\rho$ in \eqref{eq:rho0}.
As $\rho(\gamma_j)q$ for loop $\gamma_j$ around any $s$-zero $Z_j$, $\rho$ is surjective.
Hence the claim follows from Lemma~\ref{lem:bouquet2}.
\end{proof}

\emph{\noindent  Proof of Proposition~\ref{prop:rank1}.}
If $\log \widehat{\rs}^{\circ}=\log \rsp$,
the statement follows from Proposition~\ref{prop:rank2}.
Now consider the case when $\log \widehat{\rs}^{\circ}$ is a double cover of $\log \rsp$.
We first compute the rank of $\Ho{1}(\log \widehat{\rs}^{\circ};\ZZ)$ over $R$.
Recall that we have a underlying spectral/branched double cover $\widehat{\Sp}:\widehat{\rs}^{\circ}\to\rsp$
for $\widehat{\rs}^{\circ}=\widehat{\rs} \setminus \widehat{\Ram}$
and $\widehat{\Ram}=\widehat{\Sp}^{-1}(\Ram)$.
The Riemann-Hurwitz formula implies that
\[
    2\widehat{g}-2=2(2g-2)+b_1,
\]
where $\widehat{g}$ is the genus of $\widehat{\rs}$ and $b_1$ is the number
of odd integers in $\{l_1,\dots,l_b\}$.
In addition,
we have
\[
    |\widehat{\Ram}|=2|\Ram|-b_1.
%    =2b+2\sum_{i=1}^b k_i-b_1.
\]
Therefore the surface $\widehat{\rs}^{\circ}$ is homotopy equivalent
to the bouquet of $m$ circles where
\[
    m=2\widehat{g}-1+|\widehat{\Ram}|=4g-3+2|\Ram|=2n+1.
\]
Then, by Proposition~\ref{prop:rank2}, the homology group $\Ho{1}(\log \widehat{\rs}^{\circ};\ZZ)$ is
a free $R$-module of rank $2n$.
Since the $\tau$-invariant part of $\Ho{1}(\log \widehat{\rs}^{\circ};\ZZ)$
can be identified with $\Ho{1}(\log \rsp;\ZZ)$, the rank of
$\tau$-anti-invariant part is given by
\[
\rank_R \Ho{1}(\log \widehat{\rs}^{\circ};\ZZ)-
\rank_R \Ho{1}(\log \rsp;\ZZ)=2n-n=n.
\]\qed

%=========================================================
\subsection{Comparing log surfaces}\label{sec:MC}
%=========================================================
Keep the notations as above, i.e.
$\xi$ is a CY-$s$ type $q_s$-quadratic differential of genus $g$ and
$s$-polar type $(\uk,\ul)$ on $\rs$ ramified at $$\Ram=\Zer(\xi)\bigcup\Po.$$

As in Section~\ref{sec:log_surf}, near each $s$-simple zero $Z \in \Zer(\xi)$,
we can take a disk $D_Z \subset \rs$
with center $Z$ and the restriction of covering map
$\Sp\colon\log \rsp \to \rsp$ on $D_Z^*=D_Z\setminus\{Z\}$ is the
universal covering $\widetilde{D_Z^*} \subset \log \rsp$.
By adding a point $\widetilde{Z}$ on $\widetilde{D_Z^*}$ as the fiber of $Z$,
we can extend the covering $\widetilde{D_Z^*} \to D_Z^*$ to a branched $\ZZ$-covering
$\widetilde{D_Z} \to D_Z$, branching at $\widetilde{Z}$.
\begin{lemma}
The construction above, i.e., adding $\widetilde{Z}$ can be viewed as
the \emph{metric completion} with respect to the $\xi$-metric on $\log \rsp$.
\end{lemma}
\begin{proof}
Take coordinate $w \in \widetilde{D_Z^*}$ and $z \in D_Z^*$ satisfying $z=e^w$.
Then $\xi$ on $\widetilde{D_Z^*}$ can be written as
\[
    \xi|_{\widetilde{D_Z^*}}=(e^w)^{s-2}(\diff e^w)^{\otimes 2}=e^{sw}\diff w^{\otimes 2}.
\]
We identify
\[
    \widetilde{D_Z^*}\cong\{w \in \bC\,\vert\, \Re w<0\,\}\subset
    \bC \cup \{\infty\}=\bC \PP^1.
\]
Then due to \cite[Prop.~2.1 and Lemma~2.5]{HKK}, the point $\infty$
is an exponential type singularity and we can obtain a single additional point
$\widetilde{Z}$ by the (partial) completion of $\widetilde{D_Z^*}$
with respect to the $\xi$-metric.
\end{proof}

Denote by $\logrs$ the topological surface obtained from $\log \rsp$
by adding $\widetilde{Z_j}$ for each $Z_j \in \Zer(\xi)$.
When there is no confusion, we will identify $\widetilde{Z_j}$ with $Z$
and use the notation $\Tri=\Zer(\xi)$.

Next we consider an $s$-pole $p \in \Po$ of type $(k,l)$ of $\xi$,
and take a disk $D_p\subset \rs$ with center $p$.
Then the real blow-up of $D_p$ at $p$ is an annulus $A_p$ obtained by replacing $p$
with a boundary $\partial_p\simeq S^1$.

Now we consider the fiber of $D_p^{*}$ in $\log \rsp$, as calculated in \eqref{eq:PiDp},
consisting of $k$ connected component $\widetilde{D^*_p} ^i$.
Corresponding to the real blow-up of $D_p$ at $p$,
we add a real line $\RR$ on each $\widetilde{D^*_i} ^i$ as the fiber of $\partial_p \subset A_p$
and extend the universal covering $\widetilde{D^*_p}^i \to D_p^*$ to the following one.

\begin{lemma}
The construction above leads to an extended universal covering
$$\widetilde{D^*_p}^i\cup \RR=\colon \widetilde{A_p}^i \longrightarrow A_p.$$
In the coordinate
$$\widetilde{D_p^*} ^i\cong \{x+\mathbf{i}y\,\vert\,x<0, y \in \RR  \},$$
the real line $\RR$ corresponds to
$\{-\infty+ \mathbf{i}y \,\vert\, y \in \RR   \}$.
\end{lemma}

We call this construction an \emph{real blow-up of $\log \rsp$} at a pole $p$.
Denote by $(\log \rsp)^{\xi}$ the surface obtained from $\log \rs^{\circ}$
by performing real blow-up at all poles of $\xi$.

\begin{definition}
The log surface $\logrsx$ associated to
a CY-$s$ type $q_s$-quadratic differential $\xi$ on $\rs$
is the topological surface obtained from $\log \rsp$
by adding back all zeroes in $\Zer(\xi)$ as above and performing real blow-up at all poles in $\Pol(\xi)$.
\end{definition}

\begin{remark}\label{rem:num}
We define the numerical data
$$\num(\logrsx)=
    \num(\rs,\xi)=(g,b;\uk,\ul;\LP_g)$$
of $\logrsx$ as follows:
\begin{itemize}
  \item The genus $g$ and the number $b$ of boundary components are inherited from $\rs$.
  \item $(\uk,\ul)$ is given by the polar type of $\xi$.
  \item The Lekili-Polishchuk data $\LP_g$ is defined the same way as in Definition~\ref{def:numerical},
  via the winding numbers on $\rs$ with respect to $\xi$.
  Note that, when choosing (e.g. non-separating) curves, one needs to avoid
  the simply connected domain $U$ in Assumption~\ref{ass:trivial}.
\end{itemize}
\end{remark}
Then we have the following, which basically says that the winding number of
an exponential pole of polar type $(k,l)$ with local coordinate \eqref{eq:exp}
matches the winding number of a CY-$s$ pole of polar type $(k,l)$ with local coordinate \eqref{eq:s-pole}.
Cf. Figure~\ref{fig:cut} for illustration for the case when $(k,l)=(3,4)$ and $s=3$.

\begin{theorem}\label{thm:winding}
Let $\gms$ be a graded marked surface and
$(\rs,\xi)$ a Riemann surface with a CY-$s$ type $q_s$-quadratic differential.
If $\num(\gms)=\num(\rs,\xi)$,
then there is a homeomorphism
\[
    h \colon \LS \to  \logrsx,
\]
which commutes with the action of the deck transformation group $\langle q \rangle$.
Moreover, the numerical data $\num\logrsx$ matches $\num(\LS)$ under $h$.
\end{theorem}
\begin{proof}
The data $(g,b,\uk)$ is trivial to check while the tricky part is
to show that winding numbers (i.e., index $l_i$) matches.
This becomes the calculation of winding numbers in the following two situations:
\begin{itemize}
  \item of a loop $\gamma$ around an exponential type singularity,
  where the local coordinate of the quadratic differential is in the form of
  $e^{z^{-k}}z^{-l} g(z)\diff z^{\otimes2}$ as in \eqref{eq:exp}.
  This gives the grading change in $\LS$;
  \item of a loop $\gamma$ that contains an $s$-pole of type $(k,l)$ together
  with $k$ $s$-simple zeroes.
  This gives the corresponding grading change in $\logrsx$,
  cf. left picture in Figure~\ref{fig:cut}.
\end{itemize}
The two calculations are shown in Proposition~\ref{pp:1} and Proposition~\ref{pp:2}, respectively,
in the next subsection.
As the Lekili-Polishchuk data $\LP_g$ is determined by the winding function, the theorem follows.
\end{proof}

%=========================================================
\subsection{Winding numbers}\label{sec:wind}
%=========================================================
In this subsection, we compute angle changes/winding numbers of certain cycles on
a Riemann surface with a quadratic differential.

First we recall the definition of the curvature for smooth
paths in $\RR^2$ with the classical standard flat metric $\diff x^2+\diff y^2$.
Take a coordinate $(x,y) \in \RR^2$ and consider a smooth path
\[
\gamma \colon [0,1] \to \RR^2,\quad t \mapsto \gamma(t)=(x(t),y(t)).
\]
The \emph{curvature} $\kappa(t) $ of a path $\gamma(t)$ is defined by
\[
    \kappa(t):=\frac{\dot{x} \ddot{y} -\ddot{x}\dot{y}  }{(\dot{x}^2+\dot{y}^2)^{\frac{3}{2}}   },
\]
where $\dot{x}(t)=\frac{\diff}{\diff t}x(t)$ and $\ddot{x}(t)=\frac{\diff^2}{\diff t^2} x(t)$.
Then we define the \emph{angle change of  a path $\gamma$} by
\[
\AC(\gamma):=\int_0^1 \kappa(t)|\dot{\gamma}(t)|\diff t
=\int_0^1 \frac{\dot{x} \ddot{y} -\ddot{x}\dot{y}  }{\dot{x}^2+\dot{y}^2   }dt.
\]
If we identify $\RR^2$ with $\bC$ by $ w=x +\mathbf{i}y$, then the angle change can be written as
\[
\AC(\gamma)=\Im \int_0^1 \frac{\overline{\dot{w}}\ddot{w} }{|\dot{w}|^2}\diff t  =
\Im \int_0^1 \frac{\ddot{w} }{\dot{w}}\diff t.
\]
We note that if two paths $\gamma$ and $\gamma^{\prime}$ are regularly homotopic
and have the same tangent vectors at boundaries $t=0, 1$,
then $\AC(\gamma)=\AC(\gamma^{\prime})$. We also note that by definition, the angle change is
invariant under the transformation $w \mapsto \alpha w +\beta$, where $\alpha \in \bC^*$ and $\beta \in \bC$
are some constants.

Now we extend the definition of the angle change on a Riemann surface $\rs$ with
a quadratic differential $\xi$ (of various type).
Take a small open subset $U \subset \rs \setminus \Crit(\xi)$
with a coordinate $z \in U$ and consider the distinguished coordinate $$w=\int \sqrt{\xi}.$$
For a smooth path
\[
\gamma \colon [0,1] \to U,\quad t \mapsto \gamma(t)=z(t),
\]
we can similarly define the angle change of $\gamma$ by
\[
\AC(\gamma):=\Im \int_0^1 \frac{\ddot{w}(z(t))}{\dot{w}(z(t))}dt,
\]
where $\dot{w}(z(t))=\frac{\diff}{\diff t}  w(z(t))$ and
$\ddot{w}(z(t))=\frac{\diff^2}{\diff t^2}  w(z(t))$.
In the $q_s$-quadratic differential setting,
since the coordinate $w$ is determined up to $$w \mapsto \pm q_s^m w+c,$$
the angle change $\AC(\gamma)$ is also well-defined.
For a general path
$\gamma \colon [0,1] \to \rs \setminus \Crit(\xi)$, we can define the angle change
by taking an open covering $\gamma \subset \cup_{\alpha}U_{\alpha}$ and  using
the distinguished coordinate $w$ on each $U_{\alpha}$.

\begin{lemma}
\label{lem:AC_formula}
Assume that on an open subset $U \subset \rs \setminus \Crit(\xi)$ with a coordinate $z \in U$,
the quadratic differential $\xi$ takes the form $\xi=f(z)\diff z^{\otimes2}$.
Then for a path $\gamma \colon [0,1] \to U$ with the $z$-coordinate expression $z(t)$,
we have
\[
\AC(\gamma)=\frac{1}{2} \Im \int_0^1 \frac{f^{\prime}(z(t))}{f(z(t))}\dot{z}(t)dt+
\Im \int_0^1 \frac{\ddot{z}(t)}{\dot{z}(t)}dt,
\]
where $f^{\prime}(z)=\frac{\diff}{\diff t}f(z)$.
\end{lemma}
\begin{proof}
First we note that by definition $\frac{\diff}{\diff t} w(z)=\sqrt{f(z)}$. So we have
\begin{gather*}
\frac{\diff}{\diff t}w(z(t))=\sqrt{f(z(t))}\dot{z}(t),\\
\frac{\diff^2}{\diff t^2}w(z(t))=\frac{f^{\prime}(z(t)) }{2\sqrt{f(z(t))}}\dot{z}(t)^2+\sqrt{f(z(t))}\ddot{z}(t).
\end{gather*}
Then by the definition of $\AC(\gamma)$, we have the following alternative definition of winding numbers.
\end{proof}

%Set $D:=\{\,z \in \bC\, |\, |z| <1\,\}$.
\begin{definition}
When $\gamma$ is a loop,
the \emph{winding number} is given by (cf. Definition~\ref{def:numerical})
\begin{gather}
  \wind{\xi}(\gamma)\colon=\frac{1}{\pi}\AC(\gamma).
\end{gather}
\end{definition}
We now compute the winding numbers
around $s$-simple zeroes, $s$-poles (of type $(k,l)$) and exponential type singularities.

\begin{lemma}
\label{lem:winding1}
Consider a $q_s$-quadratic differential $\xi=z^{s-2} \diff z^{\otimes2}$ on $\bC$.
Then for a small loop   $\gamma$ around $s$-simple zero at $0 \in \bC$, we have
\begin{gather}\label{eq:wind s}
    \wind{\xi}(\gamma)=\Re(s).
\end{gather}
\end{lemma}
\begin{proof}
Set  $z(t)=e^{ \mathbf{i}t}$. Then by Lemma~\ref{lem:AC_formula}, we have
\begin{align*}
\AC(\gamma)&=\frac{1}{2} \Im \int_0^{2 \pi}  \frac{(s-2)z(t)^{s-3}} {z(t)^{s-2}}\dot{z}(t)\diff t
+\Im \int_0^{2 \pi}\frac{\ddot{z}(t)}{\dot{z}(t)}dt \\
&=\frac{1}{2} \Im \int_0^{2 \pi}  \frac{(s-2)e^{ \mathbf{i}t(s-3)}} {e^{ \mathbf{i}t(s-2)}}\mathbf{i}
e^{ \mathbf{i}t}\diff t
+\Im \int_0^{2 \pi}\frac{\mathbf{i}^2 e^{ \mathbf{i}t}}{\mathbf{i}e^{ \mathbf{i}t}}dt \\
&=\frac{1}{2}\Im 2 \pi \mathbf{i}(s-2)+\Im 2\pi \mathbf{i}\\
&=\pi \Re(s).\qedhere
\end{align*}
\end{proof}

Similarly, we have the following.
\begin{lemma}
\label{lem:winding2}
Consider a $q_s$-quadratic differential $\xi=z^{-k(s-2)-l} \diff z^{\otimes2}$ on $\bC$.
Then for a small loop   $\gamma$ around $s$-pole of type $(\uk.\ul)$ at $0 \in \bC$, we have
\begin{gather}\label{eq:wind k,l}
    \wind{\xi}(\gamma)=-k(\Re(s) -2)-l +2.
\end{gather}
\end{lemma}

\begin{proposition}\label{pp:1}
Let $\xi$ be a $q_s$-quadratic differential on $\rs$ and $U \subset \rs$ be a simply connected
domain which contains just $k$ $s$-simple zeroes and an $s$-pole of type $(k,l)$. Then
for a loop $\gamma \subset U$ encircling these zeroes and a pole, we have
\begin{gather}\label{eq:wind 2-l}
    \wind{\xi}(\gamma)=2-l.
\end{gather}
\end{proposition}
\begin{proof}
This follows from combing \eqref{eq:wind s}, \eqref{eq:wind k,l} and the following fact:
\begin{itemize}
  \item if the loop $\gamma$ is obtained from a loop $\gamma'$
  by including a single $s$-simple zero, then
  $$\wind{\xi}(\gamma)=\wind{\xi}(\gamma')+\Re(s-2).$$
\end{itemize}
The proof of this claim is illustrated in Figure~\ref{fig:wind-2}.
\begin{figure}[t]\centering
\begin{tikzpicture}[]
\draw[\separated,thick,->-=.5,>=stealth]
    (.8,-1.5).. controls +(0:1) and +(0:2) ..(0,1)
    (0,1).. controls +(180:2) and +(180:1) ..(-.8,-1.5);

\draw[\separated,thick,-<-=.3,>=stealth](-3,-1.5)to(-.8,-1.5);
\draw[\separated,thick,-<-=.7,>=stealth](.8,-1.5)to(3,-1.5)node[right]{\footnotesize{$\gamma$}};
\draw[\separated,thick,-<-=.5,>=stealth](-3,-2.3)to(3,-2.3)node[right]{\footnotesize{$\gamma'$}};

\draw[thin](1.1,-1.5)arc(0:-342:.3);
\draw(1.1-.3,-1.1)node{\footnotesize{$^{-1}$}};
\draw(-1.1+.3,-1.1)node{\footnotesize{$^{-1}$}};
\draw[thin](-1.1,-1.5)arc(-180:180-18:.3);
\draw[-<,>=stealth](1.1-.6,-1.5)to(1.1-.6,-1.5-.05);
\draw[-<,>=stealth](-1.1+.6,-1.5)to(-1.1+.6,-1.5+.05);

\draw(0,.2-.5)node[above]{\footnotesize{$\wind{}=\Re(s)$}}
(0,0-.5)node[white] {$\bullet$} node[red]{$\circ$} circle(.3);
\draw[->,>=stealth] (0,.3-.5)to(-.05,.3-.5);
\end{tikzpicture}
\caption{Winding changes $\Re(s)-2$ when goes around an $s$-simple zero}
\label{fig:wind-2}
\end{figure}
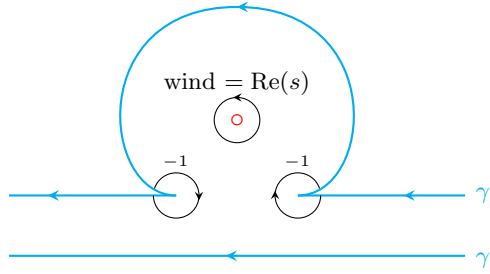
\end{proof}

Set $D=\{\,z \in \bC \mid\, |z|<2\,\}$.
\begin{proposition}\label{pp:2}
Consider a quadratic differential $\phi=e^{z^{-k}}z^{-l}  g(z) \diff z^{\otimes2} $
on $D$ with an exponential type singularity of index $(k,l)$ at $0 \in D$,
where $g(z)$ is non-zero holomorphic function on $D$.
Then, for a small loop $\gamma$ around $0$, we have
\begin{gather}\label{eq:wind exp}
    \wind{\xi}(\gamma)=2-l.
\end{gather}
\end{proposition}
\begin{proof}
Set $z(t)=  e^{ \mathbf{i}t}$.
We note that the formula in
Lemma~\ref{lem:AC_formula} can be written as
\[
\AC(\gamma)=\frac{1}{2}\Im \int_{\gamma}\frac{\diff}{\diff z} \log f(z) \diff z
+\Im \int_0^1 \frac{\ddot{z}(t)}{\dot{z}(t)}dt.
\]
Then the first term is
\begin{align*}
&\frac{1}{2}\Im \int_{\gamma}\frac{\diff}{\diff z} \log e^{z^{-k}}z^{-l}  g(z)  \diff z\\
=&\frac{1}{2}\Im \int_{\gamma}\frac{\diff}{\diff z} \left(z^{-k}-l \log z+\log g(z) \right) \diff z \\
=&\frac{1}{2}\Im \int_{\gamma} \left(-k z^{-k-1}-l z^{-1}+ \frac{g^{\prime}(z)}{g(z)} \right) \diff z \\
=&\frac{1}{2}\Im (-2 \pi \mathbf{i}l)\\=&-\pi l.
\end{align*}
Here we use the two facts:
\begin{itemize}
\item the residue of $z^{-k-1}$ at $z=0$ is zero since $k \ge 1$ and
\item the residue of $g^{\prime}(z) \slash g(z)$ at $z=0$ is zero
since this function is holomorphic near zero by the condition that $g(z)$ is non-vanishing on $D$.
\end{itemize}
Next, we can easily compute that
the contribution of the second term is $2 \pi$.
Thus we have the proposition follows.
\end{proof}

%=========================================================
\subsection{Surface framing and period maps}
%=========================================================
Let $\gms$ be a graded marked surface with numerical data $\num(\gms)$ and
an associated graded DMS $(\surfo,\uc,\Lambda)$.
Let $\LS$ be its topological log surface %with inherited grading $\log\lambda$
(cf. Definition~\ref{con:top log}).
Fix $s\in\CC$.

\begin{definition}\label{def:x.f.quad}
A \emph{$\LS$-framed $q$-quadratic differential} $\Theta=(\rs,\xi,h;s)$
consists of a CY-$s$ type $q_s$-quadratic differential $\xi$ on a genus $g$ Riemann surface $\rs$
of $s$-polar type $(\uk,\ul)$ together with an isotopy class of a homeomorphism
\begin{equation}\label{eq:h}
    h \colon \LS \to  \logrsx,
\end{equation}
such that \begin{itemize}
\item $h$ identifies the decoration $\Tri$ of $\surfo$ with the set $\Zer(\xi)$ of zeroes of $\xi$,
\item $h$ sends $\partial\LS$ to the set of real blow-ups of poles $\Pol(\xi)$ of $\xi$,
\item $h$ commutes with the action of the deck transformation group $\langle q \rangle$ and
\item $h$ matches the numerical data $\num(\logrsx)$ with $\num(\LS)$.
\end{itemize}
Two $q$-quadratic differential $\Theta_i=(\rs_k,\xi_i,h_i;s)$ are $\LS$-equivalent if
\begin{itemize}
\item there is a biholomorphism
$F \colon \rs_1 \to \rs_2$ satisfying $F(\Pol(\xi_1))=\Pol(\xi_2)$, together with the choice of a lift
$\widetilde{F} \colon \log \rs_1 \to \log \rs_2$
such that $\widetilde{F}^* \xi_{2}=\xi_{1}$ and
\item
$h_2^{-1} \circ \log\widetilde{F} \circ h_1\in\Homeo_0(\LS)$,
where
$$\log\widetilde{F} \colon (\log \rs_{1,\Tri}^{\circ})^{\xi_1}\to
    (\log \rs_{2,\Tri}^{\circ})^{\xi_2}
$$
is the induced homeomorphism.
\end{itemize}
Here, $\Homeo_0(\LS)$ consists of isotopy classes of homeomorphisms of $\LS$ that
commute with the deck transformation and are isotopy to identity.
\end{definition}

Denote by $\QQuad_s(\LS)$ the moduli space of $\LS$-framed $q$-quadratic differential
$q$-quadratic differentials.

Consider the spectral cover $\log \widehat{\rs}^{\circ}\to \log \rsp$.
This can be naturally extended to the double cover  $
(\log \widehat{\rs}^{\circ}_{\Tri})^{\xi}\to (\log \rsp_{\Tri})^{\xi}$
with metric completions at zeroes and oriented blow-ups at poles.
Correspondingly, we can also define the spectral cover of the
topological log surface $\log \widehat{\surf}_{\Tri}\to \LS$
with the covering involution $\tau^\Tri$
such that $h \colon \LS \to  \logrsx$
lifts to the homeomorphism between spectral covers
$\widehat{h} \colon \log \widehat{\surf}_{\Tri} \to (\log \widehat{\rs}^{\circ}_{\Tri})^{\xi}$.
Finally, the framing $h$ in \eqref{eq:h} induces an isomorphism of the
hat homology groups
\[
\widehat{h}_* \colon \HHo{}(\LS)\to
\HHo{}(\xi),
\]
where $\HHo{}(\LS)$ is the corresponding hat homology group (i.e.
defined by $\tau^\Tri_* \gamma=-\gamma$).

Recall that the period \eqref{eq:pp} of $\xi$ is an $R$-linear map.
Thus a $\LS$-framed $q_s$-quadratic differential $(\xi,h)$
gives an $R$-linear map
\[
Z_{\xi}\circ \widehat{h}_* \in \Hom_{R}(\HHo{}(\LS),\bC_s).
\]
\begin{definition}
The \emph{period map}
\[
\Pi_s \colon\QQuad_s(\LS) \to \Hom_{R}(\HHo{}(\LS),\bC_s)
\]
is defined by sending $
(\xi,h) \mapsto Z_{\xi}\circ  \widehat{h}_*$.
\end{definition}

\begin{conjecture}\label{conj:period}
The period map $\Pi_s$ is a local homeomorphism.
\end{conjecture}

In the next section, we interpret $q$-quadratic differentials as $q$-stability conditions.
Along the way, we will prove the conjecture above under certain assumptions.

%=========================================================
\subsection{Quadratic differentials on log surface}
%=========================================================
\begin{remark}\label{rem:s.v.}
Note that the monodromy representation $\rho_s$ defined in Section~\ref{sec:q_quad}, cf. \eqref{eq:rho},
factors through the representation $\rho$ above.
More precisely, define a group homomorphism
\[\begin{array}{rcl}
    \operatorname{ev}_s^{\otimes 2} \colon \langle q \rangle &\to&\bC^* \\
    q&\mapsto& q_s^2,
\end{array}\]
then we have $\rho_s=\operatorname{ev}_s^{\otimes 2} \circ \rho$.
This implies that
the CY-$s$ type $q_s$-quadratic differential $\xi$ is a single-valued
non-zero holomorphic quadratic differential on $\log \rsp$ satisfying
\begin{gather}\label{eq:GG0}
    q^* \xi=q_s^2 \xi.
\end{gather}
%Such a differential should correspond to the equation \eqref{eq:X=s} for $q$-stability conditions.
%We will make this expectation precisely in the next section.
Therefore,
we have the following alternative definition of $\LS$-framed $q$-quadratic differentials.
\end{remark}
%Let $\LS$ the topological log surface associated to a graded DMS $(\surfo,\uc,\Lambda)$.
%with numerical data $\num(\gms)=(g,b;\uk,\ul;\LP_g)$.

\begin{lemma}\label{lem:log}
A \emph{$\LS$-framed $q$-quadratic differential} $\Theta=(\rs,\xi,h;s)$ is equivalent to
the data $\Theta=(\log\rs,\xi,h;s)$ as follows:
\begin{itemize}
\item a genus $g$ Riemann surface $\log\rs$ with an automorphism $q$.
\item a meromorphic quadratic differential $\xi$ on $\log\rs$ such that
\begin{itemize}
\item there are $\aleph=|\Tri|$ zeroes of $\xi$, each of which is a conical singularity of infinity angle, as in Definition~\ref{def:conical}, with local coordinate of the form
    \[
        \xi=c e^{s \omega} \diff \omega^{\otimes 2}.
    \]
    Denote by $\Tri=\{Z_j\}$ the set of these zeroes.
\item there are $k_i\cdot b$ poles $\{p_i^j\mid 1\le i\le b, j\in\ZZ_{k_i}\}$ of $\xi$ such that
    the pole $p_i^j$ is a conical singularity of $k_i\pi$ angle, with local coordinate of the form \[
        \xi=c e^{-(k_i (s-2)+l_i-2)\omega}\diff \omega^{\otimes 2}.
    \]
\item The automorphisms $q$ and $\xi$ satisfy \eqref{eq:GG0}. Thus, $q$ preserves zeroes
    and $q(p_i^j)=p_i^{j+1}$ in particular.
\end{itemize}
\item a homeomorphism $h \colon \LS \to  \logrsx$
such that
\begin{itemize}
\item the numerical data of $\logrsx$ (similarly defined as in Remark~\ref{rem:num})
    matches with the one of $\LS$ under $h$.
\item the deck transformation $q$ on $\LS$ becomes the automorphism $q$ on $\logrsx$,
where $\logrsx$ is the real blow-up of $\log\rs$ at all the poles.
\end{itemize}
\end{itemize}
\end{lemma}

\begin{con}
Let $\Theta=(\logrs,\xi,h;s)$ be a $\LS$-framed $q$-quadratic differential.
Using $h$, we will pull back various structures (i.e., foliation, metric, etc.) of $\logrsx$ to $\LS$
and omit $(\logrs, h)$ in the following discussion when there is no confusion.
Namely, we will simply write $\Theta=\xi$ and call it a $q$-quadratic differential on $\LS$.
\end{con}

%=========================================================
\section{$q$-Stability conditions via $q$-quadratic differentials}\label{sec:main}
Recall that we have the following topological setting:
\begin{itemize}
  \item $\gms=(\surf,\M,\Y,\grad)$ is a graded marked surface,
  \item $\TT$ is an (initial) full formal arc system of $\surf$,
  \item $\uc$ is the cut that is compatible with $\TT$,
  \item $(\surfo,\uc,\Lambda)$ is a graded DMS associated to $\gms$.
  \item $\LS=\bigcup_{m\in\ZZ}\surfo^{(m)}$ is the topological log surface of $\surfo$
   with grading $\log\grad$, with respect to $\uc$;
\end{itemize}
and categorical setting:
\begin{itemize}
  \item $\DT$ is the topological Fukaya category associated to $\surf$,
  \item $\DX$ is the Calabi-Yau-$\XX$ category of $\DT$,
  \item we identify $\DT$ with an $\XX$-baric heart of $\DX$ under \eqref{eq:Lagrangian}.
\end{itemize}
Note that the cut $\uc$ induces an isomorphism
\begin{gather}\label{eq:otimes}
    \HHo{1}(\LS)\cong\Ho{1}(\surf,\M;\ZZ_\Sp)\otimes R.
\end{gather}

%=========================================================
\subsection{$q$-Quadratic differentials with cuts}
%=========================================================
Denote by $\Core(\xi)$ the core of $\xi$ on $\LS$.
Let $\pi_\Tri\colon\LS\to\surfo$ be the projection.
Denote by
$$\core(\xi)\colon=\pi_\Tri(\Core(\xi))\subset\surfo$$
the projection of $\Core(\xi)$ on $\surfo$.
Note that $\core(\xi)$ is well-defined topologically
since \eqref{eq:GG0} implies that $\Core(\xi)$ and angles/direction of geodesics are
invariants under the deck transformation (but the length of geodesics may be scaled).
In order to give the metric information on $\core(\xi)$,
we need to choose a cut. We will discuss the existence of cuts in the next subsection.

\begin{definition}
Let $\xi$ be a $q$-quadratic differential on $\LS$.
A cut $\ut=\{t_i\}$ of $\surfo$ is compatible with $\xi$
if any of its arc $t_i$ does not intersect the interior of $\core(\xi)$.
Denote by $\QQuac_s(\LS)$ the $q$-quadratic differentials on $\LS$
that admit some compatible cut.
\end{definition}

Now take a $q$-quadratic differential $\xi$ with compatible cut $\ut$.
Then by Lemma~\ref{lem:cuts} there is a unique element $b_*\in\SBr(\surfo)$ such that $b(\ut)=\uc$.
Consider the induced action of $b$ on $\LS$, which is also denoted by $b$.
Denote by
\begin{gather}\label{eq:beta}
    \cz(\xi)\colon=\Core(b_*(\xi))\cap \surfo^{(0)}
\end{gather}
the induced core on $\surfo$,
which is the intersection of the core $\Core(b_*(\xi))$ with
the zero sheet $\surfo^{(0)}$ of $\LS$.
Note that $\xi$ also provides foliation/metric on $\cz(\xi)$.
By pulling the set of vertices $\Tri$ of $\cz(\xi)$ on $\surfo$ along the arcs in $\uc$
to the set $\Y$ of closed marked points of $\surfo$,
we obtain some convex hull $\cx(\xi)$ (see Figure~\ref{fig:cut}).
More precisely, this is the inverse operation of $\imc$ in \eqref{eq:imc}:
\[
    \cx(\xi) \colon= \imc^{-1} (\cz(\xi))\subset\gms.
\]

\begin{figure}[hb]\centering
\begin{tikzpicture}[scale=.4,arrow/.style={->,>=stealth,thick}]
\newcommand{\vtex}{{$\bullet$}}
\draw
    (175:5.5)coordinate (Z1)
    (5:5.5)coordinate (Z3)
    (-90:3)coordinate (Z2)

    (60:7)coordinate (U2)
    (120:7)coordinate (U1)
    (-135:7)coordinate (U3)
    (-45:7)coordinate (U4);

  \foreach \j in {.1,.2,...,.9}
    {
      \path (Z1)--(Z2) coordinate[pos=\j] (m0);
      \path (U1)--(m0) coordinate[pos=.9] (m1);
      \draw[Emerald, very thin] plot [smooth,tension=.5] coordinates {(U1)(m1)(U3)};
    }
  \foreach \j in {.1,.2,...,.9}
    {
      \path (Z3)--(Z2) coordinate[pos=\j] (m0);
      \path (U2)--(m0) coordinate[pos=.9] (m1);
      \draw[Emerald, very thin] plot [smooth,tension=.5] coordinates {(U2)(m1)(U4)};
    }

  \foreach \j in {-36,-25,-16,-9,0,9,16,25,36,49,64,81,100,121,144,169,196,225,256,289,324}
    {
      \path (U1)--(U2) coordinate[pos=.5] (m0);
      \path ($(m0)!\j*.0025!(Z2)$) coordinate (m1);
      \draw[Emerald, very thin] plot [smooth,tension=.5] coordinates {(U2)(m1)(U1)};
    }

\draw[\separated, thick](Z1)to(U1)to(Z2)to(U3)--cycle;
\draw[\separated, thick](Z3)to(U2)to(Z2)to(U4)--cycle;
\foreach \j in {10,20,30,40,50,-15,-25,-35}
{\draw[\separated,thick] (Z1)to(180-\j:7);}
\foreach \j in {10,20,30,40,50,-15,-25,-35}
{\draw[\separated,thick] (Z3)to(\j:7);}
\foreach \j in {15,-15,25,-25,35,-35}
{\draw[\separated,thick] (Z2)to(-90+\j:7);}
\draw[thick] (0,0) circle (7);

\draw(180:7)coordinate (z1)(0:7)coordinate (z3)(-90:7)coordinate (z2);
\foreach \k in {1,2,3}{
\foreach \j in {.2,.5,.8}{
\draw[->-=\j,>=stealth,very thick,green](Z\k)to(z\k);}}

\draw[red,thick](Z1)to[bend left=2](Z2)to[bend left=2](Z3);
\foreach \j in {1,2,3}{
    \draw(Z\j)node[white]{$\bullet$} node[red]{$\circ$};}

\end{tikzpicture}
\quad
\begin{tikzpicture}[scale=.4,arrow/.style={->,>=stealth,thick}]
\newcommand{\vtex}{{$\bullet$}}
\draw
    (180:7)coordinate (Z1)
    (0:7)coordinate (Z3)
    (-90:7)coordinate (Z2)

    (60:7)coordinate (U2)
    (120:7)coordinate (U1)
    (-135:7)coordinate (U3)
    (-45:7)coordinate (U4);

  \foreach \j in {.1,.2,...,.9}
    {
      \path (Z1)--(Z2) coordinate[pos=\j] (m0);
      \path (U1)--(m0) coordinate[pos=.9] (m1);
      \draw[Emerald, very thin] plot [smooth,tension=.5] coordinates {(U1)(m1)(U3)};
    }
  \foreach \j in {.1,.2,...,.9}
    {
      \path (Z3)--(Z2) coordinate[pos=\j] (m0);
      \path (U2)--(m0) coordinate[pos=.9] (m1);
      \draw[Emerald, very thin] plot [smooth,tension=.5] coordinates {(U2)(m1)(U4)};
    }

  \foreach \j in {-25,-16,-9,0,9,16,25,36,49,64,81,100,121,144,169,196,225,256,289,324}
    {
      \path (U1)--(U2) coordinate[pos=.5] (m0);
      \path ($(m0)!\j*.0025!(Z2)$) coordinate (m1);
      \draw[Emerald, very thin] plot [smooth,tension=.5] coordinates {(U2)(m1)(U1)};
    }

\draw[\separated, thick](Z1)to(U1)to(Z2)to(U3)--cycle;
\draw[\separated, thick](Z3)to(U2)to(Z2)to(U4)--cycle;
\foreach \j in {10,20,30,40,50,-25,-35}
{\draw[\separated,thick] (Z1)to(180-\j:7);}
\foreach \j in {10,20,30,40,50,-25,-35}
{\draw[\separated,thick] (Z3)to(\j:7);}
\foreach \j in {-25,25,35,-35}
{\draw[\separated,thick] (Z2)to(-90+\j:7);}
\draw[thick] (0,0) circle (7);

\draw[red,thick](Z1)to[bend left=10](Z2)to[bend left=10](Z3);
\foreach \j in {1,2,3}{
    \draw(Z\j)node[white]{$\bullet$} node[red]{$\circ$};}
\end{tikzpicture}
\caption{Horizonal strip decompositions on (a sheet of) $\LS$ and $\surf$}\label{fig:cut}
\end{figure}

The foliation of $\cx(\xi)$ gives a partial foliation of $\surf$,
which can be extended to the rest of the surface by gluing upper half planes
as the core determines a flat surface.
Thus we obtain a flat surface, or equivalently,
a quadratic differential $\phi$ on some graded marked surface $\surf_0^{\lambda_0}$
determined by
\begin{gather}\label{eq:cx xi}
    \Core(\phi)=\cx(\xi).
\end{gather}

\begin{lemma}
$\phi$ is in $\FQuad_\infty(\gms)$.
\end{lemma}
\begin{proof}
By definition, we only need to show that the numerical data of $\phi$
(or of $\surf_0^{\lambda_0}$)
coincide with $\num(\gms)$.
This is exactly what we have shown in Theorem~\ref{thm:winding}.
\end{proof}

By \eqref{eq:i infinity} in Theorem~\ref{thm:HKK},
$\phi$ corresponds to a stability condition
\begin{gather}\label{eq:ns xi}
    \ns(\xi)\colon=\Xi_\infty(\phi)\in\Stap\DT.
\end{gather}
Therefore, we obtain a triple $(\DI, \ns(\xi), s)$.
We proceed to show that such a triple induces a $q$-stability condition.
\begin{proposition}\label{pp:QS=QQ2}
Suppose that $\Re(s)\geq2$.
Let $\xi\in\QQuac_s(\LS)$.
Then the triple $(\DI, \ns(\xi),s)$, constructed as above, induces a closed $q$-stability condition
$$(\sigma,s)=(\DI, \ns(\xi),s)\otimes_*R$$
on $\DX$. Then we obtain a map
\begin{equation}\label{eq:xis}
    \Xi_s\colon\QQuac(\LS)/\SBr(\surfo)\to\CStab\DX/\ST,
\end{equation}
which is in fact an isomorphism.
\end{proposition}
\begin{proof}
Without loss of generality, assume $b$ in \eqref{eq:beta} is trivial,
i.e., $\xi$ and $\uc$ are compatible.

\subsubsection*{\textbf{Step 1}}
We only need to prove \eqref{eq:closed}, i.e., $\gldim\ns+1\le\Re(s)$.
Then Theorem~\ref{thm:IQ} implies that the construction above gives a $q$-stability condition
and we obtain a map $\Xi_s$.
Suppose not, then there is a pair of semistable objects $\widehat{M}_1$ and $\widehat{M}_2$ in $\DI$ such that
\begin{gather}\label{eq:>1}
\begin{cases}
    \Hom_{\D_\infty}(\widehat{M}_1,\widehat{M}_2)\neq0,\\
    \varphi_{\ns}(\widehat{M}_2)-\varphi_{\ns}(\widehat{M}_1)>\Re(s)-1(\geq1),
\end{cases}
\end{gather}
where $\varphi_{\ns}$ denotes the phase function with respect to $\ns$.

By Theorem~\ref{thm:IQZ}, $\widehat{M}_i$ correspond to graded closed arcs $\wg_i$ on $\surf$,
$i=1,2$,
and the non-zero morphism in \eqref{eq:>1} corresponds to an intersection
$$p \in \wg_1\cap\wg_2$$
with $\ind_p(\wg_1,\wg_2)=0$.
We claim that $p$ can not be in the interior of $\surf$ as shown below.
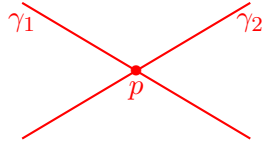
\[
\begin{tikzpicture}[xscale=.5, yscale=.3]
\draw[red, thick](3,3)node[below]{$\gamma_2$}to(-3,-3)
    (3,-3)to(-3,3)node[below]{$\gamma_1$}
    (0,0)node{\small{$\bullet$}}node[below]{$p$};
\end{tikzpicture}
\]
Otherwise, such an intersection in $\surf^\circ$ between $\wg_i$ also contributes as
$$p \in \wg_2\cap\wg_1$$
with $\ind_p(\wg_2,\wg_1)=1$ by \cite[Lemma~2.3]{IQZ}.
By Theorem~\ref{thm:IQZ} again, we also have a (non-zero) induced morphism
in $$\Hom_{\D_\infty}(\widehat{M}_2,\widehat{M}_1)[1].$$
As both $\widehat{M}_i$ (and their shifts) are (semi-)stable, we have
\[
    \varphi_{\ns}(\widehat{M}_1[1])\geq\varphi_{\ns}(\widehat{M}_2)
\]
which contradicts to the second inequality in \eqref{eq:>1}.

\begin{figure}[hb]\centering
\begin{tikzpicture}[xscale=.5, yscale=.3,
  egarrow/.style={>=stealth,thick,red}, % prev. -latex
  vline/.style={dashed,gray},
  c-vrtx/.style={blue}]
\newcommand{\vrtx}{\bullet}
% lower plane
\draw[fill=gray!7,dotted]
    (12, 5) to (6,-5) to (-12, -5) to (-6,5) -- cycle;
% left upper and lower vertices

\draw
    (1,3)coordinate (A1)
    (4,1)coordinate (A2)
    (3,-2)coordinate (A3)
    (-1,-3)coordinate (A4) node[below,red]{\tiny{$Y$}}
    (-4,-1)coordinate (A5)
    (-3,2)coordinate (A6)
    (1,3+12)coordinate (B1)
    (4,1+12)coordinate (B2)
    (-3,2+12)coordinate (B6) ;
\draw[very thick] plot [smooth,tension=1] coordinates {(-2.2,-3)(A4)(0,-4)}
plot [smooth,tension=1] coordinates {(-4.5,.5)(A5)(-3.8,-3)}
plot [smooth,tension=1] coordinates {(2,-3)(A3)(4.2,-2)};

% left vertical lines
\draw [vline] (A1) edge (1,3+15);
\draw [vline] (A2) edge (4,1+15);
\draw [vline] (A6) edge (-3,2+15);

\draw [red,thick] (A1) edge (A6) edge (A2);
\draw [red,thick] (A4) edge[bend left] (A3) edge[bend right] (A5);
\draw [red,thick] (B1) edge (B6) edge (B2);

\foreach \x/\y in {1/4,6/5,2/3}{
\foreach \j in {.2,.5,.8}{
\draw[->-=\j,>=stealth,thick,green](A\x)to(A\y);}}

\draw[blue!50,dashed,thick] (A6) to(A2);
\draw[blue!50,dashed,thick] (B6) to node[below]{\tiny{$\;\we_0$}} (B2);

\foreach \j in {3,4,5}{
    \draw(A\j)node[white]{$\bullet$} node[red]{$\circ$};}
\foreach \j in {1,2,6}{
    \draw(A\j)node[white]{$\bullet$} node[red]{$\circ$}
        (B\j)node[white]{$\bullet$} node[red]{$\circ$};}

\draw[red!49](-3,2+7)coordinate (C6)node[white]{$\bullet$} node[red!23]{$\circ$}
    to(1,3+7)coordinate (C1)node[white]{$\bullet$} node[red!23]{$\circ$}
    to(4,1+7)coordinate (C2)node[white]{$\bullet$} node[red!23]{$\circ$};

\draw(8,12)node{$\LS$}(8,3)node{$\surfo$};
\draw[red]($(A4)!.5!(A5)$)node[]{\tiny{$\gamma_1$}}
    ($(A4)!.5!(A3)$)node[]{\tiny{$\gamma_2$}};
\draw[red]($(A1)!.5!(A6)$)node[above]{\tiny{$\eta_1$}}
    ($(A1)!.5!(A2)$)node[above]{\tiny{$\eta_2$}};
\draw[red]($(B1)!.5!(B6)$)node[above]{\tiny{$q(\we_1)$}}
    ($(B1)!.5!(B2)$)node[above]{\tiny{$\we_2$}};
\draw[red!49](-1,9.5)node[above]{\tiny{$\we_1$}}(2.5,9)node[above]{\tiny{$\we_2$}};
\draw[red,>=stealth,->-=.6] (-1.7,-2)to[bend left]node[above]{\tiny{$\alpha$}}(0,-2.25);

\draw[red]plot [smooth,tension=1.5] coordinates{
    ($(A6)!.8!(A1)$)  (1+.2,3+.8)  ($(A1)!.2!(A2)$) }
    (1+.2+.2,3+.8+.5)node[]{\tiny{$_{\beta}$}};
\draw[red,thin,>=stealth,->]   (1+.2,3+.8)to(1.0001+.2,3+.8) ;
\draw[red!49]plot [smooth,tension=1.5] coordinates{
    ($(C6)!.8!(C1)$)  (1+.2,3+.8+7)  ($(C1)!.2!(C2)$) }
    (1+.2+.2,3+.8+7+.5)node[]{\tiny{$_{\widetilde{\beta}}$}};
\draw[red!49,thin,>=stealth,->]   (1+.2,3+.8+7)to(1.0001+.2,3+.8+7) ;

\draw[red,>=stealth,-<-=.65]
    ($(B6)!.8!(B1)$)  to[bend right=50]($(B1)!.2!(B2)$)
    (1-.3,15-1)node[]{\tiny{$\theta$}}
    (A1)node[below]{\tiny{$\;^Z$}};
\end{tikzpicture}
\caption{The cut (green) and lifts of closed arcs to log surface}\label{fig:cut and lift}
\end{figure}
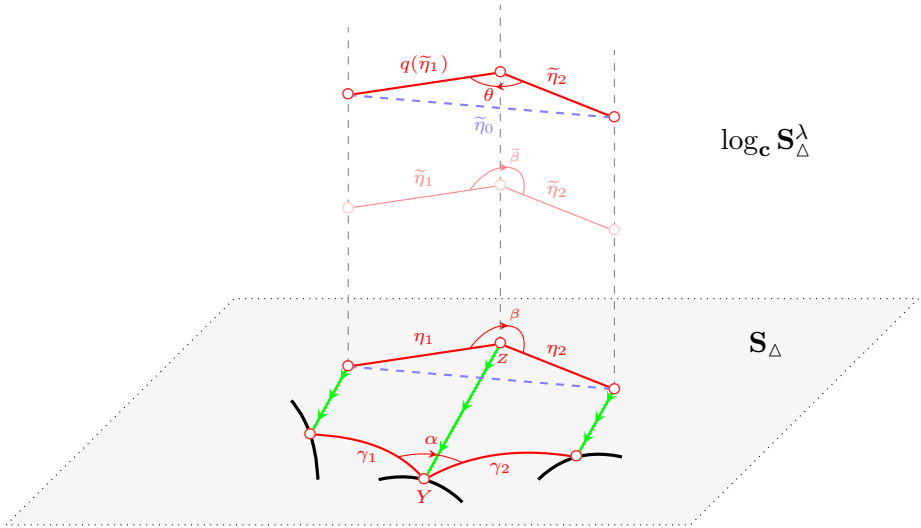

Now let $Y\in\wg_1\cap\wg_2$ be the intersection, which is in $\Y$,
corresponds to the non-zero morphism in \eqref{eq:>1}.
More precisely, such a morphism corresponds to the angle $\alpha$ from $\wg_1$ to $\wg_2$ at $Y$,
that equals to (with respect to the quadratic differential $\phi$)
\[
    \left(\varphi_{\ns}(\widehat{M}_2)-\varphi_{\ns}(\widehat{M}_1)\right)\cdot\pi.
\]
Let $\we_i=\imc(\wg_i)$, where $\imc$ is defined in \eqref{eq:imc}.
Since $\widehat{M}_i$ are semistable, $\wg_i$ are in $\Core(\phi)=\cx(\xi)$
and $\eta_i$ are in $\Core^0(\xi)$.
Denote by $Z$ the corresponding intersection between $\eta_i$, which is the decoration
connecting to $Y$ by $\uc$.
By construction,
the angle $\beta$ from $\eta_1$ to $\eta_2$ at $Z$ (with respect to the quadratic differential $\xi$)
equals the angle $\alpha$ from $\wg_1$ to $\wg_2$ at $Y$, cf. Figure~\ref{fig:cut and lift}.

Let $\we_i$ be some lifts of $\eta_i$ in $\LS$
so that $\beta$ is also lifted to an angle $\widetilde{\beta}$ from $\we_1$ to $\we_2$.
As the deck transformation $q$ acts as rotation by $e^{\ii \pi s}$ on $q$-quadratic differentials,
the angle $\theta$ from $\eta_2$ to $q(\eta_1)$ at $Z$ is
\[\begin{array}{rcl}
    \theta&=&\Re(s)\cdot\pi - \widetilde{\beta}\\
    &=&\Re(s)\cdot\pi - \beta\\
    &=&\Re(s)\cdot\pi - \alpha\\
    &=&\Re(s)\cdot\pi - (\varphi(\widehat{M}_2)-\varphi(\widehat{M}_1))\cdot\pi.
\end{array}\]
By \eqref{eq:>1}, $\theta<\pi$.
Since both $\eta_2$ and $q(\eta_1)$ are in the core $\Core(\xi)$,
which is a convex hull,
there is a geodesic $\we_0$ in $\Core(\xi)$
obtained from $q(\eta_1)\cup\eta_2$ by smoothing the intersection at $Z$
(on the side of the angle $\theta$, cf. the violet dashed arc in Figure~\ref{fig:cut and lift}).

Then the projection $\pi_\Tri(\we_0)$ of $\we_0$ on $\surfo$, which is in $\cz(\xi)$,
will intersect the segment $ZY$ in the cut $\cut$.
This contradicts to the compatibility between $\xi$ and $\uc$.
Thus \eqref{eq:closed} holds and $\sigma$ is a closed $q$-stability condition
and we obtain a map $\Xi_s$ in \eqref{eq:xis}.
The injectivity of $\Xi_s$ follows from the facts that:
$\Xi_\infty$ in \eqref{eq:i infinity} is injective and
the core determines quadratic differentials.

\subsubsection*{\textbf{Step 2}}
Reversing the process above, we have the following.

Given a point in $\CStab\DX/\ST$,
it can be represented by a $q$-stability condition $(\sigma,s)$ induced from
some triple $(\DT,\ns,s)$ for $\ns\in\Stap\DT$.
Let $\phi=\Xi_\infty^{-1}(\ns)$ be the exponential type quadratic differential on $\gms$
that corresponds to $\ns$.
Pushing the vertices of $\Y$ along the cut $\uc$ to points in $\Tri$,
one gets a convex hull $\Core^0=\imc(\Core(\phi))$ on $\surfo$ from the core $\Core(\phi)$.
The convexity of $\Core^0$ on $\LS$ follows from \eqref{eq:closed}:
    \begin{itemize}
      \item Take two saddle connections, say $\we_i$ in Figure~\ref{fig:cut and lift},
      that intersect at a decoration $Z$.
      Then the angle between them that does not intersect the cut, say from $\we_1$ to $\we_2$,
      is at most $\gldim(\ns)\cdot\pi$.
      This implies the other angle, e.g. $\theta$ from $\we_2$ to $q(\we_1)$
      in Figure~\ref{fig:cut and lift}, is at least
      $$\Re(s)\cdot\pi-\gldim(\ns)\cdot\pi\ge\pi.$$
    \end{itemize}

Now we shall use this convex hull to build a $q$-quadratic differential on $\LS$.
Consider the horizontal strip decomposition of $\surf$ with respect to $\phi$.
  Delete all the upper half planes that do not intersect the core $\Core(\phi)$.
  When pushing closed marked point in $\Y$ along the cut $\cut$ to decorations in $\Tri$,
  the foliation of the remaining of the horizontal strip decomposition provides
  a partial foliation/grading of the zero sheet $\surfo^{(0)}\subset\LS$.
  Note that this partial foliation contains $\Core^0$ mentioned above.

The partial foliation for the $m$-th sheet $\surfo^{(m)}$
    is given in the same way from the horizonal strip decomposition on $\surf$
    with respect to the rotated quadratic differential
    $e^{m \ii \pi s}\cdot\phi.$
    Then it contains the corresponding convex hull $\Core^m$
    obtained from $\Core(e^{m \ii \pi s}\cdot\phi)$.

The rest of the foliation on $\LS$ will be given/filled-in by part of the $\log$-surface
  $C_\infty^s=s\log z$ around each decoration $Z$ and by (non-essential) upper half planes.
  Here, $C_\infty^s=s\log z$ is topologically homeomorphic to $C_\infty$ in \eqref{eq:C_infty}.
  This can be done, i.e., the partial foliations do match because of
  the condition \eqref{eq:closed}:
  \begin{itemize}
  \item Each angle of $\Core^m$ between saddle connections at a decoration $Z$
    is at most $\gldim\ns\cdot\pi$.
    Therefore
    (the neighbourhood of) the cores $\Core^m$ around $Z$, do not overlap in $C_\infty^s$.
  \item These foliations are compatible, i.e.
    they can be regarded as part of $C_\infty^s$ simultaneously, because by construction
    the deck transformation $q$ becomes the rotation by $e^{m \ii \pi s}$.
  \item Thus we can fill in the gaps to get a foliation on $\LS$.
  \end{itemize}
Finally, it is straightforward to check that
the foliation we just constructed is indeed a $q$-quadratic differential on $\LS$ in Lemma~\ref{lem:log}.
This completes the proof that $\Xi_s$ is surjective.
\end{proof}

Now take a subspace $\QQuad^*_s(\LS)$ in $\QQuac_s(\LS)$
which consists of those $q$-quadratic differentials,
whose compatible cuts are in $\{b(\uc)\mid b\in\BT(\surfo)\}$.
This is the analogue of taking a connected component
$\FQuad_3^{\kong{T}}(\surfo)$ in $\FQuad(\surfo)$ in \cite[(4.13)]{KQ2},
where they fit in the commutative diagram
\begin{equation}\label{eq:quad0}
\begin{tikzpicture}[xscale=1.6,yscale=0.8,baseline=(bb.base)]
\path (0,1) node (bb) {}; % baseline
\draw (0,2) node (s0) {$\FQuad_3^{\kong{T}}(\surfo)$}
 (0,0) node (s1) {$\FQuad_3(\surf)$}
 (2.5,1) node (s2) {$\FQuad_3(\surf)$};
\draw [->, font=\scriptsize]
 (s0) edge[right hook-stealth](s1)
 (s0) edge node [above] {$\BT(\surfo)$} (s2)
 (s1) edge node [below] {$\SBr(\surf)$} (s2);
\end{tikzpicture}
\end{equation}
in that case (Calabi-Yau-3).
As the subgroup $\ST\DX$ of $\SBr(\surfo)$ is isomorphic to $\BT(\surfo)$ by Theorem~\ref{thm:IQZ0},
we have the following result,
which is the Calabi-Yau-$s$ version of Theorem~\ref{thm:BSKQ},
or the $q$-deformed version of Theorem~\ref{thm:HKK}.

\begin{theorem}\label{thm:main}
There is the following identification
\begin{gather}\label{eq:chis}
    \chi_s\colon \HHo{}(\LS) \to \Grot(\DX).
\end{gather}
Moreover,
there is an isomorphism $\Xi_s$ of complex manifolds that fits into the commutative diagram
\begin{gather}\label{eq:QQ=QS2}
\xymatrix@C=6pc{
    \QQuac_s(\LS) \ar[r]^{\Xi_s} \ar[d]_{\Pi_s} & \CStab_s\DX
        \ar[d]^{\hh{Z}_s} \\
    \Hom_R(  \HHo{}(\LS),\bC_s) \ar[r]^{ \chi_* } &
    \Hom_R(  \Grot(\DX),\bC_s),
}\end{gather}
such that the period map $\Pi_s$ is given by
\[
    \Pi_s=q_s\circ (\Pi_\infty\otimes 1)\colon\xi\mapsto\displaystyle\int\sqrt{\xi}
\]
via \eqref{eq:otimes} and becomes the map $\hh{Z}_s= q_s\circ (Z_\infty\otimes 1)$ (i.e., central charge).
More precisely, \eqref{eq:formula} holds for $X=\Xxx$ in \eqref{eq:X lift} of Theorem~\ref{thm:IQZ0}.
Here $\Pi_\infty$ and $Z_\infty$ are the ones in \eqref{eq:i infinity}.

Furthermore, for $(\sigma,s)=\Xi_s(\xi)$, the map $\Xxx$ induces one-one correspondence
between saddle connections of $\xi$ on $\LS$ and $\sigma$-semistable objects in $\DX$.
\end{theorem}
\begin{proof}
Firstly, \eqref{eq:chis} follows from \eqref{eq:i infinity} by tensoring $R$,
noticing \eqref{eq:KKK} and \eqref{eq:otimes}.

Secondly, \eqref{eq:QQ=QS2} is the lifted version of \eqref{eq:xis} in Proposition~\ref{pp:QS=QQ2},
noticing that we have \eqref{eq:BT=ST}.

Finally, suppose that $(\sigma,s)$ is induced from the triple $(\DI, \ns, s)$ up to the action of $\ST\DX$
and $\phi=\Xi_\infty^{-1}(\ns)$ as in the proof of Proposition~\ref{pp:QS=QQ2}.
By Theorem~\ref{thm:imc},
%\[\hh{L}_\TT \big(  \Xinf(\wg)  \big)= \Xxx\big( \imc(\wg) \big)\]
we deduce that the correspondence between $\sigma$-semistable objects and saddle connections of $\xi$
also inherits from the correspondence between the $\ns$-semistable objects and saddle connections of $\phi$
in Theorem~\ref{thm:HKK}.
\end{proof}
In particular, Conjecture~\ref{conj:period} holds restricted to $\QQuad^*_s(\LS)$.

%=========================================================
\subsection{Existence of cuts}
%=========================================================
One technical issue is that if the compatible cuts always exist for a $q$-quadratic differential.
In other words, such a existence is equivalent to the assumption below.
\begin{assumption}\label{ass:cuts}
$\QQuac_s(\LS)=\QQuad_s(\LS)$.
\end{assumption}
We present some sufficient conditions for the existence of cuts in this subsection.

First we have the following fact.
\begin{lemma}\label{lem:core}
The projection $\pi_\Tri(Z)$ of any zero/decoration $Z$ of $\xi$ on $\surfo$ is in the boundary of $\core(\xi)$.
\end{lemma}
\begin{proof}
Up to rotation, we only need to consider the saddle-free case.
Suppose that a zero $Z$ is in the interior of $\core(\xi)$.

Let $K\gg1$ be a positive integer.
Consider all saddle connections starting from $Z$ that are in sheets $\bigcup_{i=1}^K\surfo^i$.
Let $l_i$ be the number of saddle connections in the sheet $\surfo^i$.
As mentioned above, $\xi$ satisfies \eqref{eq:GG0} and
hence the deck transformation $q$ preserves geodesics/saddle connections.
Thus $l_i$ are the same, denoted by $l$.

On the one hand, each of these saddle connections is a horizontal strip $H$,
such that $H$ is incident at $Z$ with angle $\pi$.
Thus the sum of angles at $Z$ of horizontal strips that contains these saddle connections satisfies
\[
    K\cdot l\cdot \pi <  K\cdot \Re(s)\cdot\pi + 2\pi,
\]
where $\Re(s)\cdot\pi$ is the angle at $Z$ for each sheet.
Dividing by $K$ of the inequality above and taking $\lim_{K\to+\infty}$
we see that $l\le\Re(s)$.

On the other hand, the angle between any two adjacent saddle connections at $Z$ is less than $\pi$,
since $\Core(\xi)$ is convex and $Z$ is in the interior of $\core(\xi)$.
So the sum of angles at $Z$ of horizontal strips that contains these saddle connections satisfies
\[
    2\pi+(K\cdot l-1)\cdot\pi>K \cdot \Re(s)\cdot\pi.
\]
Similarly, we deduce that $l\ge\Re(s)$, which forces $l=\Re(s)$.
However, this will further force that
the angle between any two adjacent saddle connection at $Z$ equals $\pi$, which contradicts
to the fact that $Z$ is in the interior of $\core(\xi)$.
\end{proof}

Now we show that Assumption~\ref{ass:cuts} always holds in the disk case.

\begin{corollary}\label{cor:zero}
If the genus $g$ of $\surf$ is zero, then Assumption~\ref{ass:cuts} holds.
\end{corollary}
\begin{proof}
By Lemma~\ref{lem:core}, any decoration $Z\in\Tri$ is in the boundary of $\core(\xi)$,
so that there is no obstruction to pair points in $\Tri$ and $\Y\in\partial\surfo$.
Hence the corollary holds.
\end{proof}

Also, when $\Re(s)\gg1$, then Assumption~\ref{ass:cuts} holds.
More precisely, we have the following.

\begin{proposition}
Recall that $(\uk,\ul):=\{(k_1,l_1),\dots,(k_b,l_b)\}$. If
\begin{gather}\label{eq:sgg1}
    \Re(s)\geq\max_{1\le i\le b} \{3-k_i-l_i\},
\end{gather}
then Assumption~\ref{ass:cuts} holds.
\end{proposition}
\begin{proof}
We use the winding number and a combinatorial fact to prove the proposition.

\subsubsection*{\textbf{Step 1}}
Since the boundaries of $\LS$ contract to the core $\Core(\xi)$,
the boundaries of $\surfo$ contract to $\core(\xi)$.
If we cut $\surfo$ along $\core(\xi)$,
we will get exactly $b$ annulus $\kong{A}_i$,
each of which contains a boundary component $\partial_i$ of $\surfo$.
Equivalently, we have
$$\surfo\setminus\core(\xi)\cong \bigcup_{i=1}^b\kong{A}_i, \quad \partial_i\subset\partial\kong{A}_i,$$
where $\kong{A}_i$ can intersect only at boundaries:
$$\kong{A}_i^\circ\cap \kong{A}_j^\circ=\emptyset, \quad i\ne j.$$
Suppose that $\partial_i$ is the real blow-up of a pole $p_i$ of type $(k_i,l_i)$
and $Z_1,\ldots, Z_k$ are the $s$-simple zeroes on the other boundary $C_i(\ne \partial_i)$ of $\kong{A}_i$.

Let $L$ be the unique simple closed curve in $\kong{A}_i$ (up to isotopy), clockwise around $C_i$.
The winding number of $L$, as an anticlockwise loop around $\partial_i$ is
\[
    \wind{}(L)=-\wind{\xi}(\partial_i)=k_i(\Re(s)-2)+l_i-2.
\]
On the other hand, such a winding number can be calculated as a clockwise loop around $C_i$:
\begin{itemize}
  \item boundary arcs (connecting zeroes) of $C_i$ are (images of) saddle connections
    (with zero winding) and;
  \item around each $s$-simple zero $Z_j$, the winding number is $\Re(s)$,
    which implies that $Z_j$ contributes the winding of $L$ is
    \[
        \wind{}(L;Z_j)=\alpha-1<\Re(s)-1,
    \]
    for $0<\alpha<\Re(s)$ is the angle at the corner $Z_j$ of $C_j$ shown in Figure~\ref{fig:wind-1}.
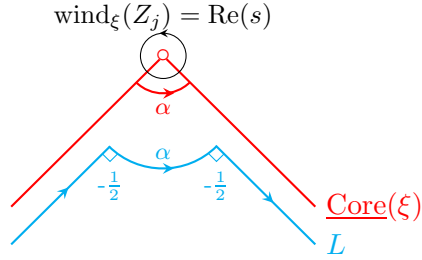
\begin{figure}[ht]\centering
\begin{tikzpicture}[]
\draw[red,thick](0+1,0+.5+1)to(3,3+.5)to(6-1,0+.5+1)node[right]{$\core(\xi)$};

\draw[\separated,thick,-<-=.5,>=stealth] (6-1,0+1)node[right]{$L$}to(3+.69,3-.69);
\draw[\separated,thick,-<-=.5,>=stealth]
    (3+.7,3-.7) arc(-45:-135:1);
\draw[\separated,thick,-<-=.5,>=stealth] (3-.7,3-.7)to(0+1,0+1);

\draw[\separated](3+.7+.1,3-.7-.1)to(3+.7,3-.7-.2)node[below]{\footnotesize{-$\frac{1}{2}$}}to(3+.7-.1,3-.7-.1);
\draw[\separated](3-.7+.1,3-.7-.1)to(3-.7,3-.7-.2)node[below]{\footnotesize{-$\frac{1}{2}$}}to(3-.7-.1,3-.7-.1);

\draw[red,thick,-<-=.5,>=stealth]
    (3+.35,3+.5-.35) arc(-45:-135:.5);
\draw[red](3,3)node[below]{\footnotesize{$\alpha$}};
\draw[\separated](3,2)node[above]{\footnotesize{$\alpha$}};

\draw(3,3+.7)node[above]{\small{$\wind{\xi}(Z_j)=\Re(s)$}}
(3,3+.5)node[white] {$\bullet$} node[red]{$\circ$} circle(.3);
\draw[->,>=stealth] (3,3+.8)to(3-.05,3+.8);
\end{tikzpicture}
\caption{Winding contribution $1+\alpha$ around a corner $Z_j$ of $C_j$}
\label{fig:wind-1}
\end{figure}

\end{itemize}
Thus, we have the estimation  that
$$\wind{}(L)=\sum_{j=1}^k \wind{}(L;Z_j)<k\cdot(\Re(s)-1)$$
or
\begin{gather}\label{eq:calculation wind}
  k>k_i+\frac{k_i+l_i-2}{\Re(s)-1}.
\end{gather}
By \eqref{eq:sgg1}, we deduce that $k>k_i-1$ or $k\ge k_i$.

\subsubsection*{\textbf{Step 2}}
Consider a bipartite (unoriented) graph $G$,
consisting of white vertices $W_Z$ labelled by $s$-simple zeroes $Z$ in $\Tri$
and black vertices $B_{\partial_i}$ labelled by boundary components $\partial_i$ of $\surfo$.
There is an edge in $G$ connecting the white vertex $W_Z$
and the black vertex $B_{\partial_i}$ if and only if
$Z\in\Tri$ are in the boundary of $\kong{A}_i$.
We associate a number $k_i$ to $B_{\partial_i}$,
where $(k_i, l_i)$ is the type of the corresponding $s$-pole of $B_{\partial_i}$.
We have the following conditions on $G$:
\begin{itemize}
\item \textbf{Step I} says that the valence of $B_{\partial_i}$ is at least $k_i$;
\item Lemma~\ref{lem:core} says that the valence of any $W_Z$ is at least 1.
\end{itemize}
We claim that for any bipartite graph $G$ with the conditions above,
there is a matching $E$ of $G$, consisting of $\aleph=\sum_{i=1}^bk_i$ edges of $G$,
such that each white vertex $W_Z$ is incident in exactly one edge of $E$
and each black vertex $B_{\partial_i}$ is incident in exactly $k_i$ edges of $E$.

Use induction on the number of edges in $G$, where the start step is trivial.
For the inductive step, there are two cases:
\begin{itemize}
\item If there is a white vertex $Z$ with valence $1$,
let $B_{\partial_i}$ be the only black vertex connected directly to $Z$ (by an edge $e$).
Then delete $Z$ and reduce the associated number $k_i$ of $B_{\partial_i}$ to $k_i-1$
to get a graph $G'$ which also satisfies the conditions above.
By inductive assumption, there is a matching $E'$ for $G'$
and then $E'\cup\{e\}$ is a matching for $G$.
\item Otherwise, the valence of each white vertex $Z$ is at least 2.
The sum of valences of all white vertices is at least
$$2\aleph=2\sum_{i=1}^b k_i,$$
which is also the sum of valences of all black vertices.
Then we deduce that the valence $v_j$ of some black vertex $B_{\partial_i}$ is strictly bigger than
the associated number $k_i$.
Delete any edge incidents at $B_j$ from $G$ and we can apply the inductive assumption  again.
\end{itemize}

\subsubsection*{\textbf{Step 3}}
Combining the two facts above, the matching tells
how to pair $k_i$ $s$-simple zeroes to the boundary $\partial_i$,
which then induces a cut (by connecting them in each of the annulus $\kong{A}_i$).
In fact, the cut can be then chosen to be images of geodesics (`separating connections').
\end{proof}

%=========================================================
\subsection{Conclusion}
%=========================================================
Up to the technical issue, i.e., Assumption~\ref{ass:cuts} above,
the spaces of induced open and closed $q$-stability conditions coincide,
which provide the `principal' components of the space $\QStab_s\DX$ of $q$-stability conditions.

\begin{theorem}\label{thm:q=x}
Suppose that Assumption~\ref{ass:cuts} holds (e.g. when $\Re(s)\gg1$) and $\Re(s)>2$.
Then we have the following
\begin{enumerate}
\item $\OStab_s\DX=\CStab_s\DX$. Thus,
\begin{gather}\label{eq:main}
    \QQuad_s(\LS)\cong\OStab_s\DX.
\end{gather}
\item The image of the inclusion $$\OStab_s\DX\hookrightarrow\QStab_s\DX$$
is open and closed and hence consists of connected components.
\end{enumerate}
We will write $\QStap_s\DX$ for $\OStab_s\DX$ in this case.
\end{theorem}
\begin{proof}
Take a closed $q$-stability condition $(\sigma,s)$ in $\CStab_s\DX$,
which is induced from some triple, i.e.
$$(\sigma,s)=(\DI, \ns,s)\otimes_*R$$
with $\gldim\ns=\Re(s)-1$ (here we take the canonical $\XX$-baric heart without loss of generality).
We will keep all the notations in Proposition~\ref{pp:QS=QQ2},
i.e.
\begin{itemize}
\item $\xi=\Xi_s^{-1}(\sigma,s)$ with a compatible cut $\cut$ and the corresponding
(exponential type) quadratic differential
$$\phi=\Xi_\infty^{-1}(\ns)\in\FQuad_\infty(\gms),$$
cf. \eqref{eq:cx xi} and \eqref{eq:ns xi}.
\end{itemize}

For the first statement, we need to show that there exists another compatible cut $\cut_0$
such that the corresponding quadratic differential $\phi_0\in\FQuad_\infty(\gms)$
and stability condition $\ns_0=\Xi_\infty(\phi_0)$ satisfying
\[
    (\sigma,s)=(\hh{L}_0(\DI), \ns_0,s)\otimes_{\oplus}R \quad\text{with}\quad
    \gldim\ns_0<\Re(s)-1.
\]
We divide the proof of this statement into three steps.
\begin{description}
\item[Step 1]
Consider a pair of semistable objects $\widehat{M}_1$ and $\widehat{M}_2$ in $\DI$ such that
\begin{gather}\label{eq:=gldim}
\begin{cases}
    \Hom_{\D_\infty}(\widehat{M}_1,\widehat{M}_2)\neq0,\\
    \varphi_{\ns}(\widehat{M}_2)-\varphi_{\ns}(\widehat{M}_1)=\gldim\ns.
\end{cases}
\end{gather}
Let $\wg_i$ be the closed arcs on $\surf$ that correspond to $\widehat{M}_i$ in $\DI$.
As in the proof of Proposition~\ref{pp:QS=QQ2},
they intersect at some $Y\in\Y$, such that
the angle $\alpha=\gldim\ns\cdot\pi$ at $Y$ provides the non-zero morphism in \eqref{eq:=gldim},
cf. Figure~\ref{fig:cut and lift}.
By \cite[Prop.5.7]{Q7}, there are only finitely many such angles $\alpha$.

\item[Step 2]
Pulling along the cut $\cut$, the arcs $\wg_i$ become $\we_i=\imc(\wg_i)$ on $\surfo$ as before.
Denote by $Z\in\Tri$ the decoration corresponding to $Y$ and
$\beta$ the angle corresponding to $\alpha$.
So we have
\[
    \beta=\alpha=\gldim\ns\cdot\pi=(\Re(s)-1)\cdot\pi,
\]
which implies the other angle $\theta$ at $Z$ between $\we_2$ and $q(\we_1)$ is exactly $\pi$.
This also means that $\we_i$ are boundary arcs of
(the zero sheet of) the core $\Core^0(\xi)$.

\item[Step 3]
Now we can perturb the quadratic differential $\xi$ a little bit to get a new one $\xi_0$,
such that, the only changes happen at all those angles $(\beta,\theta)$ as in \textbf{Step 1},
namely, we require for each of such angles,
\[
    \beta_0=\beta-\epsilon,\quad \theta_0=\theta+\epsilon=\pi+\epsilon
\]
for some small positive real number $\epsilon$.
By Assumption~\ref{ass:cuts}, $\xi_0$ admits a compatible cut $\cut_0$,
which is also compatible with $\xi$ then.

Use the cut $\cut_0$ to produce another quadratic differential $\phi_0\in\FQuad_\infty(\gms)$
that corresponds to another stability condition $\ns_0$ in $\Stab\DI$.
Note that $\cut_0$ will cut the angle $\beta$ in \textbf{Step 2}.
By \cite[Prop.5.7]{Q7},
$$\gldim\ns_0<\gldim\ns=\Re(s)-1.$$

With the Lagrangian immersion $\hh{L}_0$ that corresponds to $\cut_0$,
we obtain another triple to induce $(\sigma,s)$:
\[
    (\sigma,s)=(\hh{L}_0(\DI), \ns_0,s)\otimes_{\oplus}R,
\]
as required.
\end{description}

For the second statement, we note that $\OStab_s\DX$ is open $\QStab_s\DX$
since condition \eqref{eq:open} is an open one (cf. \cite[Thm.~5.11]{IQ1}).
For the closedness in $\QStab_s\DX$,
one can use the analogue map $K$ in \cite[\S~11.5]{BS}, as we have shown that
these stability conditions are indeed coming from quadratic differentials.
\end{proof}

%=========================================================
\part{Applications}\label{part:app}
%=========================================================
\section{Bridgeland-Smith theory on Calabi-Yau-$N$ categories}\label{sec:BS-N}
We keep all the settings at the beginning of Section~\ref{sec:main}.
%=========================================================
\subsection{$N$-reductions for categories}
%=========================================================
Suppose that $N\geq\gldim\ha_\TT+1$
and $\Gamma_\TT^N$ is the Calabi-Yau-$N$ Ginzburg algebra
obtained from $\GAX$ by collapsing the double degree $a+b\XX$ into a single degree $a+bN$.
Then there is a projection between dg algebras
$$\pi_N\colon\GAX\to\GAN,$$
which induces a fully faithful exact functor
\[
    \pi_N\colon\DX/[\XX-N]\to\DN\,(\colon=\D_{fd}(\Gamma_\TT^N)),
\]
that makes $\DN$ the $N$-reduction of $\D_{\XX} \sslash [\XX-N]$,
in the sense that
\begin{itemize}
\item $\DN$ can be identified with a triangulated hull of $\DX/[\XX-N]$,
\item $K(\D_N) \cong \ZZ^{\oplus n}$ and the induced $R$-linear map
\[
    [\pi_N] \colon K(\DX) \to K(\DN)
\]
is a surjection given by sending $q \mapsto (-1)^N$.
\end{itemize}
Thus, we have the following by \cite[Thm.~4.5]{IQ1}.
\begin{corollary}\label{cor:embed}
There is an embedding
$$\QStab_N \DX \longrightarrow \Stab\DN$$ between complex manifolds.
Denote by $\Stap\DN$ the image of $\CStab_N\DX$.
\end{corollary}

We also give the following conjecture,
which is the surface version of \cite[Cor.~6.8]{IQ1}.
\begin{conjecture}\label{conj:ses}
The cluster-$\XX$ category $\CXT$ in \eqref{eq:CXT}
admits the Serre functor $\tau\circ[1]$ (which should correspond to $[\XX-1]$),
for $\tau$ the Auslander-Reiten functor.
Moreover, the Calabi-Yau-$(N-1)$ cluster category $\hh{C}_{N-1}(\TT)$ of $\GAN$
is the (unique) triangulated hull of $\CXT\sslash\tau[2-N]$
that fits into the commutative diagram:
\begin{equation}\label{eq:cses}
    \xymatrix{
    0 \ar[r] &
        \D_{fd}(\GAX) \ar[r]\ar[d]^{_{\sslash[\XX-N]}} &
        \per\GAX \ar[r]\ar[d]^{_{\sslash[\XX-N]}} &
        \CXT \ar[r]\ar[d]^{_{\sslash\tau[2-N]}} &0 \\
     0 \ar[r] & \D_{fd}(\GAN) \ar[r]& \per\GAN \ar[r] & \hh{C}(\GAN) \ar[r] &0.
}\end{equation}
\end{conjecture}

\begin{remark}
One interesting question is what is the condition for $\DX\sslash[\XX-N]$ to be $N$-reductive.
The condition $N\geq\gldim\ha_\TT+1$ here is actually for $\GAN$ being non-positive,
so that it is a silting object in its perfect derived category and hence
ensures $\Stab\DN$ is nonempty.
\end{remark}

%=========================================================
\subsection{$N$-reductions for $q$-quadratic differentials}
%=========================================================
When $s=N\ge3$ is a positive integer, then the $q_s$-quadratic differential $\xi$,
locally of the form \eqref{eq:xi},
becomes a classical type quadratic differential on some Riemann surface in Definition~\ref{def:cqd}.
We have the definition of generalized GMN differentials as follows.
\begin{definition}\label{def:CYN GMN}
A \emph{CY-$N$ type differential} on a Riemann surface $\rs$ is a meromorphic differential
with classical type of singularities satisfying:
\begin{itemize}
\item all the zeroes have order $N-2$;
\item the order of any pole is at least 3.
\end{itemize}
The (decorated) marked surface associated to $(\rs,\phi)$ is
defined as in Definition~\ref{def:GMN}.
\end{definition}
Next, from a graded marked surface $\gms$ with associated graded DMS $\surfo$,
we construct the marked surface $\surfo^N$ for $q_N$-quadratic differentials.

\begin{definition}
The \emph{(CY-$N$) marked surface} $\surf^N$ is a marked surface modified from $\gms$ by
\begin{itemize}
  \item forgetting the grading and closed marked points and
  \item changing the number of open marked points on the boundary $\partial_i$
  to be
  \begin{equation}\label{eq:number}
    k_i\cdot (N-2)+l_i-2,
  \end{equation}
  where $(k_i,l_i)$ is the associated numerical data.
\end{itemize}
The \emph{(CY-$N$) DMS} $\surfo^N$ is modified from $\surfo$ same as above.
\end{definition}
Note that the condition \eqref{eq:higher} in Definition~\ref{def:BS_quad} that
$s$-poles are higher $s$-poles ensures the number in \eqref{eq:number} is a positive integer.

Similar to Definition~\ref{def:SFquad},
we have moduli spaces $\FQuad_N(\surfo^N)$ of $\surfo^N$-framed quadratic differentials.
%The $\surf^N$-framed version is just
%$$\FQuad_N(\surf^N)=\FQuad_N(\surfo^N)/\MCG(\surfo^N,\partial\surfo^N).$$
So we have the following straightforward observation, as the counterpart to Corollary~\ref{cor:embed}.
\begin{lemma}\label{lem:embed}
There is an embedding
\[
    \QQuad_N(\LS) \longrightarrow \FQuad(\surfo^N).
\]
Denote its image by $\FQuad_N^\circ(\surfo^N)$ and the corresponding quotient space in $\FQuad_N(\surf^N)$
by $\FQuad_N^\circ(\surf^N)$.
\end{lemma}

Combining Corollary~\ref{cor:embed} and Lemma~\ref{lem:embed}
we have the following version of Theorem~\ref{thm:q=x}
as an $N$-analogue of Bridgeland-Smith's result (or King-Qiu's upgraded version).
The condition $N\gg 1$ ensures that Assumption~\ref{ass:cuts} holds.
But one expects that this is not necessary: condition \eqref{eq:N} is more reasonable.

\begin{theorem}\label{thm:BS N}
If $N\gg1$, there is an isomorphism between complex manifolds
\begin{gather}\label{eq:BS N}
    \FQuad_N^\circ(\surfo^N)\cong\Stap\D_{fd}(\Gamma_\TT^N).%/\ST(\Gamma_\TT^N).
\end{gather}
\end{theorem}

When $\surf$ is the disk (i.e., type A case),
the theorem above is the main theorem in \cite{I}.
However, notice that the method in \cite{I} is a directly generalization of \cite{BS}
while our approach is different.

\begin{remark}
Let us explain that in the unpunctured case of \cite{BS}, i.e., $N=3$ and
when the marked surface has no punctures (equivalent to the higher order pole condition),
how the result above should reproduce \eqref{eq:BS}.
To see this, let $\mathbf{K}=(K_1,\cdots, K_b)$ be the numerical data of numbers of marked points
on $\partial\surf$ in the setting of \cite{BS}.
Recall that $g$ and $b$ are the genus and number of boundaries components of $\surf^3$.
One chooses a graded marked surface $\gms$ of genus $g$,
with integers $\ul=(l_1\cdots,l_b)$ that sums to $4-4g$
and a set of integer $\uk=(k_1\cdots,k_b)$ for $k_i=K_i-l_i+2$ as in \eqref{eq:number}.
Here we need $k_i\ge1$, which can be done since
\begin{itemize}
\item if $g\ge1$ or $b\ge2$, then
$$\sum_{i=1}^b k_i=\sum_{i=1}^b K_i+(4g-4)+2b \ge b;$$
\item if $g=0$ and $b=1$, then $K_1\ge3$ (so that $\Delta\ge1$ in \eqref{eq:Delta})
and the estimation above also holds.
\end{itemize}
Then \eqref{eq:BS N} in Theorem~\ref{thm:BS N} indeed becomes \eqref{eq:BS} as required.
\end{remark}

We hope to cover the punctured case of \cite[Theorem~1.2]{BS} in future works.

\begin{example}
Let $\surf=\dd$ be a disk with 4 marked points on its boundary.
So the numerical data is $(g=0,\;b=1;\;\uk=(4),\;\ul=(4);\, \LP_0=\emptyset)$.

Choose the full formal open arc system $\TT$ with dual closed arc system $\TT^*$ of $\gms$
as in the first picture in Figure~\ref{fig:ex1}.
\begin{figure}[ht]\centering
\begin{tikzpicture}[scale=.36,rotate=0]
\draw[very thick](0,0)circle(6);
\foreach \j in {0,2,4}{\draw[very thick,cyan](45*\j+45:6)node{$\bullet$}
    to[bend left=-25](45*\j-45:6)node{$\bullet$};}
\draw[ultra thick,red](180:6)node[white]{$\bullet$}node[red]{$\circ$}to[bend left]
    node[above]{$1$}(-45-45:6);
\draw[ultra thick,red](0:6)node[white]{$\bullet$}node[red]{$\circ$}to[bend right]
    node[above]{$3$}(-45-45:6);
\draw[ultra thick,red](45+45:6)node[white]{$\bullet$}node[red]{$\circ$}to
    node[above left]{$2$}
    (-45-45:6)node[white]{$\bullet$}node[red]{$\circ$};
\begin{scope}[shift={(14,0)}]
\draw[very thick](0,0)circle(6);
\foreach \j in {0,2,4}{\draw[very thick,cyan](45*\j+45:6)node{$\bullet$}
    to[bend left=-25](45*\j-45:6)node{$\bullet$};}
\draw[ultra thick,red](180:5)node[white]{$\bullet$}node[red]{$\circ$}to[bend left](-45-45:3);
\draw[ultra thick,red](0:5)node[white]{$\bullet$}node[red]{$\circ$}to[bend right](-45-45:3);
\draw[ultra thick,red](45+45:5)node[white]{$\bullet$}node[red]{$\circ$}to
    (-45-45:3)node[white]{$\bullet$}node[red]{$\circ$};
\end{scope}
\end{tikzpicture}\qquad
\begin{tikzpicture}[scale=.36,rotate=60]
\foreach \j in {1,...,6}{\draw[cyan,very thick](120*\j-30:6)to(120*\j+90:6);}
\foreach \j in {1,...,3}{  \draw[ultra thick,red](30+120*\j:4)to(0,0);}
\foreach \j in {1,...,3}{  \draw(30-120*\j:4)node[white]{$\bullet$}node[red]{$\circ$};}
\draw(0,0)node[white]{$\bullet$}node[red]{$\circ$};
\foreach \j in {1,...,6}{\draw[very thick](60*\j+30:6)to(60*\j-30:6)node{$\bullet$};}
\end{tikzpicture}
\caption{The $A_3$ case: $\gms$, $\surfo$ and $\surfo^3$ with arc systems}\label{fig:ex1}
\end{figure}
Note that the open arcs are drawn in blue while the closed ones are in red.
Suppose that the intersection indexes between open arcs in $\TT$ are all zero.
Then the corresponding dual closed arc systems $\TT^*_\Tri$ in $\surfo$ and $\surfo^3$
are shown in the second and third pictures in Figure~\ref{fig:ex1}.
Then we could have the corresponding graded quivers as:
\begin{gather}\label{eq:ex1}
\qquad\xymatrix@R=3.2pc@C=2pc{ &{\color{red}2}
    \ar@{<-}@<0ex>[dl]^{0} \\ {\color{red}1}&& {\color{red}3}\ar@{<-}@<0ex>[ul]^{0}
          \ar@{<-}@<0ex>[ll]^{-1}
}\quad
\xymatrix@R=3.2pc@C=2pc{ &{\color{red}2}  \ar@{<-}@(ur,ul)[]_{1-\XX}
    \ar@{<-}@<.8ex>[dr]^{2-\XX}\ar@{<-}@<.2ex>[dl]^0
      \\ {\color{red}1}\ar@{<-}@(l,d)[]_{1-\XX}\ar@{<-}@<.2ex>[rr]^{3-\XX}\ar@{<-}@<.8ex>[ur]^{2-\XX}&&
        {\color{red}3}\ar@{<-}@(d,r)[]_{1-\XX}\ar@{<-}@<.2ex>[ul]^0
          \ar@{<-}@<.8ex>[ll]^{-1}
}\quad
\xymatrix@R=3.2pc@C=2pc{ &{\color{red}2}  \ar@{<-}@(ur,ul)[]_{-2}
    \ar@{<-}@<.8ex>[dr]^{-1}\ar@{<-}@<.2ex>[dl]^0
      \\ {\color{red}1}\ar@{<-}@(l,d)[]_{-2}\ar@{<-}@<.2ex>[rr]^{0}\ar@{<-}@<.8ex>[ur]^{-1}&&
        {\color{red}3}\ar@{<-}@(d,r)[]_{-2}\ar@{<-}@<.2ex>[ul]^0
          \ar@{<-}@<.8ex>[ll]^{-1}
}
\end{gather}
where the potentials associated to the second and third quivers are just the inner three cycles
of degree $3-\XX$ and degree $0$ respectively.
\end{example}

\begin{example}
Keep the surface $\surf=\dd$ as above.
Choose another full formal open arc system $\TT$ with dual closed arc system $\TT^*$ of $\gms$
as in the left picture in Figure~\ref{fig:ex2}.
\begin{figure}[h]\centering
\begin{tikzpicture}[scale=.36,rotate=135]
\draw[very thick](0,0)circle(6);
\foreach \j in {2,4}{\draw[very thick,cyan](45*\j+45:6)node{$\bullet$}
    to[bend left=-25](45*\j-45:6)node{$\bullet$};}
\draw[very thick,cyan](45*3:6)node{$\bullet$}
    to[bend left=-0](-45:6)node{$\bullet$};
\draw[ultra thick,red](180:6)node[white]{$\bullet$}node[red]{$\circ$}to[bend left]
    node[left]{$3$}(-45-45:6);
\draw[ultra thick,red](0:6)node[white]{$\bullet$}node[red]{$\circ$}
    to[bend right]node[below]{$\quad2$}(-45-45:6)node[white]{$\bullet$}node[red]{$\circ$};
\draw[ultra thick,red](90:6)node[white]{$\bullet$}node[red]{$\circ$}to[bend left=-30]
    node[right]{$1$}(0:6)node[white]{$\bullet$}node[red]{$\circ$};
\end{tikzpicture}\quad
\begin{tikzpicture}[scale=.36,rotate=135]
\draw[very thick](0,0)circle(6);
\foreach \j in {2,4}{\draw[very thick,cyan](45*\j+45:6)node{$\bullet$}
    to[bend left=-25](45*\j-45:6)node{$\bullet$};}
\draw[very thick,cyan](45*3:6)node{$\bullet$}
    to[bend left=-0](-45:6)node{$\bullet$};
\draw[ultra thick,red](180:5)node[white]{$\bullet$}node[red]{$\circ$}to[bend left](-45-45:5);
\draw[ultra thick,red](0:5)node[white]{$\bullet$}node[red]{$\circ$}
    to[bend right](-45-45:5)node[white]{$\bullet$}node[red]{$\circ$};
\draw[ultra thick,red](90:5)node[white]{$\bullet$}node[red]{$\circ$}to[bend left=-30]
    (0:5)node[white]{$\bullet$}node[red]{$\circ$};
\end{tikzpicture}\;
\caption{The $A_3$ case: alternative arc systems $\gms$ and $\surfo$}\label{fig:ex2}
\end{figure}
Then the corresponding dual full formal closed arc systems $\TT^*_\Tri$ in $\surfo$ is
shown in the right pictures in Figure~\ref{fig:ex2}.
The associated quivers $(Q_\TT,\widetilde{Q}_\TT)$ are
just the standard $A_3$ quiver
$1\xleftarrow{a}2\xleftarrow{b}3$ and its Calai-Yau-$\XX$ double respectively,
%\begin{gather}\label{eq:ex2}
%\qquad\xymatrix@R=3.2pc@C=2pc{ &{\color{red}2}
%    \ar@<0ex>[dl]^{0} \\ {\color{red}1}&& {\color{red}3}\ar@<0ex>[ul]^{0}
%}\quad
%\xymatrix@R=3.2pc@C=2pc{ &{\color{red}2}  \ar@(ur,ul)[]_{1-\XX}
%    \ar@<.8ex>[dr]^{2-\XX}\ar@<.2ex>[dl]^0
%      \\ {\color{red}1}\ar@(l,d)[]_{1-\XX}\ar@<.8ex>[ur]^{2-\XX}&&
%        {\color{red}3}\ar@(d,r)[]_{1-\XX}\ar@<.2ex>[ul]^0
%}
%\end{gather}
whose gradings %of quivers above in the Calabi-Yau-$\infty$ and -$\XX$ cases
are controlled by grading of curves.
%, not by the numbers of marked points on the boundary of $\surf$ (or $\surfo$).
We list three cases of the gradings of $Q_\TT$ with the corresponding 4-reduction pictures as follows:
\begin{itemize}
\item If $\deg a=\deg b=0$,
then the corresponding dual full formal closed arc system $\TT^*_\Tri$ in $\surfo^4$ is shown
in the left picture of Figure~\ref{fig:ex2'}.
\item If $\deg a=\deg b=1$,
then the corresponding dual full formal closed arc system $\TT^*_\Tri$ in $\surfo^4$ is shown
in the middle picture of Figure~\ref{fig:ex2'}.
\item If $\deg a=1$ and $\deg b=2$,
then the corresponding dual full formal closed arc system $\TT^*_\Tri$ in $\surfo^4$ is shown
in the right picture of Figure~\ref{fig:ex2'}.
\end{itemize}
\begin{figure}[h]\centering
\begin{tikzpicture}[scale=.36,rotate=135]
\path (-45-36:4) coordinate (v2) (-45+36:4) coordinate (v3)
    (-45-126:4.5) coordinate (v1) (-45+126:4.5) coordinate (v4);
\draw[ultra thick,red](v1)to(v2)to(v3)to(v4);
\foreach \j in {1,...,4}{\draw[red](v\j)node[white]{$\bullet$}node[red]{$\circ$};}
\draw[very thick,cyan](-45:6)to(135:6) (135+108:6)to(135:6) (135-108:6)to(135:6);
\foreach \j in {1,...,10}{\draw[very thick](-45+36*\j:6)to(36*\j-9:6)node{$\bullet$};}
\end{tikzpicture}
\begin{tikzpicture}[scale=.36,rotate=135]
\path (-45-36:4) coordinate (v2) (-45+36:4) coordinate (v3)
    (-45-90:4.5) coordinate (v1) (-45+90:4.5) coordinate (v4);
\draw[ultra thick,red](v1)to(v2)to(v3)to(v4);
\foreach \j in {1,...,4}{\draw[red](v\j)node[white]{$\bullet$}node[red]{$\circ$};}
\draw[very thick,cyan](-45:6)to(135:6) (135+108+36:6)to(135+36:6) (135-108-36:6)to(135-36:6);
\foreach \j in {1,...,10}{\draw[very thick](-45+36*\j:6)to(36*\j-9:6)node{$\bullet$};}
\end{tikzpicture}
\begin{tikzpicture}[scale=.36,rotate=135]
\path (-45-36-90:4) coordinate (v2) (-45+36:4) coordinate (v3)
    (-45-90+36:4.5) coordinate (v1) (-45+90:4.5) coordinate (v4);
\draw[ultra thick,red](v1)to(v2)to(v3)to(v4);
\foreach \j in {1,...,4}{\draw[red](v\j)node[white]{$\bullet$}node[red]{$\circ$};}
\draw[very thick,cyan](-45:6)to(135:6) (135+108+72:6)to(135+72:6) (135-108-36:6)to(135-36:6);
\foreach \j in {1,...,10}{\draw[very thick](-45+36*\j:6)to(36*\j-9:6)node{$\bullet$};}
\end{tikzpicture}
\caption{The $A_3$ case: three dual full formal arc systems in $\surfo^4$}\label{fig:ex2'}
\end{figure}
Note that the grading of quivers in the Calabi-Yau-$N$ case is
controlled by the numbers of marked points on the boundaries of $\surfo^N$.
In the final subsection of this section, we explain the cluster theory in the Calabi-Yau-$N$ case.
\end{example}
%=========================================================
\subsection{Cluster combinatorics}
%=========================================================
The skeletons (cf. \cite[Remark~2.17, Lemma~4.8 and Lemma~4.9]{KQ2}, see also \cite{Q2,BS})
for manifolds in \eqref{eq:BS N} should be
\begin{gather}\label{eq:CEG}
  \EG(\surf^N) \cong \operatorname{CEG}(\hh{C}_{N-1}(\TT)),
\end{gather}
where on the left hand side, it is the \emph{exchange graph of $N$-angulation} of $\surf_N$
and on the right hand side, it is the \emph{cluster exchange graph} of $\hh{C}_{N-1}(\TT)$.
The vertices of $\EG^{N}(\surf^N)$ are $N$-angulation of $\surf_N$ and
the edges are forward flips.
The vertices of $\operatorname{CEG}_{N-1}(\hh{C}(\TT))$ are
\emph{$(N-1)$-cluster tilting sets} (cf. e.g. \cite[Def.~4.1]{KQ1})
in the cluster categories $\hh{C}_{N-1}(\TT)$ mentioned in Conjecture~\ref{conj:ses}
with edges correspond to forward mutations.

For instance, Figure~\ref{fig:4-flips} shows a sequence of forward flips
for $N=4$-angulations of an 12-gon,
which corresponds to meromorphic quadratic differentials of singularity type
$$(2^5,-14)=(2,2,2,2,2,-14)$$ on
$\Pone=\mathbb{P}^1$.
More precisely, these 4-angulations consist of isotopy classes of separated saddle connections
of the corresponding meromorphic quadratic differentials.
Moreover, an example of a decorated forward flip for $4$-angulation is
from the middle picture of Figure~\ref{fig:ex2'} to the right picture there.
For more details, cf. \cite[\S~3]{KQ2}.
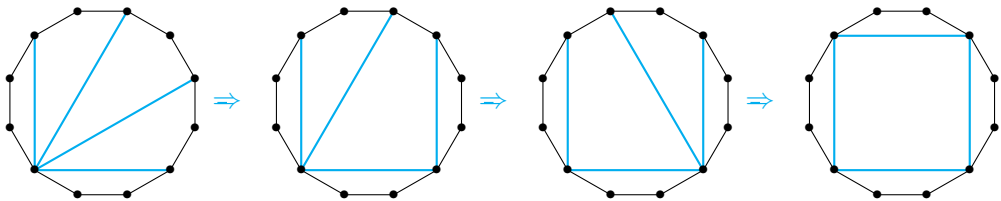
\begin{figure}[ht]\centering
  \begin{tikzpicture}[scale=.63]
    \draw[thick, cyan](360/12*7+15:2)to(360/12*4+15:2);
    \draw[thick, cyan](360/12*7+15:2)to(360/12*2+15:2);
    \draw[thick, cyan](360/12*7+15:2)to(360/12*0+15:2);
    \draw[thick, cyan](360/12*7+15:2)to(360/12*10+15:2);
    \draw[thick, cyan](0:2.4)to(0:2.6)node{$\Rightarrow$};
%    \draw(0,.6) node[red]{\tiny{$3$}}(.5,-.6) node[red]{\tiny{$2$}}
%    (-1.2,.6) node[red]{\tiny{$4$}}(0,-1.6) node[red]{\tiny{$1$}};
    \foreach \j in {1,...,12}  { \draw (360/12*\j+15:2)to(360/12*\j+360/12+15:2) node{\tiny{$\bullet$}};}
    \end{tikzpicture}
    \begin{tikzpicture}[scale=.63]
    \draw[thick, cyan](360/12*7+15:2)to(360/12*4+15:2);
    \draw[thick, cyan](360/12*7+15:2)to(360/12*2+15:2);
    \draw[thick, cyan](360/12*10+15:2)to(360/12*1+15:2);
    \draw[thick, cyan](360/12*7+15:2)to(360/12*10+15:2);
    \draw[thick, cyan](0:2.4)to(0:2.6)node{$\Rightarrow$};
    \foreach \j in {1,...,12}  { \draw (360/12*\j+15:2)to(360/12*\j+360/12+15:2) node{\tiny{$\bullet$}};}
    \end{tikzpicture}
    \begin{tikzpicture}[scale=.63]
    \draw[thick, cyan](360/12*7+15:2)to(360/12*4+15:2);
    \draw[thick, cyan](360/12*10+15:2)to(360/12*3+15:2);
    \draw[thick, cyan](360/12*10+15:2)to(360/12*1+15:2);
    \draw[thick, cyan](360/12*7+15:2)to(360/12*10+15:2);
    \draw[thick, cyan](0:2.4)to(0:2.6)node{$\Rightarrow$};
    \foreach \j in {1,...,12}  { \draw (360/12*\j+15:2)to(360/12*\j+360/12+15:2) node{\tiny{$\bullet$}};}
    \end{tikzpicture}
    \begin{tikzpicture}[scale=.63]
    \draw[thick, cyan](360/12*7+15:2)to(360/12*4+15:2);
    \draw[thick, cyan](360/12*4+15:2)to(360/12*1+15:2);
    \draw[thick, cyan](360/12*10+15:2)to(360/12*1+15:2);
    \draw[thick, cyan](360/12*7+15:2)to(360/12*10+15:2);
    \foreach \j in {1,...,12}  { \draw (360/12*\j+15:2)to(360/12*\j+360/12+15:2) node{\tiny{$\bullet$}};}
    \end{tikzpicture}
\caption{A sequence of flips of $4$-angulations}
\label{fig:4-flips}
\end{figure}

One can also associate Calabi-Yau-$N$ quivers with superpotential to $N$-angulations directly,
parallel to Definition~\ref{con:Qs}, where one replaces degree $\XX$ with $N$.
For instance, in figure~\ref{fig:W_T} which is the case when $N=5$,
the superpotential should contains all those 3-cycle of degree $-2$
(two of them are shown in the figure).
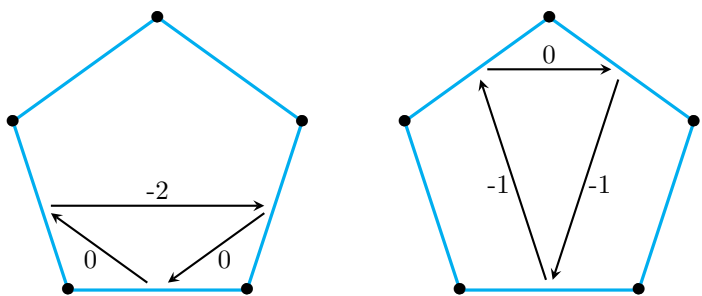
\begin{figure}[ht]\centering
  \begin{tikzpicture}[xscale=1]
    \foreach \j in {1,...,5}  {
        \draw[cyan, very thick]
            (72*\j+18:2) coordinate (v\j)node[black]{$\bullet$} --(72*\j+90:2)node[black]{$\bullet$};}
%    \draw[dashed,thick] (0,0) circle (2);
    \path (v1)--(v2) coordinate[pos=0.5] (x5)
              --(v3) coordinate[pos=0.5] (x1)
              --(v4) coordinate[pos=0.5] (x2)
              --(v5) coordinate[pos=0.5] (x3)
              --(v1) coordinate[pos=0.5] (x4);
    \foreach \j in {1,...,5}{\draw (x\j) node[red] (x\j){};}
    \draw[->,>=stealth,thick] (x3) to (x2);
    \draw[->,>=stealth,thick] (x2) to (x1);
    \draw[->,>=stealth,thick] (x1) to (x3);
    \foreach \j in {1,...,5}  { \draw (72*\j+18:2) coordinate (v\j);}
    \draw[](90+144:1.5)node{\small{0}}(90-144:1.5)node{\small{0}}(90:-.3)node{\small{-2}};
  \end{tikzpicture}
\qquad
  \begin{tikzpicture}[xscale=1]
    \foreach \j in {1,...,5}  {
        \draw[cyan, very thick]
            (72*\j+18:2) coordinate (v\j)node[black]{$\bullet$} --(72*\j+90:2)node[black]{$\bullet$};}
%    \draw[dashed,thick] (0,0) circle (2);
    \path (v1)--(v2) coordinate[pos=0.5] (x5)
              --(v3) coordinate[pos=0.5] (x1)
              --(v4) coordinate[pos=0.5] (x2)
              --(v5) coordinate[pos=0.5] (x3)
              --(v1) coordinate[pos=0.5] (x4);
    \foreach \j in {1,...,5}{\draw (x\j) node[] (x\j){};}
    \draw[->,>=stealth,thick] (x2) to (x5);
    \draw[->,>=stealth,thick] (x4) to (x2);
    \draw[->,>=stealth,thick] (x5) to (x4);
    \foreach \j in {1,...,5}  { \draw (72*\j+18:2) coordinate (v\j);}
    \draw[](90+144-36:.7)node{\small{-1}}(90-144+36:.7)node{\small{-1}}(90:1.5)node{\small{0}};  \end{tikzpicture}
\caption{Some 3-cycles in the CY-5 case.}
\label{fig:W_T}
\end{figure}

\begin{remark}\label{rem:final}
Finally, we remark that
one may want to get \eqref{eq:BS N} by generalizing the approach in \cite{BS},
as \cite{I} did for the type A case.
However, this is not a short cut by all means (nor easy) due to lack of the following results:
\begin{itemize}
  \item the topological cluster theory for surfaces, e.g.\eqref{eq:CEG}, has not been established.
  \item the categorical cluster theory for surfaces, e.g. \cite[Thm.~8.6]{KQ1} for $\GAX$, has not been established.
\end{itemize}
\end{remark}

%=========================================================
\section{Hurwitz spaces via quadratic differentials in genus zero}
%=========================================================
\label{sec:Hurwitz}
The aim of this section is to identify the moduli space of CY-$s$ type quadratic differentials
with the Hurwitz space in the case $g=0$. %we will consider the case when $g=0$.
%=========================================================
\subsection{Hurwitz spaces}\label{sec:HW}
%=========================================================
\begin{definition}
Let $\uk=(k_1,\dots,k_b) $ be a multi-set of positive integers $k_i \in \ZZ_{\ge 1}$.
A \emph{Hurwitz cover of genus $g$ and polar type $\uk$}
is a pair of
a compact Riemann surface $\rs$ of genus $g$ and
a meromorphic function $f \colon \rs \to\PP^1$ such that poles of $f$ consists of $b$ points,
i.e., $f^{-1}(\infty)=\{p_i\}_{i=1}^b$ with orders $\ord_f(p_i)=-k_i$.

Two Hurwitz covers $(\rs,f)$ and $(\rs^{\prime},f^{\prime})$ are equivalent
if there is a biholomorphism $h \colon \rs \to \rs^{\prime}$
such that $f=h^* f^{\prime} $.
The \emph{Hurwitz space $\HS(g,\uk)$} is the moduli space of Hurwitz covers
of genus $g$ and polar type $\uk$.
\end{definition}

For a Hurwitz cover $(\rs,f)$, denote by $\Zer(f), \Pol(f) \subset S$ the set
of zeroes and poles of $f$ respectively. Write $\Crit(f)=\Zer(f) \cup \Pol(f)$.
It is known that the dimension of the Hurwitz space is given by
\[
\dim_{\bC}\HS(g,\uk)=2g-2+b+\sum_{i=1}^b k_i,
\]
which will be denoted by $n$.

%A Hurwitz cover $(\rs,f)$ is called {\it simple} if all zeroes of $f$ are simple.
%The discriminant $\Sigma \subset \HS(g,\uk)$ is the subset consisting
%of Hurwitz covers with non-simple zeroes.
The \emph{regular locus} $\HS(g,\uk)_{\reg} \subset \HS(g,\uk) $ is defined to be
the subset consisting of Hurwitz covers with simple zeroes,
which will be called regular Hurwitz covers.
For $(\rs,f) \in \HS(g,\uk)_{\reg}$, denote $\Tri=\Zer(f)$.
Then since
\[
0=\sum_{p \in \Crit(f)}\operatorname{ord}_f(p)=|\Tri|-\sum_{i=1}^b k_i,
\]
we have $|\Tri|=\sum_{i=1}^b k_i$ in this case.

%=========================================================
\subsection{Marked surfaces and log surfaces from Hurwitz covers}
%=========================================================
\label{sec:log_HW}
In this subsection, we construct
(ungraded) marked surfaces and (ungraded) log surfaces from Hurwitz covers.

\begin{construction}
\label{con:MS}
For a Hurwitz cover $(\rs,f)$, we can associate a marked surface $\rsf=(\rsf,\M,\Y)$ as follows:
\begin{itemize}
\item A surface $\rsf$ is defined as the real blow up of
the underlying smooth surface of $\rs$ at the poles
$f^{-1}(\infty)=\{z_1,\dots,z_b\}$.
\item Define marked arcs $\overline{\M} \subset \partial \rsf$
as open intervals corresponding to asymptotic directions,
where $e^{f}$ decays exponentially as $z \to z_i$.
\item By shrinking each marked arc of $\overline{\M}$ to a single point in $\rsf$,
we obtain a set of closed marked points $\M$.
\item The set of open marked points in $\Y\subset\partial\rsf$ are the dual to $\M$,
i.e., they are alternative in $\partial\rsf$.
\end{itemize}
\end{construction}
Denote by $\partial_i$ the $S^1$-boundary component
obtained from the real blow up at $z_i$ for $i=1,\dots,b$.
Then by definition, there are $k_i=-\operatorname{ord}_f(z_i)$ points
on the boundary component $\partial_i$.
\begin{remark}
Though we can construct the marked surface $\rsf=(\rsf,\M,\Y)$ from the Hurwitz cover $(\rs,f)$,
this information is not enough to construct a grading $\grad$ on $\rsf$.
In Section~\ref{sec:framed_HS}, we introduce a primary quadratic differential in the case
when the genus of $\rsf$ is zero, which allows us construct a natural grading.
\end{remark}

Next we consider the construction of log surfaces from
Hurwitz covers.
Let $(\rs,f)$ be a regular Hurwitz cover.
Consider a multi-valued holomorphic
function $\log f$ on $\rsp:=\rs \setminus \Crit(f)$.
Take a base point $* \in \rsp$ and choose a value
$\log f (*)$.
First we note that the analytic continuation of $\log f(*)$
along with a cycle
$\gamma \in \pi_1(\rsp,*)$ is
given as the form $ \log f(*)+2\pi m$ for some $m \in \ZZ$.
Then we obtain a representation
\[
\rho \colon \pi_1(\rsp,*) \to \langle q \rangle,
\quad \gamma \mapsto \rho(\gamma):= q^m.
\]

\begin{definition}
\label{def:log}
The {\it logarithmic ($\log$) surface $\log \rsp$} associated to a
Hurwitz cover $(\rs,f)$
is defined to be the covering $\pi \colon \log \rsp\to \rsp$
corresponding to the representation
$\rho \colon \pi_1(\rsp,*) \to \langle q\rangle$.
The infinite cyclic group $\langle q\rangle$ acts on $\log S^{\circ}$
as the group of deck transformations.
\end{definition}
By definition, the multi-valued function $\log f$ can be
lifted on the log surface $\log \rsp$ as
a single-valued function through the analytic continuation. In other words,
the surface  $\log \rsp$ is the Riemann surface of the function $\log f$.
We also note that for any complex number $s \in \bC$,
the function
\[
f^{s-2}=e^{  (s-2)\log f     }
\] is also a
single-valued function on $\log \rsp$.

Finally we note that as in Section~\ref{sec:MC},
for a log surface $\log \rsp$
we can construct the topological log surface $\log \rs_{\Tri}^f$
by adding zeroes $\Tri$ of $f$ and
doing the real blow-ups at poles of $f$.

%=========================================================
\subsection{Primary quadratic differentials}
%=========================================================
In the following, we assume $g=0$, namely $\rs \cong \PP^1$, and fix the
coordinate $z \in \bC \cup \{\infty\}=\PP^1$.
%We will write $\Pone$ for $\PP^1$.

\begin{definition}
Let $\ul=(l_1,\dots,l_b) \in \ZZ^b$ be an ordered multi-set of $b$ integers satisfying
the condition $$\sum_{i=1}^b l_i=4.$$
For a Hurwitz cover $(\Pone,f)$ with poles
$f^{-1}(\infty)=\{z_1,\dots,z_b\}$, %indexed by $\{1,\dots,b\}$,
we define a meromorphic quadratic differential $\Omega_{\ul}$ by
\begin{gather}\label{eq:Omega l}
    \Omega_{\ul}(z):=\prod_{i=1}^b \frac{\diff z^2}{(z-z_i)^{l_i}}.
\end{gather}
We call $\Omega_{\ul}$ the \emph{primary quadratic differential
compatible with $(\Pone,f)$}.
\end{definition}

For a Hurwitz cover $f \colon \Pone \to \Pone$ with poles of
order $k_1,\dots,k_b$ at $z_1,\dots,z_b$ respectively,
consider the primary quadratic differential $\Omega_{\ul}$. Then the quadratic differential
\begin{gather}\label{eq:phi=e^f}
    \phi=\phi(f,\ul)\colon=e^f \Omega_{\ul}
\end{gather}
on $\Pone$ gives a quadratic differential with exponential type singularities of indexes
$$(\uk,\ul)=\{(k_1,l_1),\dots,(k_b,l_b)\}$$ at $z_1,\dots,z_b$ respectively.
The following statement gives the other direction.

\begin{proposition}
\label{prop:exp_HC}
Let $\phi$ be a quadratic differential on $\Pone$
with exponential type singularities of indexes $(\uk,\ul)$ at $z_1,\dots,z_b$ respectively.
Then there is a unique Hurwitz cover $f \colon \Pone \to \Pone$
with poles of order $k_1,\dots,k_b$ at $z_1,\dots,z_b$,
up to $f \mapsto f+2 \pi \ii m$ for $m \in \ZZ$,
such that $\phi$ is of the form in \eqref{eq:phi=e^f}.
\end{proposition}
\begin{proof}
First we consider the holomorphic function $g(z):= \phi(z) \slash \Omega_{\ul}(z)$
on $\Pone \setminus\{z_1,\dots,z_b\}$.
Then by definition, $g(z)$ is non-zero on $\Pone \setminus\{z_1,\dots,z_b\}$ and points $z_1,\dots,z_b$
are exponential type singularities. Since $g(z)$ is non-zero away from $z_i$,
it is sufficient to show that the function $\log g(z)$ has a pole of
order $k_i$ at $z_i$ and is single-valued around $z_i$. By definition of
exponential type singularity,
there is a coordinate change $v(z)$ satisfying $v(z_i)=0$ and $v^{\prime}(z_i) \neq 0$
near $z_i$ such that
\[
    \phi=e^{v^{-k_i}}v^{-l_i}\cdot h_1(v)\diff v^2,
\]
where $h_1(v)$ is some non-zero holomorphic function near $v=0$. We also note
that under the above coordinate change, the primary quadratic differential $\Omega_{\ul}$
can be written as
\[
    \Omega_{\ul}=v^{-l_i}\cdot h_2(v)\diff v^2,
\]
where $h_2(v)$ is also some non-zero holomorphic function near $v=0$.
Thus $g(z(v))=\exp(v^{-k_i} )\cdot h_1(v)\slash h_2(v)$ near $v=0$.
Thus $\log g$ has a pole of order $k_i$ at $z_i$ as required.
\end{proof}

Note that the primary quadratic differential $\phi=\phi(f,\ul)$ in \eqref{eq:phi=e^f}
provides a canonical grading $\grad=\grad(f,\ul)$ on the associated marked surface $\rsf$.
Define the numerical data for the Hurwitz cover $(\Pone, f)$ by
\begin{gather}\label{eq:Hnum}
    \num(\Pone,f)=\num\left((\rsf)^\grad\right).
\end{gather}
%=========================================================
\subsection{$q$-Quadratic differentials}
%=========================================================
Similar to Proposition~\ref{prop:exp_HC},
we consider the analogue statement for CY-$s$ type quadratic differentials.
Fix a complex number $s$ with $\Re(s)>2$.
For a Hurwitz cover $f \colon \Pone \to \Pone$ with simple zeroes and poles of
order $k_1,\dots,k_b$ at $z_1,\dots,z_b$ respectively, define a multi-valued
quadratic differential on $\Pone$ ramified at $\Crit(f)$ by
\begin{gather}\label{eq:xxi}
    \xi=\xi(f,s,\ul)\colon=f^{s-2} \,\Omega_{\ul}=e^{(s-2)\log f}\Omega_{\ul}
\end{gather}
Then $\xi$ is a CY-$s$ type $q_s$-quadratic differential
with $s$-poles of type $(\uk,\ul)$ at $z_1,\dots,z_b$.
The following statement gives the other direction in this case.

\begin{proposition}
\label{prop:CY_HC}
Let $\xi$ be a CY-$s$ type $q_s$-quadratic differential on $\Pone$
with $s$-poles of type
$(k_1,l_1),\dots,(k_b,l_b)$ at $z_1,\dots,z_b$
together with a fixed value $\xi(*)$ on some point $* \in \log (\Pone )^{\circ}$.
Then there is a unique regular Hurwitz cover
$f \colon \Pone \to \Pone$ with poles of order
$k_1,\dots,k_b$ at $z_1,\dots,z_b$ and a single-valued function
$\log f$ on $\log \Pone^{\circ}$,
up to multiplying by $\omega_{s-2}^m$, for $\omega_{s-2}=e^{2\pi \ii  \slash (s-2)}$ and $m \in \ZZ$:
\begin{gather}\label{eq:r^m}
    f \mapsto \omega_{s-2} ^m f\quad\text{or}\quad\log f \mapsto \log f +\frac{2 \pi \ii m}{s-2},
\end{gather}
such that
\[
\xi=f^{s-2} \,\Omega_{\ul},
%\quad \text{and}\quad \xi(*)=e^{(s-2)\log f(*)}\Omega_{\ul}
\]
where $\ul=(l_1,\dots,l_b)$ and the value $\log f(*)$ is determined by $\xi(*)=e^{(s-2)\log f(*)}\Omega_{\ul}(*)$.
\end{proposition}
\begin{proof}
First we note that the ambiguity $f \mapsto \omega_{s-2}^m f$ comes from
$f^{s-2}=(\omega_{s-2}^mf)^{s-2}$ since $\omega_{s-2}^{s-2}=1$.
Consider a multi-valued function $g(z):=\xi(z) \slash \Omega_{\ul}(z)$ and let
\[
    f(z)\colon=g(z)^{1 \slash (s-2)}=e^{\frac{1}{s-2} \log g(z)}
\]
up to multiplying $\omega_{s-2}^m$.
We only need to show that $f(z)$ has a pole of order $k_i$ at $z_i$ and is single-valued near $z_i$.

Similar to the proof of Proposition~\ref{prop:exp_HC},
there is a coordinate change $v(z)$ satisfying $v(z_i)=0$ and $v^{\prime}(z_i) \neq 0$
near $z_i$ such that
\[
    \xi=v^{-k_i(s-2)-l_i   }\diff v^2
\]
and
\[
    \Omega_{\ul}=v^{-l_i}\cdot h(v)\diff v^2,
\]
where $h(v)$ is also some non-zero holomorphic function near $v=0$.
Then $g(z(v))=v^{-k_i(s-2)} \cdot h(v)^{-1}$ near $v=0$. Thus
$f(z(v))=g(z(v))^{1 \slash (s-2)}$ has a pole of order $k_i$ at $v=0$ and
is single-valued near $v=0$ since $h(v)^{-1}$ is non-zero.
\end{proof}

%=========================================================
\subsection{Framed regular Hurwitz spaces}
%=========================================================
\label{sec:framed_HS}
Let $\gzero^\grad=(\gzero,\M,\Y,\grad)$ be a graded marked surface of genus zero.
Then the numerical data of $\gzero$ is just
\[
    \num(\gzero^\grad)=(b;\uk,\ul)
\]
since $g=0$ and $\LP_0$ is empty.
We also choose an arc system $\TT$ and a cut $\cut$
and thus have the log surface $\log_{\cut}\gzero_\Tri^\grad$,
which will be simply written as $\log \gzero_\Tri$.

Let $(\Pone,f)$ be a regular Hurwitz cover of polar type $\uk$ and
$\log \Pone_{\Tri}^f$ be the associated topological log surface constructed in Section~\ref{sec:log_HW}
with numerical data $\num(\Pone,f)$ in \eqref{eq:Hnum}.
Then Theorem~\ref{thm:winding} implies that there is a homeomorphism
\[
    h \colon  \log \gzero_\Tri \to \log \Pone_{\Tri}^f
\]
which commutes with the action of $\langle q \rangle$.
\begin{definition}
A \emph{$\log \gzero_\Tri$-framed Hurwitz cover $(\Pone,f,\log f,h)$}
consists of a regular Hurwitz cover $(\Pone,f)$,
a fixed choice of the (single-valued) function $\log f$ on $\log \Pone^{\circ}$ and
a homeomorphism
\[
    h \colon  \log \gzero_\Tri \to (\log \Pone_{\Tri}^{\circ})^f,
\]
which commutes with the action of $\langle q \rangle$ and
matches the numerical data $\num(\Pone,f)$ with $\num(\log\gzero_\Tri)$.
\end{definition}
Two $\log \gzero_\Tri$-framed Hurwitz cover $(\Pone,f_i,\log f_i,h_i)$ are $\log\gzero_\Tri$-equivalent if
\begin{itemize}
\item there is a biholomorphism $F \colon \Pone \to \Pone$ satisfying $F(f_1^{-1}(\infty))=f_2^{-1}(\infty)$,
together with the choice of a lift
$\widetilde{F} \colon \log \Pone^{ \circ} \to \log \Pone^{ \circ}$
such that $\widetilde{F}^* (\log f_{2})=\log f_{1}$ and
\item
$h_2^{-1} \circ \widetilde{F} \circ h_1\in\Homeo_0(\log\gzero_\Tri)$,
where
\[
    \widetilde{F} \colon \log \Pone_{\Tri}^{f_1}\to
    \log \Pone_{\Tri}^{f_2}
\]
is the induced homeomorphism.
\end{itemize}

Denote by $\HS(\log \gzero_\Tri)$ the moduli space of $\log \gzero_\Tri$-framed Hurwitz covers.

%We note that by definition, there is a canonical isomorphism
%\[
%\HS(\log \gzero_\Tri) \slash \MCG^\circ(\log \gzero_\Tri) \cong \HS(0,\uk)_{\reg}.
%\]
Recall that $\MCG^\circ(\log \gzero_\Tri)$ is generated by the deck transformation $q$
and those homeomorphism induced from $\MCG(\gzero_\Tri)$.
The deck transformation $q \in \MCG^\circ(\log \gzero_\Tri)$ acts on $\HS(\log \gzero_\Tri)$ as
\[
 (\Pone,f,\log f,h)\cdot q=(\Pone,f,\log f+2\pi \ii,h).
\]
The right action of $g \in \MCG^\circ(\log \gzero_\Tri)$ on $\HS(\log \gzero_\Tri)$ is given by
\[
 (\Pone,f,\log f,h)\cdot g=(\Pone,f,\log f,h\circ g).
\]

Choose a boundary component $\partial_1 \subset \gzero_\Tri$ with $k_1$ marked points and
set  $\widetilde{\partial}_1:=\pi^{-1}(\partial_1)$
for the covering map $\pi_{\Tri} \colon \log \gzero_\Tri \to \gzero_\Tri$.
For a $\log \gzero_\Tri$-framed Hurwitz cover $(\Pone,f,\log f,h)$,
assume that a pole $z_1$ of $f$ corresponds to $\widetilde{\partial}_1$
through the framing $h$. Then by the automorphism of $\Pone$, we can always take
the representative $(\Pone,f,\log f,h)$ with $z_1=\infty$ and the
coefficient of the pole of $f$ at $z_1=\infty$ is one, namely
\begin{gather}\label{eq:f=}
    f=z^{k_1}+O(z^{k_1+1}).
\end{gather}
Recall from Definition~\ref{def:x.f.quad} that $\QQuad_s(\log \gzero_\Tri)$ is
the moduli space of $q$-quadratic differentials
with the numerical data $\num(\log \gzero_\Tri)=(b,\uk,\ul)$.
\begin{theorem}
\label{thm:HW_Quad_iso}
There is an isomorphism of complex manifolds
\begin{gather}\label{eq:Psis}
    \begin{array}{rcl}
    \Psi_s \colon \HS(\log \gzero_\Tri)\;\;\; &\longrightarrow& \QQuad_s(\log \gzero_\Tri)\\
    (\Pone,f,\log f, h) &\longmapsto& (\Pone,\xi,h,s)
    \end{array}
\end{gather}
for $\xi=\xi(f,x,\ul)$ as in \eqref{eq:xxi}.
Here we use the representative $f$ with the form \eqref{eq:f=}.
\end{theorem}
\begin{proof}
The existence and injectivity of the map $\Psi_s$ are explicit.
We only need to show the surjectivity of $\Psi_s$.
For a $q$-quadratic differential, by the automorphism of $\Pone$,
we can take its representative as $(\Pone,\xi,h,s)$ with $z_1=\infty$,
where $z_1$ is an $s$-pole corresponding to $\widetilde{\partial}_1$.
In addition, we can normalize as
\[
\xi=c ( z^{k_1(s-2)}+\cdots)\Omega_{\ul},
\]
where $c=q_s^{2m}$ for some $m \in \ZZ$.
Take some point $* $ from the interior of
$\log \gzero_\Tri$. Then by Proposition~\ref{prop:CY_HC}, there is a regular Hurwitz cover
$f \colon \Pone \to \Pone$ and a single-valued function
$\log f$ on $\log \Pone^{\circ}$, up to multiplying by $\omega_{s-2}^m$ as in \eqref{eq:r^m},
such that
\[
    \xi=f^{s-2} \,\Omega_{\ul} \quad \text{and} \quad \xi(h(*))=
    e^{(s-2)\log f(h(*))}\Omega_{\ul}(h(*)).
\]
By the construction of $f$ in the proof of Proposition~\ref{prop:CY_HC},
we can take $f$ with the form $f=z^{k_1}+\cdots$. This implies the surjectivity of $\Psi_s$.
\end{proof}

\begin{remark}

As a complex manifold, the moduli space $\QQuad_s(\log \gzero_\Tri)$ is
independent of the date $\ul$ of $\num(\log \gzero_\Tri)$. However, the period map
$$\Pi_s=\int \sqrt{\xi}$$ strongly depends on $\ul$ since $\xi=f^{s-2}\Omega_{\ul}$.
The choice $\ul$ corresponds to the choice of the primitive form \cite{Sa1,SaTa}
(the primary differential \cite[Section 5]{Dub1}) and controls the
Frobenius structure on $\HS(\log \gzero_\Tri)$.
\end{remark}

\begin{corollary}\label{cor:n+1}
Consider the union
\[
    \QQuad_{>2}(\log \gzero_\Tri)\colon=
        \bigcup_{s \in \bC^{>2}} \QQuad_s(\log \gzero_\Tri),
\]
where $\bC^{>2}=\{s \in \bC \,\vert\, \Re(s)>2\}$.
Then the space $\QQuad(\log \gzero_\Tri)$ has the structure of
a complex manifold of dimension $(n+1)$
induced from the bijection
\[
    \QQuad_{>2}(\log \gzero_\Tri) \cong \HS(\log \gzero_\Tri) \times \bC^{>2}
\]
given by Theorem~\ref{thm:HW_Quad_iso}.
\end{corollary}

Passing to $q$-stability conditions, we have the following result
on gluing $q$-stability conditions along $s$-direction for
\begin{gather}
  \QStap_{>2}\DX=\bigcup_{s \in \bC^{>2}} \QStap_s\DX,
\end{gather}
which provides an answer for \cite[Conj.~3.16]{IQ1}.

\begin{corollary}\label{cor:n+2}
For any $s$ with $\Re(s)>2$, we have
\begin{gather}\label{eq:XiPsi}
    \Xi_s\circ\Psi_s \colon \HS(\log \gzero_\Tri)\cong\QStap_{s}\DX,
\end{gather}
where $\Xi_s$ and $\Psi_s$ are as in \eqref{eq:QQ=QS2} and \eqref{eq:Psis}.
Hence
$\QStap_{>2}\D_\XX$ admits the structure of a complex manifold of dimension $n+1$
and the projection map
\[
    \pi \colon \QStap\D_\XX \to \CC,\quad (\sigma,s) \mapsto s
\]
is holomorphic.
\end{corollary}
\begin{proof}
In our case, $\surf=\Pone$ has genus zero and thus Corollary~\ref{cor:zero} applies.
By Theorem~\ref{thm:q=x}, we have \eqref{eq:main}, i.e., $\QStap_{s}\DX\cong\QQuad_s(\log \gzero_\Tri)$
in this case.
Thus the claim follows from Theorem~\ref{thm:main} and Corollary~\ref{cor:n+1} respectively.
\end{proof}

%\begin{remark}
%\end{remark}

%=========================================================
\subsection{Disk case: type A example}\label{sec:A}
%=========================================================
Now consider the type A case, where $\rs=\Pone$, $b=1$, $\uk=(n+1)$
and $\ul=(4)$. In this case, the marked surface $\gzero$ is
a disk with $n+1$ points on the boundary, which will be denoted by $\dd$.
Consider the space of polynomials of degree $n+1$
\begin{gather}
    \Poly_{n}:=\left\{f(z)\,\vert\, f(z)=z^{n+1}+a_1 \,
        z^{n-1}+a_2\, z^{n-2}+\cdots+a_n,\,a_i \in \bC \right\} \cong \bC^n
\end{gather}
and its regular subset
\[
    \Poly_{n,\reg}
        :=\{f(z) \in \Poly_n \,\vert\,  \text{all zeroes of $f(z)$ are simple}\}.
\]
Then we can regard $f(z) \in \Poly_{n,\reg}$
as a genus $0$ regular Hurwitz cover
\[
f \colon \Pone \to \Pone
\]
with a unique pole of order $n+1$ at $z =\infty$.

Let $\mathfrak{h}$ be the Cartan subalgebra of type $A_n$ and
$\mathfrak{h}_{\reg} \subset \mathfrak{h}$ be the
regular subset.
Then there is an identification
$$\Poly_{n,\reg}\cong\mathfrak{h}_{\reg} \slash W,$$
where $W$ is the Weyl group of type $A_n$, namely the symmetric group of degree $n+1$.
More details on $\mathfrak{h}$ in the Dynkin case will be discussed in Section~\ref{sec:ADE}.
Denote by $\widetilde{\Poly}_{n,\reg}$ the universal cover of $\Poly_{n,\reg}$.

Let $\omega_{n+1}=e^{2\pi \ii \slash (n+1) }$ be the $(n+1)$-th root of unity and
denote by $C_{n+1}:=\langle \omega_{n+1} \rangle$ the cyclic group of order $n+1$ generated by $\omega_{n+1}$.
We define the action of $C_{n+1}$ on $\Poly_{n,\reg}$ by
\[
    (\omega_{n+1} \cdot f) (z):=f(\omega_{n+1}^{-1} z).
\]
Alternatively, we can write $\omega_{n+1} \cdot (a_i)_i=(\omega_{n+1}^{i+1}a_i)_i$
for coefficients $(a_i)_i$ of $f(z)$.
\begin{lemma}
\label{lem:HS_poly}
There is an isomorphism
\begin{gather}\label{eq:PolyHS}
    \Poly_{n,\reg} \slash C_{n+1} \cong \HS(0,(n+1))_{\reg}
\end{gather}
by regarding a regular polynomial as a regular Hurwitz cover on $\Pone$.
\end{lemma}
\begin{proof}
It is straightforward to check that the only
automorphisms of $\Pone$ which preserve $\Poly_{n,\reg}$ are $z \mapsto \omega z$ for $\omega \in C_{n+1}$.
Hence the claim follows.
\end{proof}
We note that the cyclic group $C_{n+1}$ can be identified with
the \emph{marked mapping class group} $\MCG_\bullet(\dd)$, whose generator is
the rotation $\tau$ that sends each marked point on the boundary to the next (clockwise) point.
Recall that $\MCG_\bullet$ here is only required to preserve $\M$ setwise.
Let $\ddo$ be the decorated version of $\dd$ (with $|\Tri|=n+1$).
It is well-known that
\[
    \MCG_\bullet(\ddo)=\left( \Br_{n+1}\times\<\tau\> \right)/ \<\tau^{n+1}\>,
\]
which follows from the short exact sequence
\[
    0\to \SBr(\ddo) \to \MCG_\bullet(\ddo)\to \MCG_\bullet(\dd)\to 0
\]
for $$\SBr(\ddo)=\MCG(\ddo,\partial\ddo)=\Br_{n+1}.$$
Here, $\MCG(\ddo,\partial\ddo)$ is the classical mapping class group that preserves the boundary pointwise.
More precisely, the $\tau^{n+1}$ is the Dehn twist along $\partial\ddo$,
which is also the generator of the center of $\Br_{n+1}$.

Let $\log \ddo$ be the topological (ungraded) log surface of type A. Then
first we note that the lift of the action $\tau \in \MCG_\bullet(\dd)$ on
$\log \ddo$ can be identified with the deck
transformation $q \in \MCG^\circ(\log \ddo)$.
Then we have the following description of the mapping group $\MCG^\circ(\log \ddo)$.
\begin{lemma}\label{lem:MCG_braid}
As
$$
    \MCG^\circ(\log \ddo) = \left( \Br_{n+1} \times \langle q \rangle\right) \slash \langle q^{n+1} \rangle,
$$
where $q^{n+1}$ acts on $\Br_{n+1}$ as the lift of $D_{\partial\ddo}$,
there is an isomorphism of the groups
\[
    \MCG^\circ(\log \ddo) \cong \Br_{n+1} \rtimes C_{n+1}.
\]
\end{lemma}

\begin{proposition}\label{prop:uc=HS}
There is an isomorphism of complex manifolds
\begin{gather}\label{eq:UC poly}
    \widetilde{\Phi}_{n+1} \colon \widetilde{\Poly}_{n,\reg}\to \HS(\log \ddo),
\end{gather}
as universal covering of
\begin{gather}\label{eq:poly}
    \begin{array}{ccc}
    \Phi_{n+1} \colon \HS(0,(n+1))_{\reg} &\longrightarrow& \HS(\log \ddo) \slash \MCG^\circ(\log \ddo) \\
    (\Pone,f) &\longmapsto& (\Pone,f,\log f,h).
    \end{array}
\end{gather}
\end{proposition}
\begin{proof}
By Definition, we have \eqref{eq:poly}.

Let
\[
    \conf_m(X)\colon=\{ \{x_1,\ldots,x_n\}\in X^n\mid x_i\neq x_j,\;\forall i\neq j\}
\]
be the $m^{\mathrm{th}}$ (unordered) configuration space of $X$.
Then there is a natural identification
\begin{equation*}\begin{array}{rcl}    \vspace{1ex}
    \Poly_{n,\reg}&=&\{f(z)=\displaystyle\prod_{k=1}^{n+1}(z-z_k) \mid z_k\in\bC, z_i\ne z_j,\;\forall i\neq j\}\\ &=&\conf_{n+1}(\bC).
\end{array}\end{equation*}
It is well-known that $\pi_1(\Poly_{n,\reg})=\Br_{n+1}$.
By Lemma~\ref{lem:HS_poly}, we have
\[
    \pi_1(\HS(0,(n+1))_{\reg})=\pi_1(\Poly_{n,\reg}/C_{n+1})=\Br_{n+1}\rtimes C_{n+1}.
\]
Combining Lemma~\ref{lem:MCG_braid}, we obtain the required covering version \eqref{eq:UC poly}.
\end{proof}

%=========================================================
\section{Almost Frobenius structures for stability conditions: ADE case}\label{sec:SF}
%=========================================================
The aim of this section is to give a conjectural description
of the spaces of $q$-stability conditions on the Calabi-Yau-$\XX$ completions
of type ADE quivers.

%=========================================================
\subsection{Almost Frobenius structures}\label{sec:almost}
%=========================================================
We recall the definition of an almost Frobenius structure
from \cite[\S~3]{Dub1}.
\begin{definition}
\label{def:Frob_alg}
A {\it Frobenius algebra} $(A,*,(\,,\,))$ is a pair of a commutative
associative $\CC$-algebra $(A,*)$ with the unit $E \in A$ and
a non-degenerate symmetric bilinear form
\[
    (\,,\,)\colon A \otimes A \to \CC
\]
satisfying the condition
\[
    (X * Y,Z)=(X,Y*Z)
\]
for $X,Y,Z \in A$.
\end{definition}

For a complex manifold $M$, denote by $\TM$ the
holomorphic tangent sheaf of $M$.
\begin{definition}
An \emph{almost Frobenius manifold}
$(M,*,(\,,\,),E,e)$ of the charge
$d \in \CC\,(d \neq 1)$
consists of the following data:
\begin{itemize}
\item a complex manifold $M$ of dimension $n$,
\item a commutative associative $\O_M$-linear product
\[
    * \colon \TM \otimes \TM \to \TM
\]
\item  a non-degenerate symmetric $\O_M$-bilinear form
\[
    (\,,\, ) \colon \TM \otimes \TM \to \O_M
\]
satisfying the condition
\[
    (X*Y,Z)=(X,Y*Z)
\]
for (local) holomorphic vector fields $X,Y,Z$,
\item global holomorphic vector fields $E$ (called
the {\it Euler vector field}) and $e$ (called
the {\it primitive vector field}).
\end{itemize}
They satisfy the following axioms:
\begin{itemize}
\item[(1)]
The metric $(\,,\,)$ is flat.
\item[(2)]
Let $p^1,\dots,p^n$ be flat coordinates of the flat metric $(\,,\,)$
on some open subset $U \subset M$ and set $\rd_i:=\frac{\rd}{\rd p^i}$.
Then there is some function
$F(p) \in \O_M(U)$ such that
the multiplication $*$
can be written as
\[
    \rd_i *\rd_j=\sum_{k=1}^n C_{ij}^k(p) \rd_k,
\]
where
\[
    C_{ij}^k(p)=\sum_{l=1}^n
    G^{kl}\rd_i \rd_j \rd_l F(p)
\]
and $(G^{kl})$ is the inverse matrix of the metric $(G_{kl}=(\rd_k,\rd_l))$.
\item[(3)]
The function $F(p)$ satisfies the homogeneity equation
\[
    \sum_{i=1}^n p^i \rd_i F(p)=2 F(p)+\frac{1}{1-d}\sum_{i,j=1}^n G_{ij}p^i p^j.
\]

\item[(4)]The Euler vector field takes the form
\[
    E=\frac{1-d}{2}\sum_{i=1}^n p^i \rd_i
\]
and is the unit of the multiplication $*$.
\item[(5)]The primitive vector field $e$ is invertible with respect to the multiplication $*$
and the derivation $P_{\nu}(p) \mapsto e P_{\nu}(p)$
on the space of solutions
of the following integrable equation with the parameter $\nu \in \CC$
\begin{equation}
\label{eq:period}
    \rd_i \rd_j P_{\nu}(p) =\nu \,\sum_{k=1}^n C_{ij}^k(p)
    \rd_kP_{\nu}(p)
\end{equation}
acts as the shift of the parameter $\nu \mapsto \nu-1$,
i.e., if $P_{\nu}(p)$ is the solution
of the equation (\ref{eq:period}) with the parameter $\nu$, then $e P_{\nu}(p)$ is
the solution with the parameter $\nu-1$.
\end{itemize}
\end{definition}
We note that
the restriction of the multiplication and the metric on the
tangent space $T_pM$ for each point $p \in M$ gives
a family of Frobenius algebras $(T_p M,*_p,(\,,\,)_p)$.

The flat coordinates $p^1,\dots,p^n$ of the flat metric $(\, ,\,)$
are called {\it periods} and independent solutions $P_{\nu}^1,\dots,P_{\nu}^n$
of the equation in Axiom $(5)$ are called  {\it twisted periods}.
The twisted periods satisfy the following quasi-homogeneity condition.
\begin{proposition}
The derivative of the twisted period $P_{\nu}$
along with the Euler field $E$ satisfies
\[
    E  P_{\nu}=\left(\frac{1-d}{2}+\nu \right)P_{\nu}
    +\mathrm{const.}
\]
\end{proposition}
\begin{proof}
We show
\[
    \rd_j E P_{\nu}=\left(\frac{1-d}{2}+\nu \right)\rd_j P_{\nu}.
\]
Since $[\rd_j,E]=\frac{1-d}{2} \rd_j$, the above equality is
equivalent to $E \rd_j P_{\nu}=\nu \rd_j P_{\nu}$.
Recall that the Euler field is the unit of the Frobenius algebra, i.e.
$E*\rd_j=\rd_j$. This is equivalent to
\[
    \frac{1-d}{2}\sum_{i,k=1}^n p^i C_{ij}^k \rd_k=\rd_j.
\]
Together with the equation (\ref{eq:period}) for the twisted period
$\rd_i \rd_j P_{\nu}(p)=\nu\sum_k C_{ij}^k \rd_k P_{\nu}$, we have
\[
    E \rd_j P_{\nu}=
    \frac{1-d}{2}\sum_{i=1}^n p^i\rd_i \rd_j P_{\nu}
    =\nu\frac{1-d}{2}\sum_{i,k=1}^n p^i  C_{ij}^k \rd_k P_{\nu}=\nu \rd_j P_{\nu}.
\]
\end{proof}

In the following, we always normalize the twisted period $P_{\nu}$
to satisfy the quasi-homogeneity condition
\begin{equation}
\label{eq:quasi}
    E P_{\nu}=\left(\frac{1-d}{2}+\nu \right)P_{\nu}
\end{equation}
by shifting some constant.

%=========================================================
\subsection{Flat torsion-free connections and twisted periods}\label{sec:flat}
%=========================================================
In this section, we summarize the relation between flat torsion-free connections
and affine structures on manifolds. In the following
we consider in the category of complex manifolds and
all vector fields, differential forms and connections are holomorphic.
Let $M$ be a complex manifold of dimension $n$ and
$\nabla$ be a connection of the holomorphic
tangent bundle $TM$. We recall that
$\nabla$ is called {\it torsion-free} if
\[
    \nabla_{X}Y-\nabla_Y X -[X,Y]=0
\]
for vector fields $X,Y$ on $M$.
Recall that
a connection $\nabla$ on the tangent bundle $TM$ induces
the connection $\nabla$ on the cotangent bundle $T^*M$
by the formula
\[
    \diff \langle X, \omega \rangle=
    \langle \nabla X, \omega \rangle+\langle X, \nabla \omega \rangle
\]
for a vector field $X $ and a $1$-form $\omega $ on $M$.
The following statement is well-known.
\begin{proposition}
Let $M$ be a simply connected complex manifold of dimension $n$
and $\nabla$ be a flat torsion-free connection on $TM$. Then
there are $n$ holomorphic functions $P^1,\dots,P^n \colon M \to \CC$
such that
\[
    \diff P^1 \wedge \cdots \wedge \diff P^n \neq 0
\]
everywhere on $M$ and
\[
    \nabla \diff P^i=0.
\]
These functions are unique up to affine transformation
\[
    P^i \mapsto \sum_{j=1}^n A^i_j P^j +B^i,
\]
where $(A^i_j) \in \GL(n,\CC)$ and $(B^i) \in \CC^n$.
\end{proposition}
Thus functions $P^1,\dots,P^n$ give local coordinates
on $M$. This coordinate system is called a {\it flat coordinate system}.
Consider the vector space
\[
    V_{M}:=\bigoplus_{i=1}^n \CC \diff P^i.
\]
Then since $\diff P^1,\dots,\diff P^n$ are independent global sections of $T^*M$,
we have a trivialization of the cotangent bundle
\[
    T^*M \cong V_{M} \times M
\]
and the flat torsion-free connection $\nabla$ is just
equal to the exterior differential $\diff$ on the trivial bundle.

Consider an almost Frobenius manifold $(M,*,(\,,\,),E,e)$. Let $\nabla$ be
the Levi-Civita connection with respect to the flat metric $(\,,\,) $ on $M$.
Fix a complex number $\nu \in \CC$. Then we can define a
flat torsion-free connection $\check{\nabla}$ on $TM$
\[
    \check{\nabla}_X Y:=\nabla_X Y+\nu X *Y
\]
for vector fields $X,Y$.
Then the equation (\ref{eq:period}) for twisted periods $P_{\nu}$
is equivalent to the equation
\[
    \check{\nabla} \diff P_{\nu}=0.
\]
\begin{remark}
\label{rem:primitive}
The connection $\check{\nabla}$ is called the {\it Dubrovin connection}
or the {\it second structure connection} of a Frobenius manifold.
For details, we refer to \cite[\S~3]{Dub1} and \cite[\S~2.1]{M}.
\end{remark}

\def\MQ{\widetilde{\hh{M}}_Q}
%=========================================================
\subsection{ADE type}\label{sec:ADE}
%=========================================================
We recall the almost Frobenius structure on the
universal covering of the regular locus of the Cartan subalgebra
in the case of type ADE following \cite[\S~5.2]{Dub1}.

Let $Q$ be a quiver of type ADE and
$(L_Q,(\,,\,))$ be the root lattice associated with
$Q$. Let $\Tri_+ \subset L_Q$ be the set of positive roots.
The Cartan subalgebra is
$$\hs:=\Hom_{\ZZ}(L_Q,\CC)$$ and let $e_1,\dots,e_n \in \hs$
be the dual basis of $\alp_1,\dots,\alp_n$.
Then
$$\hs \cong \oplus_{i=1}^n \CC e_i \quad\text{with metric}\quad
    (e_i,e_j):=(A^{-1})_{ij},$$
where $A^{-1}$ is the inverse of the Cartan matrix.
We regard
a root $\alp \in \Tri_+$ as a function on
$\hs$ via the pairing between $\hs$
and $L_Q$.
Consider the regular subset
\begin{gather}\label{eq:hreg}
  \hs_{\reg}:=\hs \setminus \bigcup_{\alp \in \Tri_+} \ker \alp
\end{gather}
and its universal covering space
\[
    \MQ:=\widetilde{\hs_{\reg}}.
\]
When $Q$ is an $A_n$ quiver, these definitions/notations match the ones in Section~\ref{sec:A}.
In the following, we construct the almost Frobenius structure on $\MQ$.

Let $\alp_1,\dots,\alp_n$ be simple roots and introduce the
coordinate on $\hs$ given by
$p^i:=\alp_i(p)$ for $p \in \hs$, namely $p=\sum_{i=1}^n p^i e_i$.
Then the coordinates
$p^1,\dots,p^n$ also form the local coordinates on $\MQ$
through the projection $\MQ \to \hs_{\reg}$.
Since $T\MQ \cong \hs \times \MQ$, the metric $(\,,\,)$ on $\hs$
induces the flat metric on $\MQ$ and the coordinates $p^1,\dots,p^n$
form the flat coordinates.

Define the potential function $F \colon \MQ \to \CC$ by
\[
F(p):=\frac{1}{2}\sum_{\alp \in \Tri_+}\alp(p)^2 \log \alp(p).
\]
Since $\MQ$ is the universal covering of $\hs_{\reg}$,
$F$ is a single-valued holomorphic function on $\MQ$.
On an open subset $U \subset \MQ$ with the flat coordinate $p^1,\dots,p^n$,
the third derivative of $F$ is given by
\begin{equation}
\label{eq:tensor}
C_{ijk}(p):=\rd_i \rd_j \rd_j F(p)=\sum_{\alp \in \Tri_+}\frac{\alp(e_i) \alp(e_j) \alp(e_k)}{\alp(p)}.
\end{equation}
This defines the multiplication $*$ together with  the
inverse of the metric, which is just the Cartan matrix $(dp^i,dp^j)=A_{ij}$.
The charge $d$ is given by
\[
    d=1-\frac{2}{h},
\]
where $h$ is the Coxeter number.
In particular, the Euler field takes the form
\[
E=\frac{1}{h}\sum_{i=1}^n p^i \rd_i.
\]

For the construction of the primitive vector field $e$, we refer to \cite[\S~5.2]{Dub1}.

%=========================================================
\subsection{Coxeter-KZ connections}
%=========================================================
Following \cite{C} (and the survey \cite[\S~4]{TL}), we consider
a certain integrable connection on $\hs_{\reg}$, called the Coxeter-KZ connection.

As in the previous section, let $(L_Q,(\,,\,))$ be the root lattice associated with
a quiver $Q$ of type ADE.
For a root $\alp \in \Tri$,
consider reflections
\[
    r_{\alp} \colon L_Q \to L_Q,\quad \lam \mapsto \lam
        -(\lam,\alp)\alp,
\]
for $\alp \in L_Q$. These reflections form
the {\it Weyl group}
\[
    W_Q:=\langle r_{\alp} \,|\,\alp \in \Tri \rangle,
\]
which naturally acts on the Cartan subalgebra $\hs$.

Consider a representation
of $W_Q$ on a finite dimensional vector space $V$:
\[
    \rho \colon W_Q \to \GL(V).
\]
On the trivial bundle $V \times \hs_{\reg} \to \hs_{\reg}$,
we define the {\it Coxeter-KZ (CKZ)} connection $\nabla^{\CKZ}_{\nu}$ by
\[
    \nabla^{\CKZ}_{\nu}:=\diff-\nu \sum_{\alp \in \Tri_+}
    (\rho(r_{\alp})-\id_V)\frac{d \alp}{\alp},
\]
where $\nu \in \CC$ is some complex number.
\begin{theorem}
\label{thm:CKZ}
The connection $\nabla^{\CKZ}_{\nu}$ is a $W$-invariant flat connection
on $V \times \hs_{\reg} $. The monodromy representation
$$\rho_{\nu \colon }\pi_1(\hs_{\reg}\slash W) \cong \Br_Q \to \GL(V)$$
factors through the
Hecke algebra $H_Q$:
\begin{equation*}
    \xymatrix{
        R[\Br_Q] \ar@{|->}[rr] \ar[dr]&&
        R[\GL(V)]  & \bb_i \ar@{|->}[rr] \ar@{|->}[dr]&&
        \rho_{\nu}(\psi_i) \\
        &H_Q \ar[ur] &&& -T_i \ar@{|->}[ur]&.
    }
\end{equation*}
where we specialize $q=e^{2\mathbf{i} \pi \nu}$ for the Hecke algebra $H_Q$.
\end{theorem}
\begin{proof}
See \cite{C} and also the survey \cite[\S~4]{TL}.
\end{proof}

\begin{remark}
In \cite{TL}, the CKZ connection is given by
\[
    \nabla^{\prime \CKZ}_{\nu}=\diff -\nu \sum_{\alp \in \Tri_+}
    \rho(r_{\alp})\frac{d \alp}{\alp}.
\]
Two connections  $\nabla^{\CKZ}_{\nu}$ and $\nabla^{\prime \CKZ}_{\nu}$
are related as the following gauge transformation.
Consider the function
\[
    f\colon=\prod_{\alp \in \Tri_+}\alp
\]
on $\hs$.
Then the equality
\[
    \nabla^{\prime \CKZ}_{\nu}=f^{\nu}\circ \nabla^{\CKZ}_{\nu} \circ f^{-\nu}
\]
holds since
\[
    \diff f^{-\nu}=-\nu \sum_{\alp \in {\Tri_+}}\id_V\frac{d \alp}{\alp}.
\]
\end{remark}

The connection $\nabla^{ \CKZ}_{\nu}$ can be naturally extended  on
$\MQ=\widetilde{\hs_{\reg}}$.
As the special case, we consider the CKZ connection on the tangent bundle
of $\MQ$. We note that  the tangent bundle of $M$ is given by
\[
    T\MQ\cong \hs \times \MQ.
\]
Since the Weyl group $W$ acts on $\hs$ as the reflection group, we have the
CKZ connection on $T\MQ$.

\begin{proposition}
The CKZ connection $\nabla^{ \CKZ}_{\nu}$ on
$T\MQ$ coincides with the Dubrovin connection
$\check{\nabla}$. In particular, $\nabla^{ \CKZ}_{\nu}$ on $T\MQ$ is torsion-free
and the twisted period $P_{\nu}$ satisfies
\[
    \nabla^{ \CKZ}_{\nu} \diff P_{\nu}=0.
\]
\end{proposition}
\begin{proof}
First we note that under the identification $T\MQ\cong \hs \times \MQ$,
the vector field $\rd_i=\frac{\rd}{\rd p^i}$ is identified with $e_i$.
By using equation (\ref{eq:tensor}) we have
\[
    \check{\nabla}_{\rd_i}\rd_j=\nu \rd_i*\rd_j=\nu
    \sum_{\alp \in \Tri_+}\sum_{k,l}\frac{\alp(e_i) \alp(e_j) \alp(e_k)}{\alp(p)}G^{kl}\rd_l.
\]
On the other hand,
\begin{align*}
    \nabla^{ \CKZ}_{\nu, e_i}e_j
    =-\nu \sum_{\alp \in {\Tri_+}}(r_{\alp}(e_j)-e_j)\frac{\langle d\alp,e_i \rangle}{\alp(p)}
    =\nu \sum_{\alp \in {\Tri_+}}(e_j-r_{\alp}(e_j))\frac{\alp(e_i)}{\alp(p)}.
\end{align*}
Since
\[
    e_j-r_{\alp}(e_j)=e_j-(e_j-\alp(e_j)\alp^{\vee} )=\alp(e_j)\sum_{k,l}\alp(e_k)G^{kl}e_l,
\]
the result follows.
\end{proof}
Thus the twisted periods for Frobenius manifolds of type ADE can be characterized
as the flat coordinates of the CKZ connections.
Together with Theorem~\ref{thm:CKZ}, we have the following.
\begin{corollary}
The monodromy representation of $\Br_Q$ for the twisted period map
$P_{\nu}$ factors through the Hecke algebra $H_Q$.
\end{corollary}

\begin{remark}
\label{rem:primitive2}
In \cite{Sa1}, Saito defined  the period map associated with the primitive form
on the unfolding of an isolated singularity. He also defined the holonomic D-module
(with the complex parameter $s \in \CC$) whose solutions characterize his period map.
The above twisted period maps
 $P_{\nu}$ on almost Frobenius manifolds of type ADE
 correspond to
 Saito's period maps in the case that isolated singularities are of type ADE.
 For the details, we refer to \cite[\S~5]{Sa1}, \cite[\S~4]{Sa2} and
 \cite[\S~5.1]{Dub1}.
\end{remark}

%=========================================================
\subsection{Cotangent bundles of the spaces of stability conditions}
%=========================================================
Let $\D_{\XX}$ be a triangulated category satisfying Assumption~\ref{ass:R} and
take a basis $[E_1],\dots,[E_n]$ of $K(\D_{\XX})$. Then
\[
    K(\D_{\XX}) \cong \bigoplus_{i=1}^n R[E_i].
\]
Consider holomorphic functions
\[
    Z^i \colon \QStab_s\D_{\XX} \to \CC \quad (i=1,\dots,n)
\]
given as the components of the cental charge defined
by $Z^i(\sigma):=Z(E_i)$ for $\sigma=(Z,\sli) \in \QStab_s\D_{\XX} $.
Consider holomorphic $1$-forms $\diff Z^1,\dots,\diff Z^n$ on $\QStab_s\D_{\XX}$
and define a vector space
\[
    V_\mathrm{S}=\bigoplus_{i=1}^n \CC \diff Z^i.
\]
The following is a direct consequence of Theorem~\ref{thm:localiso2}.
\begin{proposition}
The cotangent bundle of the space $\QStab_s\D_{\XX}$ is
the trivial bundle given by
\[
T^*\QStab_s\D_{\XX} \cong V_\mathrm{S} \times \QStab_s\D_{\XX}.
\]
\end{proposition}
\begin{proof} By Theorem~\ref{thm:localiso2}, for any small open subset
$U \subset \QStab_s\D_{\XX}$, functions $Z^1,\dots,Z^n$
form local coordinates on $U$. This implies the result.
\end{proof}
If we consider the trivial connection $\diff$  on
$V_\mathrm{S} \times \QStab_s\D_{\XX}$, the central charges
$Z^i$ satisfies the trivial equation
\[
    \diff \diff Z^i=0.
\]
Together with the content of Section~\ref{sec:flat}, this suggests the following fundamental problem
on the spaces of stability conditions.
\begin{problem}
To find a simply connected complex manifold $M$ and a torsion-free flat connection $\nabla$
such that there is an isomorphism of complex manifolds
\[
\phi \colon M \to \QStab_s\D_{\XX}
\]
satisfying $\phi^* \diff=\nabla$.
\end{problem}
In the next subsection, we shall give a conjectural answer for $\DXQ$ with type ADE-quiver $Q$.

If we find an isomorphism $\phi$ in the above problem,
then the flat coordinates $P^i$ of $\nabla$ on $M$ can be
identified with the central charges $Z^i$ up to affine transformation.
In Section~\ref{sec:almost}, we fix the ambiguity of
constant shift of twisted periods by the equation (\ref{eq:quasi})
via the Euler field. Correspondingly,
we introduce
the Euler field in the side of stability conditions as follows.
Let $\frac{\rd}{\rd Z^i}$ be the dual basis of $dZ^i$.
Then the tangent bundle of $\QStab_s\D_{\XX}$ is given by
\[
T\QStab_s\D_{\XX}\cong \left(
\bigoplus_{i=1}^n \CC \frac{\rd}{\rd Z^i}  \right) \times \QStab_s\D_{\XX}.
\]
The Euler vector field $\mathcal{E}$ of charge $d$ on  $\QStab_s\D_{\XX}$
is defined by
\[
\mathcal{E}:=\left( \frac{1-d}{2}+
\frac{s-2}{2} \right) \sum_{i=1}^n Z^i \frac{\rd}{\rd Z^i}.
\]
Then central charges $Z^i$ clearly satisfy the  quasi-homogeneity condition
\[
\mathcal{E}Z^i=\left( \frac{1-d}{2}+
\frac{s-2}{2} \right) Z^i.
\]

%=========================================================
\subsection{Conjectural descriptions}
%=========================================================
Similar to Definition~\ref{def:GXQW}, we have the Calabi-Yau-$\XX$ version of a path algebra.

\begin{definition}\cite[\S~2.4]{IQ1}
Let $Q$ be a quiver of type ADE.
Denote by $\Gamma_{\XX}Q$ the associated Calabi-Yau-$\XX$ Ginzburg $\ZZ^2$-graded differential algebra,
$\DXQ$ its finite dimensional derived category and
$\QStab_s\DXQ$ the space of $q$-stability conditions on $\DXQ$.
\end{definition}

\begin{conjecture}
\label{conj:main}
Assume that $\Re (s) \ge 2$ and set $\nu=(s-2) \slash 2$.
There is a $\Br_Q$-equivariant
%and $\CC^*$-equivariant
isomorphism $\lozenge_s$ of complex manifolds
such that the diagram
\begin{equation}\label{eq:Frob}
\xymatrix{
\MQ \ar[rr]^{\lozenge_s\qquad} \ar[dr]_{P_{\nu}} &&
\QStab_s^{\circ}\DXQ  \ar[dl]^{\mathcal{Z}_s}  \\
&\CC^n&
}
\end{equation}
commutes.
In particular, $\lozenge_s$ satisfies that
\[\begin{cases}
    \lozenge_s^* \, \mathcal{E}=E,\\
    \lozenge_s^* \, \mathrm{d}=\nabla^{\CKZ}_{\nu}
\end{cases}\]
for Euler fields and torsion-free flat connections on each side.
\end{conjecture}

%=========================================================
\subsection{Type A case}
%=========================================================
When $Q$ is of type A, i.e., $\surf=\dd$ is a (graded) disk with $n+1$ marked points on its boundary
as in Section~\ref{sec:A}.
Take a full formal arc system $\TT$ of $\dd$ and we identify $\DXQ=\DX$.

Then putting the results of Corollary~\ref{cor:n+2} and Proposition~\ref{prop:uc=HS} together,
we have the identification
\begin{equation}\label{eq:identification}
    \QStap_s\D_\XX(A_n)\cong\widetilde{\Poly}_{A_n}=\widetilde{\mathfrak{h}_\reg/W}
\end{equation}
via moduli spaces of (framed) $q$-quadratic differentials and (framed) Hurwitz spaces.
Thus, this implies the following.

\begin{theorem}\label{thm:pline}
Assume that $\Re (s) > 2$.
Let $Q$ be the $A_n$-quiver.
Then Conjecture~\ref{conj:main} holds.
\end{theorem}
\begin{remark}
For integer valued $s$, there are following previous works:
\begin{itemize}
\item $s=2$ is essentially shown in \cite{B3};
\item $s=3$ is essentially shown in \cite{BS};
\item $s\in\ZZ_{\ge2}$ is shown in \cite{I} based on the theory of \cite{BS}.
\end{itemize}
\end{remark}

%=========================================================
%=========================================================
\addcontentsline{toc}{part}{Reference}
\let\oldaddcontentsline\addcontentsline% Store \addcontentsline
\renewcommand{\addcontentsline}[3]{}% Make \addcontentsline a no-op
%\let\addcontentsline\oldaddcontentsline% Restore \addcontentsline
%=========================================================

%=========================================================

\begin{thebibliography}{99999}
%=========================================================
\newcommand{\au}[1]{\textrm{#1},}
\newcommand{\ti}[1]{\textrm{#1},}
\newcommand{\jo}[1]{\textit{#1}}
\newcommand{\vo}[1]{\textbf{#1}}
\newcommand{\yr}[1]{(#1)}
\newcommand{\pp}[2]{#1--#2.}
\newcommand{\arxiv}[1]{\href{http://arxiv.org/abs/#1}{arXiv:#1}}


\bibitem[A]{A}
  \au{C.~Amiot}
  \ti{Cluster categories for algebras of global dimension 2 and quivers with potential}
  \jo{Ann. Inst. Fourier} \vo{59} \yr{2009} \pp{2525}{2590}
  (\arxiv{0805.1035})

\bibitem[1]{ABCDGKMSSW}
  \au{P.~Aspinwall \and T.~Bridgeland \and A.~Craw \and M.~Douglas \and
  M.~Gross \and A.~Kapustin \and G.~Moore \and G.~Segal \and
  B.~Szendroi \and P.~Wilson}
  \ti{Dirichlet branes and mirror symmetry}
  \jo{Clay Mathematics Monographs} \vo{4}
  \jo{American Mathematical Society, Providence, RI; Clay Mathematics Institute, Cambridge, MA},
  \yr{2009}
  %\pp{x+681 pp}

\bibitem[AT]{AT}
  \au{P.~Achar and D.~Treumann}
  \ti{Baric structures on triangulated categories and coherent sheaves}
  \jo{Int. Math. Res. Notices} \vo{16} \yr{2011} \pp{3688}{3743}
  (\arxiv{0808.3209})

\bibitem[BM]{BM}
  \au{A.~Bayer and E.~Macr{\`{\i}}}
  \ti{The space of stability conditions on the local projective plane}
  \jo{Duke Math. J.} \vo{160(2)} \yr{2011} \pp{263}{322}
  (\arxiv{0912.0043})

\bibitem[BB]{BvdB}
\au{A.~Bondal, \and M.~van den Bergh}
  \ti{Generators and representability of functors in commutative and
noncommutative geometry}
  \jo{Mosc. Math. J.}
  \vo{3} \yr{2003} \pp{1}{36}
  (\href{http://arxiv.org/abs/math/0204218}{arXiv:math/0204218})

\bibitem[BrSa]{BSai}
\au{E.~Brieskorn \and K.~Saito}
  \ti{Artin-{G}ruppen und {C}oxeter-{G}ruppen}
  \jo{Invent. Math.} \vo{17} \yr{1972} \pp{245}{271}

\bibitem[B1]{B4}
  \au{T.~Bridgeland}
  \ti{Stability conditions on a non-compact {C}alabi-{Y}au threefold}
  \jo{Comm. Math. Phys.} \vo{266(3)} \yr{2006} \pp{715}{733}
  (\href{http://arxiv.org/abs/math/0509048}{arXiv:math/0509048})

\bibitem[B2]{B1}
  \au{T.~Bridgeland}
  \ti{Stability conditions on triangulated categories}
  \jo{Ann. Math.} \vo{166} \yr{2007} \pp{317}{345}
  (\href{http://arxiv.org/abs/math/0212237}{arXiv:math/0212237})

\bibitem[B3]{B2}
  \au{T.~Bridgeland}
  \ti{Stability conditions on {K}$3$ surfaces}
  \jo{Duke Math. J.} \vo{141(2)} \yr{2008} \pp{241}{291}
  (\href{https://arxiv.org/abs/math/0307164}{arXiv:math/0307164})

\bibitem[B4]{B5}
  \au{T.~Bridgeland}
  \ti{Spaces of stability conditions}
  \jo{In Algebraic geometry--Seattle 2005, volume 80 of Proc. Sympos. Pure Math. }
  \pp{1}{21}
  \jo{Amer. Math. Soc., Providence, RI.} \yr{2009}
  (\href{http://arxiv.org/abs/math/0212237}{arXiv:math/0611510})

\bibitem[B5]{B3}
  \au{T.~Bridgeland}
  \ti{Stability conditions and Kleinian singularities}
  \jo{Int. Math. Res. Not.} \vo{21} \yr{2009} \pp{4142}{4157}
  (\href{http://arxiv.org/abs/math/0508257}{arXiv:math/0508257})

\bibitem[BQS]{BQS}
  \au{T.~Bridgeland, Y.~Qiu and T.~Sutherland}
  \ti{Stability conditions and the ${A}_2$ quiver}
  \jo{Adv. Math.} \vo{365} \yr{2020} 107049.
  (\arxiv{1406.2566})

\bibitem[BS]{BS}
  \au{T.~Bridgeland \and I.~Smith}
  \ti{Quadratic differentials as stability conditions}
  \jo{Publ. Math. de l'IH\'{E}S}
  \vo{121} \yr{2015} \pp{155}{278}
  (\arxiv{1302.7030})

\bibitem[BZ]{BZ}
  \au{T.~Br\"{u}stle and J.~Zhang}
  \ti{On the cluster category of a marked surface without punctures}
  \jo{Algebra Number Theory} \vo{5} \yr{2011} \pp{529}{566}

\bibitem[BMRRT]{BMRRT}
  \au{A.~Buan, R.J.~Marsh, M.~Reineke, I.~Reiten, G.~Todorov}
  \ti{Tilting theory and cluster combinatorics}
  \jo{Adv. Math.} \vo{204} \yr{2006} 572--618.
  (\href{http://arxiv.org/abs/math.RT/0402054}{arXiv:math/0402054})

\bibitem[BQZ]{BQZ}
  \au{A. B.~Buan, Y.~Qiu \and Y.~Zhou}
  \ti{Decorated marked surfaces III: The derived category of a decorated marked surface}
  \jo{Int. Math. Res. Notices} \vo{17} \yr{2021} \pp{12967}{12992}
  (\arxiv{1804.00094})

\bibitem[CL]{CL}
  \au{G.~Cerulli-Irelli \and D.~Labardini-Fragoso}
  \ti{Quivers with potentials associated to triangulated surfaces, part III:
  Tagged triangulations and cluster monomials}
  \jo{Compositio Mathematica} \vo{148} \yr{2012} \pp{1833}{1866}
  (\arxiv{1108.1774})

\bibitem[C]{C}
  \au{I.~V.~Cherednik}
  \ti{Generalized Braid Groups and Local $r$-Matrix Systems}
  \jo{Soviet Math. Dokl.}
  \vo{40} \yr{1990} \pp{43}{48}

\bibitem[DWZ]{DWZ}
  \au{H.~Derksen, J.~Weyman \and A.~Zelevinsky}
  \ti{Quivers with potentials and their representations I: Mutations}
  \jo{Selecta Mathematica} \vo{14} \yr{2008} \pp{59}{119}
  (\arxiv{0704.0649})

\bibitem[DHKK]{DHKK}
  \au{G.~Dimitrov, F.~Haiden, L.~Katzarkov and M.~Kontsevich}
  \ti{Dynamical systems and categories}
  The influence of Solomon Lefschetz in geometry and topology,
  \jo{Contemp. Math.} \vo{621} \yr{2014} \pp{133}{170}
  (\arxiv{1307.8418})

\bibitem[D1]{Dub1}
  \au{B.~Dubrovin}
  \ti{Geometry of {$2$}{D} topological field theories}
  \jo{Integrable systems and quantum groups ({M}ontecatini {T}erme} \pp{120}{348}
  \jo{ecture Notes in Math.} \vo{1620}, \jo{Adv. Math. Sci.} \vo{55}
  \jo{Springer, Berlin} \yr{1996}
  (\href{https://arxiv.org/abs/math/9407018}{arxiv:math/9407018})


\bibitem[D2]{Dub2}
  \au{B.~Dubrovin}
  \ti{On almost duality for {F}robenius manifolds}
  \jo{Geometry, topology, and mathematical physics} \pp{75}{132}
  \jo{Amer. Math. Soc. Transl. Ser.} \vo{212}, \jo{Adv. Math. Sci.} \vo{55}
  \jo{Amer. Math. Soc., Providence, RI} \yr{2004}
  (\href{https://arxiv.org/abs/math/0307374}{arxiv:math/0307374})

\bibitem[FM]{FM}
  \au{D.~Fiorenza \and G.~Marchetti}
  \ti{The (he)art of gluing}
  (\arxiv{1806.00883})

%\bibitem[FHKV]{FHKV}
%  \au{B.~Feng, Y-H.~He, K.D.~Kennaway \and C~.Vafa}
%  \ti{Dimer models from mirror symmetry and quivering amoebae}
%  \jo{Adv. Theor. Math. Phys.} \vo{12} \yr{2008} \pp{489}{545}
%  (\href{http://arxiv.org/abs/hep-th/0511287}{arXiv:0511287})

\bibitem[FST]{FST}
  \au{S.~Fomin, M.~Shapiro \and D.~Thurston}
  \ti{Cluster algebras and triangulated surfaces, part I: Cluster complexes}
  \jo{Acta Math.} \vo{201} \yr{2008} \pp{83}{146}
  (\href{http://arxiv.org/abs/math/0608367}{arXiv:0608367})

\bibitem[GMN]{GMN}
  \au{D.~Gaiotto, G.~Moore \and A.~Neitzke}
  \ti{Wall-crossing, Hitchin systems and the WKB approximation}
  \jo{Adv. Math.} \vo{234} \yr{2013} \pp{239}{403}
  (\arxiv{0907.3987})

%\bibitem[GLS]{GLFS}
%  \au{C.~Geiss, D.~Labardini-Fragoso, \and J.Schr\"{o}er}
%  \ti{The representation type of Jacobian algebras}
%  \arxiv{1308.0478}.

\bibitem[Gu]{Guo}
  \au{L.~Guo}
  \ti{Cluster tilting objects in generalized higher cluster categories}
  (\arxiv{1005.3564})

\bibitem[H]{H}
  \au{F. Haiden}
  Personal communication (emails).

\bibitem[HKK]{HKK}
  \au{F.~Haiden, L.~Katzarkov and M.~Kontsevich}
  \ti{Stability in Fukaya categories of surfaces}
  \jo{Publ. Math. de l'IH\'{E}S}
  \vo{126} \yr{2017} \pp{247}{318}
  (\arxiv{1409.8611})

\bibitem[HK]{HK}
  \au{Z.~Hua \and B.~Keller}
  \ti{Cluster categories and rational curves}
  (\href{https://arxiv.org/abs/1810.00749}{arXiv:1810.00749})

\bibitem[I]{I}
  \au{A.~Ikeda}
  \ti{Stability conditions on {$\text{CY}_N$} categories associated to {$A_n$}-quivers and period maps}
  \jo{Math. Ann.} \vo{367} \yr{2017} \pp{1}{49}
  (\arxiv{1405.5492})

\bibitem[IQ]{IQ1}
  \au{A.~Ikeda \and Y.~Qiu}
  \ti{$q$-Stability conditions on Calabi-Yau-$\XX$ categories and twisted periods}
  \arxiv{1807.00469}.

\bibitem[IQZ]{IQZ}
  \au{A.~Ikeda, Y.~Qiu \and Y. Zhou}
  \ti{Graded decorated marked surfaces: Calabi-Yau-$\XX$ categories of gentle algebras}
  \arxiv{2006.00009}.

\bibitem[IUU]{IUU}
  \au{A.~Ishii, K.~Ueda, \and H.~Uehara}
  \ti{Stability conditions on $A_n$-singularities}
  \jo{J. Diff. Geom.} \vo{84} \yr{2010} \pp{87}{126}
  (\href{https://arxiv.org/abs/math/0609551}{arxiv:math/0609551})

%\bibitem[IN]{IN}
%  \au{K.~Iwaki \and T.~Nakanishi}
%  \ti{Exact WKB analysis and cluster algebras}
%  \jo{ J. Phys. A: Math. Theor.} \vo{47} \yr{2014} 474009.
%  (\arxiv{1401.7094})

\bibitem[IY]{IY}
  O.~Iyama, \and D.~Yang,
  \ti{Quotients of triangulated categories and Equivalences of Buchweitz, Orlov and Amiot--Guo--Keller}
  (\arxiv{1702.04475})

\bibitem[LP]{LP}
  \au{Y.~Lekili \and A.~Polishchuk}
  \ti{Auslander orders over nodal stacky curves and partially wrapped Fukaya categories}
  (\arxiv{1705.06023})

\bibitem[LP2]{LP2}
  \au{Y.~Lekili \and A.~Polishchuk}
  \ti{Derived equivalences of gentle algebras via Fukaya categories}
  (\arxiv{1801.06370})

\bibitem[KaY]{KaY}
  \au{M.~Kalck \and D.~Yang}
  \ti{Derived categories of graded gentle one-cycle algebras}
  \jo{J. Pure App. Algebra} \vo{222} \yr{2018} \pp{3005}{3035}
  (\arxiv{1611.01903})

\bibitem[K1]{K8}
  \au{B.~Keller}
  \ti{Deformed Calabi-Yau completions}
  \jo{J. Reine Angew. Math.} \vo{654} \yr{2011} \pp{125}{180}
  (\arxiv{0908.3499})

\bibitem[K2]{K2}
   \au{B.~Keller}
   \ti{On triangulated orbit categories}
   \jo{Doc. Math.} \vo{10} \yr{2005} \pp{551}{581}

   (\href{https://arxiv.org/abs/math/0503240}{arXiv:math/0503240})

\bibitem[K3]{K3}
  \au{B.~Keller}
  \ti{On cluster theory and quantum dilogarithm identities}
  \jo{EMS Series of Congress Reports}  \yr{2011} \pp{85}{116}
  (\arxiv{1102.4148})

\bibitem[K4]{K18}
  \au{B.~Keller}
  \ti{Erratum to `Deformed Calabi-Yau completions'}
  (\arxiv{1109.2924})

\bibitem[K5]{Ke2}
  B.~Keller,
  \ti{On differential graded categories}
  (\href{https://arxiv.org/abs/math/0601185}{arxiv:0601185})

\bibitem[K6]{Keller}
  B.~Keller,
  Triangulated Calabi-Yau categories, Trends in Representation Theory of Algebras (Zurich)
  (A. Skowronski, ed.),
  \jo{European Mathematical Society} \yr{2008} \pp{467}{489}

\bibitem[KQ1]{KQ1}
  \au{A.~King \and Y.~Qiu}
  \ti{Exchange graphs and Ext quivers}
  \jo{Adv. Math.} \vo{285} \yr{2015} \pp{1106}{1154}
  (\arxiv{1109.2924})

\bibitem[KQ2]{KQ2}
  \au{A.~King \and Y.~Qiu}
  \ti{Cluster exchange groupoids and framed quadratic differentials}
  \jo{Invent. Math.} \vo{220} \yr{2019} \pp{479}{523}
  (\arxiv{1805.00030})

\bibitem[KhS]{KhS}
  \au{M.~Khovanov \and P.~Seidel}
  \ti{Quivers, floer cohomology and braid group actions}
  \jo{J. Amer. Math. Soc.} \vo{15} \yr{2002} \pp{203}{271}
  (\href{https://arxiv.org/abs/math/0006056}{arxiv:0006056})

\bibitem[Ko]{K}
  M.~Kontsevich,
  Homological algebra of mirror symmetry,
  \emph{Proceedings of the International Congress of Mathematicians}
  (Z\"{u}rich, 1994), Birkh\"{a}user, 1995, pp. 120--139.

\bibitem[KoS]{KoSo}
  \au{M.~Kontsevich \and Y.Soibelman}
  \ti{Stability structures, motivic Donaldson-Thomas invariants and cluster transformations}
  \arxiv{0811.2435}.

\bibitem[M]{M}
  \au{Y.~Manin}
  \ti{Frobenius manifolds, quantum cohomology, and moduli spaces}
  \jo{Amer. Math. Soc. Colloq. Publ.} \vo{47}
  \jo{Amer. Math. Soc., Providence, RI} \yr{1999}

\bibitem[Op]{Op}
  \au{S.~Oppermann}
  \ti{Quivers for silting mutation}
  \jo{Adv. Math. } \vo{307} \yr{2017} \pp{684}{714}
  (\arxiv{1504.02617})

\bibitem[Or]{Or}
  \au{D.~Orlov}
  \ti{Remarks on generators and dimensions of triangulated categories}
  \jo{Mosc. Math. J.} \vo{9} \yr{2009} \pp{153}{159}
    (\arxiv{0804.1163})

\bibitem[Q1]{Q2}
  \au{Y.~Qiu}
  \ti{Stability conditions and quantum dilogarithm identities for Dynkin quivers}
  \jo{Adv. Math.} \vo{269} \yr{2015} \pp{220}{264}
  (\arxiv{1111.1010})

\bibitem[Q2]{QQ}
  \au{Y.~Qiu}
  \ti{Decorated marked surfaces: Spherical twists versus braid twists}
  \jo{Math. Ann.} \vo{365} \yr{2016} \pp{595}{633}
  (\arxiv{1407.0806})

\bibitem[Q3]{QQ2}
  \au{Y.~Qiu}
  \ti{Decorated marked surfaces (part B): topological realizations}
  \jo{Math. Z.} \vo{288} \yr{2018} \pp{39}{53}

\bibitem[Q4]{Q5}
  \au{Y.~Qiu}
  \ti{Topological structure of spaces of stability conditions and topological Fukaya type categories}
  Proceedings of the \jo{International Consortium of Chinese Mathematicians 2017} \pp{521}{538}
  Int. Press, Boston, MA, \yr{2020}.
  (\arxiv{1806.00010})

\bibitem[Q5]{Q3}
  \au{Y.~Qiu}
  \ti{Global dimension function, Gepner equations and $q$-stability conditions}
  \jo{Math. Zeit.} to appear.
  (\arxiv{1807.00010})

\bibitem[Q6]{Q7}
  \au{Y.~Qiu}
  \ti{Contractible flow of stability conditions via global dimension function}

  \arxiv{2008.00282}.

\bibitem[QW]{QW}
  \au{Y.~Qiu \and J.~Woolf}
  \ti{Contractible stability spaces and faithful braid group actions}
  \jo{Geom. Topol.} \vo{22} \yr{2018} \pp{3701}{3760}
  (\arxiv{1407.5986})

\bibitem[QZ1]{QZ}
  \au{Y.~Qiu \and Y.~Zhou}
  \ti{Cluster categories for marked surfaces: punctured case}
  \jo{Compos. Math.} \vo{153} \yr{2017} \pp{1779}{1819}
  (\arxiv{1311.0010})

\bibitem[QZ2]{QZ2}
  \au{Y.~Qiu \and Y.~Zhou}
  \ti{Decorated marked surfaces II: Intersection numbers and dimensions of Homs}
  \jo{Trans. Amer. Math. Soc.} \vo{372} \yr{2019} \pp{635}{660}
  (\arxiv{1411.4003})

%\bibitem[QZ3]{QZ3}
%  \au{Y.~Qiu \and Y.~Zhou}
%  \ti{Finite presentations for spherical/braid twist groups from decorated marked surfaces}
%  \arxiv{1703.10053}.

\bibitem[Sa1]{Sa1}
\au{K.~Saito}
\ti{Period Mapping Associated to a Primitive Form}
\jo{Publ. RIMS.} \vo{19} \yr{1983} \pp{1231}{1264}

\bibitem[Sa2]{Sa2}
\au{K.~Saito}
\ti{Uniformization of the orbifold of a finite reflection group}
\jo{Frobenius Manifolds, Aspects Math.} \vo{36} \yr{2004} \pp{265}{320}

\bibitem[SaTa]{SaTa}
  \au{K.~Saito and A.~Takahashi}
  \ti{From primitive forms to Frobenius manifolds}
  \jo{From Hodge theory to integrability and TQFT $tt^*$-geometry} \pp{31}{48}
  \jo{Proc. Sympos. Pure Math.}  \vo{78}
  \jo{Amer. Math. Soc., Providence, RI} \yr{2008}.
  (\href{https://arxiv.org/abs/math/0307374}{arxiv:math/0307374})

\bibitem[Se]{Se}
  \au{P.~Seidel}
  \ti{Picard-Lefschetz theory and dilating $\mathbb{C}^*$-actions}
  \jo{J. Topology} \vo{8} \yr{2015} \pp{1167}{1201}
  (\arxiv{1403.7571})

\bibitem[ST]{ST}
  \au{P.~Seidel \and R.~Thomas}
  \ti{Braid group actions on derived categories of coherent sheaves}
  \jo{Duke Math. J.} \vo{108} \yr{2001} \pp{37}{108}
  (\href{https://arxiv.org/abs/math/0001043}{arxiv:0001043})

\bibitem[TL]{TL}
  \au{V.~Toledano Laredo}
  \ti{Flat connections and quantum groups}
  \jo{Acta Appl. Math} \vo{73} \yr{2002} \pp{155}{173}
  (\arxiv{0205185})

\bibitem[Ta]{T}
  \au{A.~Takeda}
  \ti{Relative stability conditions on Fukaya categories of surfaces}
  \arxiv{1811.10592}.

\bibitem[To]{To}
  \au{Y.~Toda}
  \ti{Gepner type stability conditions on graded matrix factorizations}
  \jo{Algebr. Geom. } \vo{1} \yr{2014} \pp{613}{665}
  (\arxiv{1302.6293})

\bibitem[S]{S}
  \au{I.~Smith}
  \ti{Quiver algebras and Fukaya categories}
  \jo{Geom. Topol.} \vo{19} \yr{2015} \pp{2557}{2617}

  (\arxiv{1309.0452})

\bibitem[Y]{Y}
  \au{W-K.~Yeung}
  \ti{Relative Calabi-Yau completions}
  \arxiv{1612.06352}.


\end{thebibliography}
\end{document}